\documentclass[english,11pt]{smfart}

\usepackage{etex}
\usepackage{tcolorbox}

\usepackage{amsbsy}
\usepackage{amsmath,amsfonts,amssymb,amsthm,mathrsfs,mathtools}

\usepackage{bm}

\usepackage[a4paper,vmargin={3cm,3cm},hmargin={3.5cm,3.5cm}]{geometry}
\usepackage[font=sf, labelfont={sf,bf}, margin=1cm]{caption}
\usepackage{graphicx}
\usepackage{epsfig}
\usepackage{latexsym}
\usepackage{xcolor}
\usepackage{ae,aecompl}
\usepackage{soul,framed}
\usepackage{comment}

\usepackage{xcolor}
\usepackage[pdfpagemode=UseNone,bookmarksopen=false,colorlinks=true,urlcolor=blue,citecolor=blue,citebordercolor=blue,linkcolor=blue]{hyperref}
\usepackage{smfhyperref}
\usepackage[capitalize]{cleveref}

\usepackage{pstricks}
\usepackage{enumerate}
\usepackage{tikz,animate,media9}						
\usepackage{todonotes}
\usepackage{pifont}
\usepackage{bm,marvosym}
\usepackage{algorithm}
\usepackage{algorithmic}

\usepackage{bbm}

\usepackage{tcolorbox}

\definecolor{aleacolor}{rgb}{0.16,0.59,0.78}

\hypersetup{
	breaklinks,
	colorlinks=true,
	linkcolor=aleacolor,
	urlcolor=aleacolor,
	citecolor=aleacolor}

\newcount\colveccount
\newcommand*\colvec[1]{
	\global\colveccount#1
	\begin{pmatrix}
		\colvecnext
	}
	\def\colvecnext#1{
		#1
		\global\advance\colveccount-1
		\ifnum\colveccount>0
		\\
		\expandafter\colvecnext
		\else
	\end{pmatrix}
	\fi
}

\newcommand{\ndN}{\mathbb{N}}
\newcommand{\ndZ}{\mathbb{Z}}

\newcommand{\ndR}{\mathbb{R}}

\renewcommand{\Pr}[1]{\mathbb{P}(#1)}

\newcommand{\Prb}[1]{\mathbb{P}\left(#1\right)}

\newcommand{\Ex}[1]{\mathbb{E}[#1]}

\newcommand{\Exb}[1]{\mathbb{E}\left[#1\right]}

\newcommand{\Va}[1]{\mathbb{V}[#1]}

\newcommand{\Vab}[1]{\mathbb{V}\left[#1\right]}

\newcommand{\convd}{\,{\buildrel \mathrm{d} \over \longrightarrow}\,}

\newcommand{\convv}{\,{\buildrel \mathrm{v} \over \longrightarrow}\,}

\newcommand{\convp}{\,{\buildrel \mathrm{p} \over \longrightarrow}\,}

\newcommand{\eqdist}{\,{\buildrel \mathrm{d} \over =}\,}

\newcommand{\cC}{\mathcal{C}}
\newcommand{\cD}{\mathcal{D}}

\newcommand{\cG}{\mathcal{G}}
\newcommand{\cH}{\mathcal{H}}
\newcommand{\cI}{\mathcal{I}}
\newcommand{\cJ}{\mathcal{J}}

\newcommand{\cL}{\mathcal{L}}

\newcommand{\cN}{\mathcal{N}}

\newcommand{\cP}{\mathcal{P}}
\newcommand{\cQ}{\mathcal{Q}}
\newcommand{\cR}{\mathcal{R}}
\newcommand{\cS}{\mathcal{S}}

\newcommand{\cW}{\mathcal{W}}
\newcommand{\cX}{\mathcal{X}}
\newcommand{\cY}{\mathcal{Y}}
\newcommand{\cZ}{\mathcal{Z}}

\newcommand{\Sym}{\mathrm{Sym}}

\newcommand{\mA}{\mathsf{A}}

\newcommand{\mD}{\mathsf{D}}

\newcommand{\mS}{\mathsf{S}}
\newcommand{\mT}{\mathsf{T}}
\newcommand{\mU}{\mathsf{U}}

\newcommand{\aut}{\mathrm{aut}}

\newcommand{\Aut}{\mathrm{Aut}}
\newcommand{\Autp}{\mathrm{Aut}^+}
\newcommand{\Autm}{\mathrm{Aut}^-}

\newcommand{\mob}{\text{m{\"o}b}}
\newcommand{\cyl}{\text{cyl}}

\newcommand{\id}{\mathrm{id}}

\newtheorem{theorem}{Theorem}[section]

\newtheorem{corollary}[theorem]{Corollary}
\newtheorem{proposition}[theorem]{Proposition}
\newtheorem{lemma}[theorem]{Lemma}

\newtheorem{definition}[theorem]{Definition}
\newtheorem{remark}[theorem]{Remark}

\numberwithin{equation}{section}


\title{\textbf{Asymptotic enumeration of unlabelled cubic planar graphs}}
\date{}

\author{Benedikt Stufler}

\address[Benedikt Stufler]{Vienna University of Technology}
\email{benedikt.stufler at tuwien.ac.at}

\begin{document}

\vspace {-0.5cm}

\begin{abstract}
	We determine the precise asymptotic number of unlabelled cubic planar graphs with $n$ vertices. Our approach blends generating series methods with computational bounds and probabilistic local large deviation theorems.
\end{abstract}


\maketitle

\section{Introduction}

In their survey on ten steps to counting planar graphs~\cite{zbMATH04055646}, Liskovets and Walsh wrote:  ``One well-known long-standing unsolved graph-enumeration problem is to count (nonisomorphic) planar graphs. The aim of this brief survey is to draw the reader's attention to the considerable progress which has been achieved to that end, and which suggests that this problem may soon be completely solved.''

Forty years later, the problem of counting nonisomorphic planar graphs remains open. A major breakthrough was achieved by Gim\'enez and Noy~\cite{MR2476775} who solved the labelled case, after Bender, Gao and Wormald~\cite{MR1946145} enumerated $2$-connected labelled planar graphs. Determining the asymptotic number laid the foundation for a detailed analysis of shape parameters of labelled planar graphs~\cite{noysurvey}. The related class of labelled cubic planar graphs was enumerated by Bodirsky, Kang, L\"offler and McDiarmid~\cite{zbMATH05122852} and Noy, Requil\'e and Ru\'e~\cite{zbMATH07213288}.  Structural recurrence formulas for the exact number of unlabelled cubic planar graphs were determined by Bodirsky, Gr\"opl and Kang~\cite{zbMATH05255661}, and symmetries of simple triangulations were described by Brown~\cite{MR168485} and Kang and Spr\"ussel~\cite{zbMATH06841878}. However, despite these successes, no asymptotic enumerative formula could be established for unlabelled cubic planar graphs either. 

There are several major obstacles in the asymptotic enumeration of planar graphs up to graph isomorphism. For example, the bijective methods for enumerating planar maps initiated by Tutte's pioneering census papers~\cite{zbMATH03187384, zbMATH03169204} break down in the case of graphs. No adequate bijections for encoding planar graphs in the same way as planar maps are known. While it is known that specific classes of planar maps are  asymmetric, see the paper by  Richmond and Wormald~\cite{zbMATH00767676} for a general criterion, this is not the case for  planar graphs unless they are $3$-connected. The reason for this is that the abundance of pendant copies of symmetric subgraphs in a single large planar graph with less connectivity constraints yields  exponentially many  non-trivial automorphisms and hence shifts the growth constant.

Moreover, even  the non-trivial automorphisms at the $3$-connected level  cannot be ignored: the symmetries of a large graph are decomposed into many types of symmetries of subgraphs and substructures, the vast majority of which have a small number of vertices. Thus, although   a large $3$-connected component is unlikely to exhibit symmetries due to the relation with planar maps via Whitney's theorem, there is typically a linear number of small $3$-connected components (such as the complete graph on four vertices) with non-trivial symmetries. In this way non-trivial symmetries even persist at the $3$-connected component level and influence the asymptotic number of unlabelled planar graphs. This creates a major challenge. Fully characterising the  structure of the non-trivial symmetries at any level of connectivity is a highly complex  problem. 


Furthermore, the decomposition of connectivity involves a large number of rooted versions of graphs and networks. De-rooting in the presence of symmetries, that is, relating the number of rooted objects to the number of unrooted objects, adds further levels of complexity. Classical methods such as the dissymmetry theorem multiply the number of types of rooted versions that are necessary. In order to work, they also require precise second-order asymptotics of growth.

The present work overcomes these challenges and is the first to solve the asymptotic counting problem for a critical class of unlabelled planar graphs:

\begin{theorem}
	\label{te:main}
	The number $g_n$ of unlabelled cubic planar graphs with an even number $n$ of vertices satisfies
	\[
		g_n \sim \kappa n^{-7/2} \rho^{-n}
	\]
	for positive constants $\kappa, \rho>0$. 
\end{theorem}
The constants admit the numerical approximations $\rho \approx 0.31560$ and $\kappa \approx 0.07$. For the number $c_n$ of unlabelled connected cubic planar graphs with an even number $n$ of vertices we obtain
\[
c_n \sim \gamma n^{-7/2} \rho^{-n}
\]
with $\rho$ as before and $\gamma \approx 0.070$ a positive constant satisfying $\gamma < \kappa$.

In broad outline, we proceed in the same direction as the ten-step program. We start with the $3$-connected case. Then we enumerate unlabelled cubic planar networks with an oriented root edge. Next, we enumerate simple cubic planar networks. Finally, we de-root to enumerate unlabelled connected  cubic planar graphs and make the final step to the unconstrained case.

However, our approach differs  in several important points. 
We evade the need for a full structural decomposition and asymptotic enumeration for all kinds of symmetries of $3$-connected cubic planar graphs. A structural analysis of automorphisms of triangulations by Brown~\cite{MR168485} is only used to improve the accuracy of numerical approximations carried out after proving Theorem~\ref{te:main}. The numeric approximations are achieved using a mix of bounds stemming from the system of equations of generating series we establish, as well as truncations of power series obtained using the software \texttt{plantri} by Brinkmann, McKay and Van den Camp~\cite{zbMATH05382101,zbMATH02244630, zbMATH02245027} and the \texttt{nauty} software suite by McKay and Piperno~\cite{zbMATH06264238}.

The underlying reason why we manage to evade the need for a full structural decomposition in the $3$-connected case is as follows. Non-trivial symmetries at the $3$-connected level influence the structure and number of graphs with a lower level of connectivity through a large number of small symmetric components. In the proof we apply local large deviation bounds in a heavy-tailed regime established by Denisov, Dieker and Shneer~\cite{MR2440928} that describe a single big step that decouples the small steps and their dependencies. At the random walk level, this behaviour of the small steps was described in detail by Loulakis and Armendar\'iz~\cite{MR2775110}. Through these deviation bounds we decouple the non-maximal $3$-connected components with complicated symmetries from the maximal one, which is likely to be asymmetric.

Furthermore, we evade combinatorial de-rooting methods and use concentration inequalities instead. This reduces complexity by allowing us  to de-root in a single step, without the need to introduce various types of rooted objects.

The framework we develop works directly with asymptotics of coefficients of generating series, computational methods and probabilistic local large deviation theorems. Our enumerative results constitute a solid foundation for the analysis of limiting laws and shapes for random unlabelled cubic planar graphs. We will pursue this in follow-up work. Our approach through local large deviation methods is reusable for other critical classes of planar graphs as it evades the need for a full structural decomposition of the $3$-connected case. We will resolve the long-standing open problem of enumerating unconstrained unlabelled planar graphs in subsequent work.

\subsection*{Organisation}

The paper is organised as follows. Section~\ref{sec:isopmg} discusses isomorphisms of planar maps and graphs with a focus on the $3$-connected case. In Section~\ref{sec:nedewa} we establish a system of equations for the Walsh index series of cubic planar networks. Section~\ref{sec:ucpg} derives rough numeric bounds and verifies that unlabelled cubic planar graphs form a critical graph class. In Section~\ref{sec:countingtl} we conduct a probabilistic analysis of generating series to determine the asymptotic number of unlabelled cubic planar networks. Section~\ref{sec:sicup} enumerates non-isomorphic simple cubic networks, connected cubic planar graphs and unconstrained cubic planar graphs. In Section~\ref{sec:numerics} we improve numeric approximations of the constants appearing in the asymptotic enumerative formulas.

\subsection*{Notation}
We  let $\ndN = \{1, 2, \ldots\}$ and $\ndN_0 = \{0, 1, 2, \ldots\}$ denote the sets of positive and nonnegative integers. For each $n \in \ndN_0$ we set $[n] = \{1, \ldots, n\}$, so that~$[0]=\emptyset$. All considered random variables are defined on a common probability space whose measure is denoted by $\mathbb{P}$. Unspecified limits are taken as $n$ tends to infinity.
We use $\convd$ and $\convp$ to denote convergence in distribution and probability. We let $\convv$ denote convergence in the vague topology. Equality in distribution is denoted by $\eqdist$. 
We let $O_{\mathrm{p}}(1)$ denote an unspecified random variable $X_n$ of a stochastically bounded sequence $(X_n)_{n \ge 1}$. 
 The total variation distance is denoted by $d_{\mathrm{TV}}$. 
For a power series $f(x) \in \ndR[[x]]$ we let $[x^n]f(x)$ denote the $n$th coefficient of $f$. We use similar notation for multivariate series. We say a power series $g(x)$ dominates $f(x)$ coefficient-wise if $[x^n]f(x) \le [x^n]g(x)$ for all~$n \in \ndN_0$. We denote this by $f(x) \le_{\mathrm{coeff}} g(x)$. We use the notation $O(1)$ to denote an unspecified sequence $a_n$ with $\limsup_{n \to \infty} |a_n| < \infty$. Depending on the context, $a_n$ may depend on further parameters $t_1, t_2, \ldots$. If we state that the $O$ term is uniform for some range of these parameters, it means that $\limsup_{n \to \infty} \sup_{t_1, t_2, \ldots} |a_n| < \infty$, with the indices $t_1, t_2, \ldots$ constrained to that range. Given a sequence $b_n$ we use the short notation $O(b_n)$ for $b_n O(1)$. We use similar conventions for the $o$-notation and $\Theta$-notation defined as follows. The term $o(1)$ denotes an unspecified sequence that tends to zero as $n$ tends to infinity. The term $\Theta(1)$ denotes an unspecified sequence $\bar{a}_n$ such that there exist $c,C>0$ and $N>0$ such that for all $n>N$ we have $c < \bar{a}_n < C$.

\section{Isomorphisms of planar maps and graphs}
\label{sec:isopmg}

\subsection{Graphs and planar maps}

A multigraph is called \emph{simple} (or just a graph) if it has no multi-edges or loops. That is, between any unordered pair of distinct vertices there is at most one edge, and no edge starts and ends at the same vertex. A simple graph is $k$-connected for $k \ge 2$ if it is connected, has at least $k+1$ vertices and removing any subset of at most $k-1$ vertices does not disconnect the graph. A multigraph is called $2$-connected if it has no loops and the underlying simple graph is $2$-connected. For $k \ge 3$ a multigraph is called $k$-connected if it is simple and the underlying simple graph is $k$-connected. A multigraph is called cubic if each vertex has degree $3$.

An \emph{isomorphism} between multigraphs $G$ and $G'$ is a bijection from the vertex set of $G$ to the vertex set of $G'$ such that the number of edges between any two (not necessarily distinct) vertices of $G$ is equal to the number of edges between their images.  An \emph{automorphism} of a multigraph $G$ is an isomorphism from $G$ to itself.

If $G$ is equipped with a root vertex (or several root vertices), a root edge, or an oriented root edge, we require that graph isomorphisms between $G$ and other rooted multigraphs preserve the root structures. That is, root vertices get mapped to root vertices (if an ordering of the root vertices is specified then it must be respected too), root edges get mapped to root edges, and if the root edge is oriented then its start must be mapped to the start of the target root edge and its end needs to be mapped to the end of the target root edge.

A \emph{planar map} is a drawing of a connected multigraph on the $2$-dimensional sphere such that edges are represented by arcs that may only intersect at their endpoints~\cite{zbMATH03187384}.  Planar maps are viewed up to orientation-preserving homeomorphism of the $2$-sphere. 

The \emph{faces} of a planar map correspond to the connected components created when removing it from the $2$-sphere. We require faces to be homeomorphic to open disks.  Each face has a \emph{boundary}, given by a counter-clockwise cyclically ordered list of sides of edges. Thus each edge corresponds to a pair of \emph{half-edges} which  are oriented in opposing directions.   The \emph{degree} of a face is the number of half-edges on its boundary. A planar map in which all faces have degree $3$ is called a  \emph{triangulation}.


A \emph{rooted planar map} is a planar map with a distinguished \emph{root edge} that is oriented.  Rooting maps at an oriented root edge is equivalent to rooting maps at a corner. We use the convention that the face to the right of the root edge is called the \emph{outer face}.

We call a planar map simple or $k$-connected (for any $k \ge 2$), if its underlying multigraph is simple or $k$-connected. Thus, a triangulation with at least $4$ vertices is simple if and only if it is $3$-connected.  There are differing conventions in the literature. Tutte~\cite{zbMATH03169204} additionally requires a simple triangulation to have no separating $3$-cycles. We do not make this requirement here.

A \emph{combinatorial isomorphism} between planar maps $M$ and $M'$ consists of bijections between their sets of vertices, edges, and faces that preserve adjacency of vertices with edges, and edges with faces. That is, a vertex is incident to some edge if and only if its image is incident to the image of the edge, and we additionally require that an edge lies on the boundary of some face if and only if the image of that edge lies on the boundary of the image of that face. If $M$ is $3$-connected, then Whitney's theorem ensures that such an isomorphism is induced by a homeomorphism of the sphere. See Section~\ref{sec:whitney} below. A \emph{combinatorial automorphism} of $M$ is a combinatorial isomorphism between $M$ and itself. 

Any planar map $M$ has a \emph{dual map}. It may be constructed by placing a ``blue'' vertex inside each face of $M$, adding for each edge $e$ of $M$ a blue edge across $e$ between the blue vertices corresponding to the two faces adjacent to $e$, and then finally deleting all the original edges and vertices of $M$.   The dual map construction is an involution, meaning that the dual map of the dual map of $M$ is equal to $M$ (since we view maps up to orientation-preserving homeomorphism).

If $M$ is a rooted map and $e$ is the root edge of $M$, we orient the corresponding blue edge in the dual map from the face to the right of $e$ to the face to the left of $e$. This way, the dual map of a rooted planar map has an oriented root edge as well. Taking the dual twice of a rooted map yields the same map, but with the orientation of the root edge reversed.

\subsection{Whitney's theorem}
\label{sec:whitney}

If $M$ is $3$-connected, then any edge lies on the boundary of exactly two faces and the  face boundaries are precisely the non-separating induced cycles of the underlying graph. No two distinct faces share the same set of vertices or edges in this case.

Consequently, for two $3$-connected planar maps any graph isomorphism of the underlying $3$-connected planar graphs extends uniquely to a combinatorial isomorphism of the maps.

Whitney's theorem~\cite{zbMATH03008955} states that given an embedding of a $3$-connected planar graph into the sphere, any other embedding may be obtained by applying a homeomorphism of the sphere.  In particular, a combinatorial isomorphism between $3$-connected planar maps is induced by a homeomorphism of the sphere.

This homeomorphism may be orientation-preserving or orientation-reversing. Hence any unlabelled $3$-connected planar graph is either \emph{chiral} and corresponds to two planar maps that are mirror images of each other, or it is \emph{achiral} and corresponds to only one. See Figure~\ref{fi:chiral} for an illustration.

\begin{figure}[H]
	\centering
	\begin{minipage}{\textwidth}
		\centering
		\includegraphics[scale=1.0]{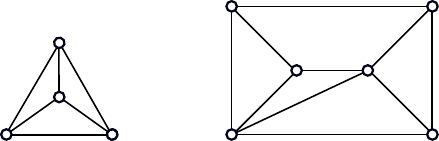}
		\caption{An achiral graph on the left and a chiral graph on the right.}
		\label{fi:chiral}
	\end{minipage}
\end{figure}

Likewise, an unlabelled $3$-connected planar graph with an oriented root edge may correspond to one or two rooted planar maps. We call it accordingly \emph{rooted achiral} or \emph{rooted chiral}.
 The number of embeddings need not be identical to the number of embeddings of the underlying unrooted graph. See Figure~\ref{fi:rchiral} for examples.

\begin{figure}[H]
	\centering
	\begin{minipage}{\textwidth}
		\centering
			\includegraphics[scale=1.0]{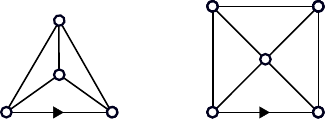}
		\caption{On the left: a rooted graph that is  rooted achiral and also achiral as an unrooted graph. On the right: a rooted graph that is rooted chiral but achiral as an  unrooted graph.}
		\label{fi:rchiral}
	\end{minipage}
\end{figure}

For vertex-labelled planar maps the neighbours of any vertex are cyclically ordered. Orientation-preserving homeomorphisms of the sphere preserve this ordering, whereas orientation-reversing homeomorphisms reverse it. Hence any vertex-labelled $3$-connected planar graph corresponds to precisely two vertex-labelled planar maps, regardless of whether the underlying unlabelled graph corresponds to one or two unlabelled planar maps.  See Figure~\ref{fi:lchiral}.

\begin{figure}[H]
	\centering
	\begin{minipage}{\textwidth}
		\centering
		\includegraphics[scale=0.6]{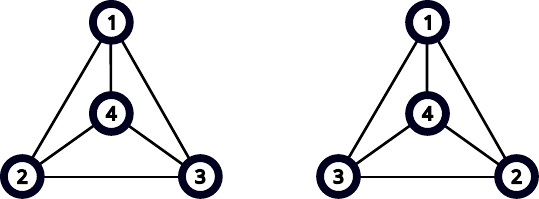}
		\caption{Two vertex-labelled planar maps that are not related through an orientation-preserving homeomorphism, although the corresponding unlabelled planar map is achiral.}
		\label{fi:lchiral}
	\end{minipage}
\end{figure}

\subsection{Automorphisms of $3$-connected planar graphs}

It was shown by~\cite{zbMATH03234261} that any $3$-connected planar graph with $m$ edges has at most $4m$ automorphisms. That is, the number of automorphisms admits a linear bound in terms of the number of edges or equivalently vertices, by Euler's polyhedron formula.

The linear order of growth is asymptotically sharp. For example, the automorphism group of the wheel graph on $n+1$ vertices for $n \ge 4$ is the dihedral group with $2n$ elements. The graph consists of a cycle of length $n$ and an additional vertex that is adjacent to each vertex of the cycle.

However, automorphisms preserving a specified root edge have a more rigid structure:

\begin{proposition}
	\label{pro:aut1}
	Consider a $3$-connected  planar graph $G$ with an oriented root edge $e$. Then $G$ has at most two graph automorphisms $\sigma$ that preserve $e$ and its orientation (i.e. such that both ends of $e$ are fixed points). If $\sigma$ is non-trivial then it is orientation-reversing (with respect to a fixed embedding of $G$) and an involution.
\end{proposition}
\begin{proof}
	Let $\sigma$ denote such an automorphism. Let $M$ denote a planar map corresponding to $G$. $\sigma$ extends to a combinatorial automorphism of $M$ that is induced by some homeomorphism of the sphere.
	
	If the homeomorphism is orientation-preserving, then $\sigma$ maps the face $f$ to the left of $e$ to itself. Since $\sigma$ fixes the endpoints of $e$, this means any vertex on the boundary of $f$ is also a fixed point of $\sigma$. Likewise, any face $f'$ that is adjacent to $f$  gets fixed by $\sigma$, and all the vertices on the boundary of $f'$ are fixed points too. Consequently, $\sigma=\id$.
	
	Suppose that the homeomorphism is orientation-reversing.  For any automorphism $\nu$ of $G$ that preserves the oriented root edge and whose action on $M$ is induced by an orientation-reversing homeomorphism of the sphere, it follows that the action of $\sigma \nu$ on $M$ is induced by an orientation-preserving homeomorphism. As $\sigma \nu$ fixes the oriented root edge, it follows by the argument in the preceding paragraph that $\sigma \nu = \id$. In particular, $\sigma^2 = \id$ and $\nu= \sigma^{-1} = \sigma$. This completes the proof.
\end{proof}

\begin{proposition}
	\label{pro:aut2}
	Let $G$ denote 	 a $3$-connected  planar graph  with an oriented root edge $e$. Then $G$ (without marking the root edge) has at most two automorphisms that swap the ends of $e$. Any such automorphism is an involution.
\end{proposition}
\begin{proof}
	Suppose that there exists such an automorphism $\sigma$. We fix a planar map  corresponding to $G$.  Let $\nu$ denote an arbitrary automorphism of $G$ that also swaps the ends of $e$. Then $\sigma \nu$ preserves $e$ and its orientation. 
	
	By Proposition~\ref{pro:aut1} there are at most two automorphisms that preserve $e$ and its orientation. Let us denote them by $\id$ and $\tau$ (with possibly $\tau=\id$). Consequently, $\nu \in \{\sigma^{-1}, \sigma^{-1} \tau\}$. That is, there are at most two  possible choices for $\nu$. 
	
	Furthermore, $\nu^2$ is induced by an orientation-preserving automorphism of $G$ (regardless of whether $\nu$ is), so by Proposition~\ref{pro:aut1} it follows that $\nu^2 = \id$. That is, $\nu$ is an involution.
\end{proof}

The bounds in Propositions~\ref{pro:aut1} and~\ref{pro:aut2} are sharp. The complete graph on four vertices with a distinguished oriented root edge $e$ admits two automorphisms that preserve $e$ and its orientation, and two automorphisms that swap the endpoints of~$e$.

Any unlabelled $3$-connected planar graph $U$ with $n$ vertices that is equipped with an oriented root edge corresponds to precisely $n!$ pairs $(G, \sigma)$ of a $3$-connected planar graph $G$ (isomorphic to $U$) with vertex set $[n]$ that has a specified oriented root edge, and an automorphism $\sigma$ of $G$ that fixes the root edge and its orientation (i.e. both ends of the root edge are fixed points). This is due to the symmetric group operating via vertex relabelling on the collection of all such rooted graphs $G$.

\begin{proposition}
	\label{pro:wlabelled}
	An unlabelled $3$-connected planar graph with $n$ vertices and a marked oriented root edge corresponds to either one or two rooted planar maps.
	\begin{enumerate}
		\item If it is rooted achiral, then it corresponds to $n!/2$ labelled graphs with an oriented root edge and admits exactly one non-trivial automorphism that fixes the oriented root edge and its orientation.
		\item If it is rooted chiral, then it corresponds to $n!$ labelled graphs with an oriented root edge and the only automorphism that fixes the oriented root edge and its orientation is the identity.
	\end{enumerate}
\end{proposition}
\begin{proof}
	Let $U$ denote an $n$-vertex unlabelled $3$-connected planar graph with an oriented root edge.
	
	It follows from Whitney's theorem that $U$ corresponds  to either one or two (unlabelled) rooted planar maps that are embeddings of $U$ on the sphere. 
	
	Each of these one or two maps admits $n!$ vertex labellings with numbers from $1$ to $n$.
	
	At the same time, any $3$-connected rooted planar graph with vertex set $[n]$ corresponds by Whitney's theorem to precisely two rooted $3$-connected vertex-labelled planar maps.
	
	Therefore, if $U$ corresponds to two (unlabelled) rooted planar maps, then it corresponds to precisely $n!$ labelled versions with vertex set $[n]$. If $U$ corresponds to only one unlabelled rooted planar map, then it corresponds to precisely $n!/2$ labelled versions with vertex set $[n]$.
	
	The symmetric group operates via vertex relabelling on the collection of all labelled versions $G$ of $U$. Accordingly, depending on whether $U$ corresponds to $n!$ or $n!/2$ labelled versions, its orbit has $n!$ or $n!/2$ elements and hence its stabiliser has $1$ or $2$ elements. This completes the proof.	
\end{proof}


\subsection{Automorphisms of 3-connected cubic planar maps and graphs}

The  number $q_n$ of rooted simple triangulations with $n + 2$ vertices (and hence $2n$ faces and $3n$ edges) was determined by Tutte~\cite{TUTTE1973437, zbMATH03169204}. Asymptotically as $n \to \infty$, 
\begin{align}
	\label{eq:enumsimple}
	q_n = \frac{\sqrt{6}}{32 \sqrt{\pi}} n^{-5/2} \left( \frac{27}{256} \right)^{-n} \left(1 + O\left(\frac{1}{n}\right) \right). 
\end{align}
Moreover, $\cQ(z) = \sum_{n \ge 1} q_n z^n$ satisfies
\begin{align}
	\label{eq:qeval}
	\cQ(27/256) = 1/8 \qquad \text{and} \qquad \cQ'(27/256) = 16/9.
\end{align}

The dual map of a $3$-connected map is $3$-connected.  Hence the dual map construction yields a bijection between rooted $3$-connected triangulations and rooted $3$-connected cubic planar maps. Care has to be taken with the orientation of the root edge in the dual. Taking the dual map twice flips the orientation of the root edge, hence the inverse of the dual map construction is taking the dual-map and reversing the orientation of the root edge.

The number of  rooted cubic $3$-connected maps with $2n$ vertices and $3n$ edges is hence also equal to $q_n$ for $n \ge 2$.

It was shown by \cite[Cor. 1]{zbMATH00767676} and~\cite{zbMATH03553320}  that only an exponentially small proportion of simple unrooted triangulations with a given size admit a non-trivial combinatorial automorphism.  See also related results in~\cite{zbMATH03635119,zbMATH03732086,zbMATH03821780, zbMATH05116340,zbMATH03781197,zbMATH03897726,zbMATH04178018,zbMATH03908476,zbMATH04083684,zbMATH03979089,zbMATH04073006, zbMATH00179625} that approach this problem for various types of classes using varying techniques.

\begin{proposition}
	\label{pro:expmap}
	\begin{enumerate}
		\item  Only an exponentially small proportion  of the unrooted $3$-connected cubic planar maps with $2n$ vertices admits a non-trivial combinatorial automorphism.
		\item Only an exponentially small proportion of the $q_n$  rooted $3$-connected cubic planar maps with $2n$ vertices admits a non-trivial combinatorial automorphism of the underlying unrooted planar map.
	\end{enumerate}
\end{proposition}
\begin{proof}
	If a $3$-connected unrooted cubic planar map admits a non-trivial combinatorial automorphism, then so does its dual. Consequently, by  \cite[Cor. 1]{zbMATH00767676},  only an exponentially small proportion of unlabelled unrooted $3$-connected cubic planar maps with $2n$ vertices admit a non-trivial combinatorial automorphism. 
	
	Any unrooted cubic planar map with $2n$ vertices has $3n$ edges and hence corresponds to at most $6n$ rooted cubic planar maps. So the number $\tilde{q}_n$ of $2n$-vertex unrooted cubic planar maps satisfies
	\[
		\tilde{q}_n \le q_n \le 6n \tilde{q}_n.
	\]
	Due to this polynomial bound, it follows that only an exponentially small fraction of the $q_n$ rooted $2n$-vertex $3$-connected cubic planar maps admits a combinatorial automorphism of the underlying unrooted planar map.
\end{proof}

\begin{proposition}
	\label{pro:unl3con}
	\begin{enumerate}
		\item The number $u_n$ of $2n$-vertex unlabelled $3$-connected cubic planar graphs with an oriented root edge  satisfies
		\[
		u_n = \frac{q_n}{2}(1 + O(c^n))
		\]
		for some constant $0<c<1$.
		\item Only an exponentially small fraction of these graphs admit a non-trivial automorphism of the underlying unrooted graph.
	\end{enumerate}
\end{proposition}
\begin{proof}
Any $3$-connected $2n$-vertex unlabelled cubic planar graph $U$ with an oriented root edge corresponds by Whitney's theorem to one or two unlabelled rooted cubic planar maps. 

By Proposition~\ref{pro:wlabelled}, if $U$ only corresponds to one such map $M$ then it admits a non-trivial graph automorphism that fixes the root edge and its orientation.  This automorphism extends to a combinatorial isomorphism between rooted planar maps corresponding to $U$. Since there is only one, it hence extends to a combinatorial automorphism of $M$.  It follows by Proposition~\ref{pro:expmap} that the number $q_n'$ of choices for $U$ is hence bounded by \[
q_n' = O(c^n)q_n
\] for a fixed constant $0<c<1$, and $q_n'$ is equal to the number of rooted achiral $3$-connected cubic planar maps with $2n$ vertices.

In the other case, $U$ corresponds to precisely two rooted planar maps. Hence there are  precisely $(q_n - q_n')/2$ choices for $U$. Thus
\begin{align}
	\label{eq:uq}
	u_n = q_n' + (q_n - q_n')/2 = q_n/2(1+O(c^n)).
\end{align}

We know by Proposition~\ref{pro:wlabelled} that if $U$ corresponds to two rooted planar maps then the only automorphism of $U$ that fixes the root edge and its orientation is the identity. But there could exist non-trivial  automorphisms of the underlying unrooted graph.

If $U$ admits a non-trivial automorphism of the underlying unrooted unlabelled graph $U'$, then each of the one or two unrooted planar maps corresponding to $U'$ admits a non-trivial combinatorial automorphism. $U$ has $3n$ edges, so there are at most $6n$ unlabelled graphs with an oriented root edge that have $U'$ as underlying unrooted  graph.

By Proposition~\ref{pro:expmap} it follows that in this case the number of possible choices for $U$ and the unrooted maps corresponding to it is bounded by $O(c^n) q_n$.
\end{proof}

\section{Network decomposition and Walsh index series}

\label{sec:nedewa}

\subsection{Walsh index series of graphs and networks}
\label{sec:wiso}

Let $G$ denote a simple graph. We let $\Aut(G)$ denote the collection of automorphisms of $G$. For any $\sigma \in \Aut(G)$  and $i \ge 1$ we let $\sigma_i$ denote the number of cycles of $\sigma$ of length $i$. The automorphism $\sigma$ induces a permutation $\sigma^{(2)}$ of the set of edges of $G$. There are two types of cycles of $\sigma^{(2)}$. A \emph{cylindrical} cycle $c$ of length $i$ has the property that for each edge $e=ab$ of $c$ we have $\sigma^i(a) = a$ and $\sigma^i(b) = b$. A \emph{m\"obius} cycle $c$ of length $i$ is characterised by $\sigma^i(a) = b$ and $\sigma^i(b) = a$ for each edge $e=ab$ of $c$. We let $\cyl_i(\sigma)$ and $\mob_i(\sigma)$ denote the numbers of cylindrical and m\"obius cycles of $\sigma^{(2)}$. The \emph{weight monomial} of $(G, \sigma)$ is defined by
\[
w(G, \sigma) = s_1^{\sigma_1} s_2^{\sigma_2} \cdots a_1^{\cyl_1(\sigma)} a_2^{\cyl_2(\sigma)} \cdots b_1^{\mob_1(\sigma)} b_2^{\mob_2(\sigma)} \cdots.
\]
A \emph{species} of graphs $\cG$ may be defined as a class of graphs that is closed under isomorphism. This is essentially a special case of the elegant and more general notion of combinatorial species that is formulated using category theoretic concepts~\cite{MR633783,MR1629341}. The \emph{Walsh index series} of $\cG$ introduced in~\cite{zbMATH03702699} is defined by the formal power series
\[
W_{\cG}(\bm{s}, \bm{a}, \bm{b}) = \sum_{G}  \frac{1}{|\Aut(G)|} \sum_{\sigma \in \Aut(G)} w(G, \sigma)
\]
with the sum index $G$ assuming one representative of each isomorphism class of multigraphs in $\cG$. This way, the ordinary generating series $\cG(x,y)$ with $[x^ny^m]\cG(x,y)$ the number of unlabelled graphs in $\cG$ with $n$ vertices and $m$ edges is given by
\[
\cG(x,y) = W_{\cG}(x, x^2, \ldots; y, y^2, \ldots; y, y^2, \ldots).
\]
Often, the notation $\tilde{\cG}(x,y)$ is used instead in order to distinguish it from the exponential generating function. We drop the tilde since there won't be any risk of confusion. We also use the short notation $\cG(x) = \cG(x,1)$.

For any $n \ge 1$ we let
$
	\Sym_{\cG}[n]
$
denote the finite set of pairs $(G, \sigma)$ with $G$ a graph from $\cG$ with vertex set $[n]$ and $\sigma$ a graph automorphism of $G$. Such a pair is called a \emph{symmetry}. It follows from basic properties of group operations that any unlabelled graph $G$ in $\cG$ with $n$ vertices corresponds to precisely $n!$ symmetries. The underlying graphs of these symmetries are labelled versions of $G$.

We define a \emph{network} as a connected cubic planar multigraph $N$ with an oriented root edge $e$ such that the graph $\hat{N}$ obtained by removing the root edge (but not its endpoints) is simple. The endpoints of the root edge are called the poles of the network. Networks are fundamental in decompositions of planar graphs~\cite{MR2524178,zbMATH05564308,zbMATH05216539,zbMATH05192214,zbMATH05158242, zbMATH05138815}.

 We let $\Autp(N)$ denote the set of graph automorphisms of $\hat{N}$ that fix each pole vertex. Likewise,  $\Autm(N)$ denotes the set of graph automorphisms of $\hat{N}$ that swap the pole vertices. If the root edge is a loop, then $\Autp(N) = \Autm(N)$.

An automorphism of $N$ is a pole-preserving automorphism, so we set $\Aut(N) = \Autp(N)$. Clearly $\Autp(N)$ always contains the identity, but it's possible that $\Autm(N) = \emptyset$. If $\Autm(N) \ne \emptyset$, then we can select any $\tau \in \Autm(N)$ and obtain a bijection
\[
\Autp(N) \to \Autm(N), \qquad \sigma \mapsto \tau \sigma.
\]
The inverse of that bijection maps $\nu \in \Autm(N)$ to $\tau^{-1}\nu \in \Autp(N)$.

For a species $\cN$ of networks, we define 
\[
	W_{\cN}^+(\bm{s}, \bm{a}, \bm{b}) = \sum_{n \ge 1} \frac{1}{n!} \sum_{(N, \sigma)}  w(\hat{N}, \sigma),
\]
with the sum index  $(N, \sigma)$ ranging over all pairs of a network $N$ in the species $\cN$ with vertex set $[n]$ and $\sigma \in \Autp(N)$. Likewise,
\[
	W_{\cN}^-(\bm{s}, \bm{a}, \bm{b}) = \sum_{n \ge 1} \frac{1}{n!} \sum_{(N, \sigma)}   w(\hat{N}, \sigma),
\]
with the sum index $(N, \sigma)$ ranging over networks $N$ of $\cN$ with vertex set $[n]$ and $\sigma \in \Autm(N)$.

\begin{remark}
	\label{re:diff}
	The notion of networks with which we work here differs slightly from the literature. We allow the two pole vertices to be identical, and elements of $\Autp(N)$ and $\Autm(N)$ are permutations of the entire vertex set of the network, including the poles. The edge index series $W^+$ and $W^-$ also differ slightly from the usual Walsh index series: We count the poles of the network, but not the marked oriented root edge. This notion will be more convenient for the network decomposition of cubic planar graphs.
\end{remark}

For each $i \ge 1$, define
\[
	W_{\cN, i}^{\pm} (\bm{s}, \bm{a}, \bm{b})  = W_{\cN}^{\pm}(s_i, s_{2i}, s_{3i}, \ldots; a_i, a_{2i}, a_{3i}, \ldots; b_i, b_{2i}, b_{3i}, \ldots),
\]
with $W_{\cN, i}^{\pm}$ referring to $W_{\cN, i}^{+}$ or $W_{\cN, i}^{-}$.
We let
\[
\cN(x,y) = y W^+_{\cN}(x, x^2, \ldots; y, y^2, \ldots; y, y^2, \ldots)\]
denote the ordinary generating series of networks (with $x$ counting  vertices and $y$ counting  edges). Furthermore,
\[
	\cN^{-}(x,y) = y W^-_{\cN}(x, x^2, \ldots; y, y^2, \ldots; y, y^2, \ldots)
\]
is the ordinary generating series of pole-symmetric networks in $\cN$.

\subsection{Decomposition}
\label{sec:dec}

We let $\cN$ denote the species of all networks. This species may be decomposed into five subspecies. We recall this decomposition and refer the reader to~\cite{zbMATH05122852,  zbMATH05255661,zbMATH07213288} for details.

A network $N$ with oriented root edge $e$ belongs to some of these subspecies according to the following criteria.
\begin{enumerate}
	\item  $\cL$ (Loop networks). The root edge  is a loop.
	\item $\cI$ (Isthmus networks). The root edge is an isthmus, meaning the graph $N - e$ obtained by deleting $e$ is disconnected.
	\item $\cS$ (Series networks). $N - e$ is connected, but contains a bridge that separates the endpoints of $e$.
	\item $\cP$ (Parallel networks). $N - e$ is connected, contains no bridge that would separate the endpoints of $e$, and either $e$ is part of a double edge in $N$ or deleting the endpoints of $e$ disconnects $N$.
	\item $\cH$ (Polyhedral networks). $N$ is obtained from a $3$-connected network (that is, a network with no multi-edge whose underlying graph is $3$-connected) by possibly replacing each non-root edge with a non-isthmus network.
\end{enumerate}

We use sum and minus operations on these species, such that, for example, 
\begin{align}
	\label{eq:D}
	\cD := \cL + \cS + \cP+ \cH = \cN - \cI
\end{align}
is the species of non-isthmus networks. We will also drop the arguments of the associated generating series when there is no risk of confusion, writing for example $W_\cN^+$ instead of $W_\cN^+(\bm{s}, \bm{a}, \bm{b})$. Recall from Remark~\ref{re:diff} that we work with slightly adjusted definitions of networks and edge index series.

\subsubsection{Loop networks}

A  network $N$ belongs to the species $\cL$ of loop-networks if its root edge  is a loop. As illustrated in Figure~\ref{fi:loop}, the vertex $s$ of the loop is adjacent to a single vertex $s'$, which is adjacent to two distinct vertices $u \ne v$ that form the poles of a non-loop network. There are two ways to orient this associated network.  Any automorphism of $N$ fixes $s$ and $s'$, and maps $\{u,v\}$ to $\{u,v\}$. Hence
\begin{align}
	\label{eq:ldec}
	W_{\cL}^+ &= \frac{s_1^2 }{2}(a_1^3 W_{\cN - \cL}^+ + a_1 a_2 W_{\cN - \cL}^-), \\
	W_{\cL}^- &= W_{\cL}^+.
\end{align}

\begin{figure}[H]
	\centering
	\begin{minipage}{\textwidth}
		\centering
		\includegraphics[scale=0.7]{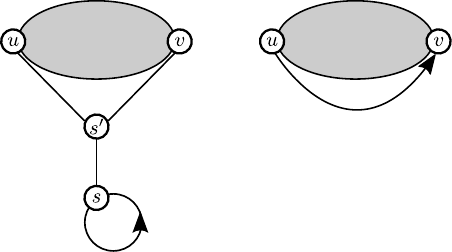}
		\caption{Decomposition of loop networks.}
		\label{fi:loop}
	\end{minipage}
\end{figure}

\subsubsection{Isthmus networks}

An isthmus network $N$ corresponds to an ordered pair of loop networks, each having an additional vertex. See Figure~\ref{fi:isthmus}. An automorphism of $\hat{N}$ that fixes the pole vertices is composed of automorphisms of the two loop network components. $N$ together with an automorphism of $\hat{N}$ that reverses the poles is composed from a single loop with an automorphism via the usual cycle composition. Hence
\begin{align}
	\label{eq:argc}
	W_\cI^+ = \frac{1}{s_1^2a_1^2} (W_\cL^+)^2, \\
	W_\cI^- = \frac{1}{s_2a_2} W_{\cL,2}^+.
\end{align}

\begin{figure}[H]
	\centering
	\begin{minipage}{\textwidth}
		\centering
		\includegraphics[scale=0.7]{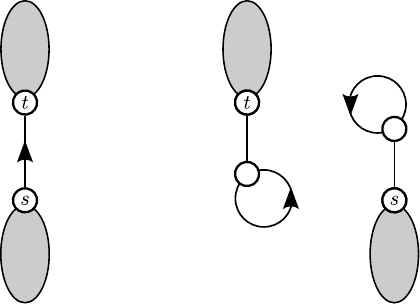}
		\caption{Decomposition of isthmus networks.}
		\label{fi:isthmus}
	\end{minipage}
\end{figure}

\subsubsection{Series networks}

Let $N$ denote a series network with poles $s$ and $t$. Then $N - e$ contains one or more bridges that separate the poles. Let $uv$ denote the bridge that is closest to $s$, directed from $u$ to $v$ such that $u$ is closer to $s$ than $v$. Then, as illustrated in Figure~\ref{fi:series}, $N$ corresponds to two networks $N_1$ and $N_2$, with root edges $su$ and $vt$. If $s=u$, then $N_1$ is a loop network, and likewise if $v=t$ then $N_2$ is a loop network. Both networks $N_1$ and $N_2$ cannot be isthmus networks, since there are multiple paths between their poles. Since we selected the bridge $uv$ which is closest to $s$, the network $N_1$ additionally cannot be a series network.  An automorphism of $\hat{N}$ that fixes the pole vertices must also fix the pole vertices of $N_1$ and $N_2$. An automorphism of $\hat{N}$ that maps $s$ to $t$ maps $N_1$ to an isomorphic copy $N_1'$ of itself which is attached to $t$. The vertex $u$ is mapped to the other pole  of $N_1'$, denoted by $u'$. Between $u$ and $u'$ there is either a single edge, or $u$ and $u'$ each are sources of a bridge and in the middle there is a non-isthmus network together with an automorphism that flips its pole vertices. Hence
\begin{align}
	\label{eq:S}
	W_\cS^+ &= W^+_{\cD- \cS} a_1 W_\cD^+, \\
	W_\cS^- &=  (b_1 +   a_2 W_\cD^-)  W^+_{\cD - \cS, 2}.
\end{align}
We remark that the subspecies $\cS_{\mathrm{d}}$ of series networks whose root edge is a double edge corresponds to the series composition of two loop networks and hence 
\begin{align}
	\label{eq:sdplus}
	W_{\cS_{\mathrm{d}}}^+  &= a_1 (W^+_{\cL})^2, \\
	W_{\cS_{\mathrm{d}}}^-  &=  b_1 W^+_{\cL,2}.
\end{align}
Consequently, the species  $\cS_{\mathrm{s}}$ of series networks whose root edge is a single edge  satisfies
\begin{align}
	W_{\cS_{\mathrm{s}}}^+  &= W_\cS^+ - 	W_{\cS_{\mathrm{d}}}^+, \\
	W_{\cS_{\mathrm{s}}}^-  &=  W_\cS^- - 	W_{\cS_{\mathrm{d}}}^- .
\end{align}

\begin{figure}[H]
	\centering
	\begin{minipage}{\textwidth}
		\centering
		\includegraphics[scale=0.7]{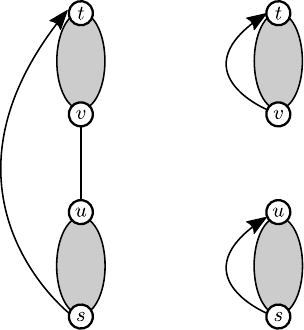}
		\caption{Decomposition of series networks.}
		\label{fi:series}
	\end{minipage}
\end{figure}

\subsubsection{Parallel networks}

We distinguish  two types of parallel networks, corresponding to species $\cP_{\mathrm{s}}$ and $\cP_{\mathrm{d}}$,  depending on whether the root edge is a simple edge or a double edge.

 If the root edge of a parallel network $N$ is a double edge, then its poles $s$ and $t$ are adjacent to (possibly identical) vertices $u$ and $v$, which form the poles of a smaller  non-isthmus network. See the first row of Figure~\ref{fi:parallel}. An automorphism of $\hat{N}$ that fixes or flips the poles of $N$ corresponds to an automorphism that fixes or flips the poles of the smaller network.
Hence
\begin{align}
	\label{eq:pdplus}
	W_{\cP_{\mathrm{d}}}^+ &=  s_1^2a_1^3 W_\cD^+, \\
	W_{\cP_{\mathrm{d}}}^- &=  s_2 a_2 b_1 W_\cD^-.
\end{align}
If the root edge is not a double edge, then as illustrated in the second row of Figure~\ref{fi:parallel} the parallel network corresponds to an unordered pair $\{N_1, N_2\}$ of two non-isthmus networks with poles $uv$ and $u'v'$. An automorphism of $\hat{N}$ that fixes the poles of $N$ may either fix the poles of $N_1$ and $N_2$ (and then map $N_i$ to $N_i$ for $i=1,2$), or it may transpose $u$ with $u'$, and $v$ with $v'$ (and then map $N_1$ to $N_2$). 
Likewise, an automorphism of $\hat{N}$ that reverses the poles of $N$ may either reverse the poles of $N_1$ and $N_2$ (and then map $N_i$ to $N_i$, $i=1,2$) or have cycles $(u v')$ and $(u'v)$ (and then map $N_1$ to $N_2$).
 Hence
\begin{align}
	\label{eq:rhopsplus}
	W_{\cP_{\mathrm{s}}}^+ &= \frac{s_1^2}{2} \left( a_1^4(W_{\cD}^+)^2 + a_2^2 W_{\cD,2}^+ \right), \\
	W_{\cP_{\mathrm{s}}}^- &= \frac{s_2}{2} \left( a_2^2(W_{\cD}^-)^2 + a_2^2W_{\cD,2}^+\right).
\end{align}

\begin{figure}[H]
	\centering
	\begin{minipage}{\textwidth}
		\centering
		\includegraphics[scale=0.7]{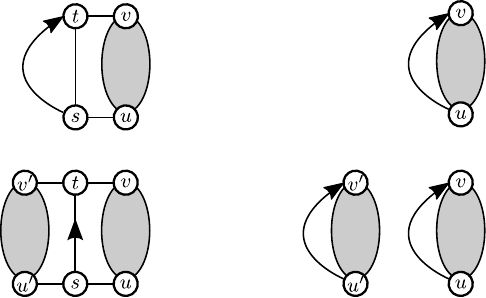}
		\caption{Decomposition of the two types of parallel networks.}
		\label{fi:parallel}
	\end{minipage}
\end{figure}

\subsubsection{Polyhedral networks}

A polyhedral network $N$ is obtained from a $3$-connected cubic planar graph $M$ with a directed root edge by inserting components $D(1), D(2), \ldots$ at its non-root edges as illustrated in Figure~\ref{fi:poly}.

Each component is either an edge or a non-isthmus network. For inserting components we fix a canonical enumeration and orientation of each non-root edge of $M$. If the component is an edge then we do not insert anything. If we insert a non-isthmus network $D$ with oriented root edge $st$ at a non-root edge $uv$ of $M$ (with its canonically chosen orientation) we delete the edge $st$ from $D$ and the edge $uv$ from $M$, and add edges $us$ and $vt$. We say $M$ is the $3$-connected core of $N$. It is uniquely determined by $N$.

\begin{figure}[H]
	\centering
	\begin{minipage}{\textwidth}
		\centering
		\includegraphics[scale=0.7]{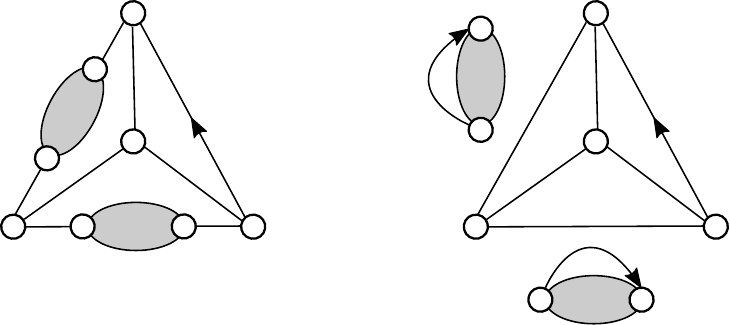}
		\caption{Decomposition of polyhedral networks.}
		\label{fi:poly}
	\end{minipage}
\end{figure}

Recall that $q_n$ denotes the number of $3$-connected rooted cubic planar maps with $2n$ vertices and hence $3n$ edges.

\begin{proposition}
	\label{pro:H}
	There exist multivariate power series $R_1(s_1, s_2,a_1,a_2, b_1)$ and $R_2(s_1, s_2,a_1,a_2,b_1)$  with nonnegative coefficients and a constant $0<c<1$ such that for all $n \ge 2$ and $i=1,2$
	\[
	[x^{2n}y^{3n-1}]R_i(x,x^2,y,y^2,y) = O(c^n)  q_n,
	\]
	with all other coefficients equal to zero, and
	\begin{align*}
		W_\cH^+ &= \sum_{n \ge 2} \frac{q_n}{2} s_1^{2n} (a_1 + a_1^2W_\cD^+)^{3n-1} + R_1(s_1,s_2, a_1 + a_1^2 W_\cD^+, a_2 + a_2^2 W_{\cD,2}^+, b_1 + a_2 W_\cD^-) \\
		W_\cH^- &= R_2(s_1,s_2, a_1 + a_1^2 W_\cD^+, a_2 + a_2^2 W_{\cD,2}^+, b_1 + a_2 W_\cD^-).
	\end{align*}
\end{proposition}
\begin{proof}

It follows from~\cite[Thm. 15]{MR2524178} (originating from \cite{zbMATH03702699}) that
	\begin{align}
		\label{eq:WWWp}
	W_{\cH}^+ =& \frac{1}{a_1 + a_1^2W_\cD^+}\sum_{n \ge 2} \frac{1}{(2n)!} \sum_{(M, \sigma)}  s_1^{\sigma_1} s_2^{\sigma_2}  \cdots  \\
	& (a_1 + a_1^2 W_\cD^+)^{\cyl_1(\sigma)} (a_2 + a_2^2 W_{\cD,2}^+)^{\cyl_2(\sigma)}  (a_3 + a_3^2 W_{\cD,3}^+)^{\cyl_3(\sigma)}  \cdots \nonumber \\ 
	& (b_1 + a_2 W_\cD^-)^{\mob_1(\sigma)} (b_2 + a_4 W_{\cD,2}^-)^{\mob_2(\sigma)} (b_3 + a_6 W_{\cD,3}^-)^{\mob_3(\sigma)} \cdots, \nonumber 
	\end{align}
	with $(M, \sigma)$ ranging over all pairs of a $3$-connected $2n$-vertex cubic planar graph $M$ with an oriented root edge, together with an automorphism $\sigma$ of $M$ that fixes each end of the root edge.
	 Likewise,
	\begin{align}
		\label{eq:WWWm}
	W_{\cH}^- =& \frac{1}{b_1 + a_2 W_\cD^-}\sum_{n \ge 2} \frac{1}{(2n)!} \sum_{(M, \sigma)}  s_1^{\sigma_1} s_2^{\sigma_2}  \cdots  \\
& (a_1 + a_1^2 W_\cD^+)^{\cyl_1(\sigma)} (a_2 + a_2^2 W_{\cD,2}^+)^{\cyl_2(\sigma)}  (a_3 + a_3^2 W_{\cD,3}^+)^{\cyl_3(\sigma)}  \cdots \nonumber  \\ 
& (b_1 + a_2 W_\cD^-)^{\mob_1(\sigma)} (b_2 + a_4 W_{\cD,2}^-)^{\mob_2(\sigma)} (b_3 + a_6 W_{\cD,3}^-)^{\mob_3(\sigma)} \cdots, \nonumber 
	\end{align}
	with $(M, \sigma)$ ranging over all pairs of a $3$-connected $2n$-vertex cubic planar graph $M$ with an oriented root edge, together with an automorphism $\sigma$ of $M$ that swaps the ends of the root edge.

	By Whitney's theorem, any rooted cubic planar map with $2n$ vertices admits $(2n)!$ vertex labellings. Any labelled $3$-connected cubic planar graph with an oriented root edge corresponds by Whitney's theorem to two labelled rooted planar maps. Hence there are
	\[
	q_n (2n)! / 2
	\]
	labelled $3$-connected cubic planar graphs with an oriented root edge. Therefore, $\sum_{n \ge 2} \frac{q_n}{2} s_1^{2n} (a_1 + a_1^2 W_\cD^+)^{3n-1}$ is precisely the contribution in $	W_\cH^+ $ where $\sigma=\id$.
	
	By Propositions~\ref{pro:aut1} and~\ref{pro:aut2}, any non-trivial automorphism of the core that swaps or fixes the ends of the oriented root edge is an involution, and there is at most one non-trivial automorphism per core-graph that fixes each end of the root edge, and there are at most two automorphisms per core-graph that swap the ends of the root edge.	
	
	 An involution can only have fixed points and $2$-cycles on the vertex set and on the edge-set, and any $2$-cycle on the edge set needs to be cylindrical. Hence we may write
		\begin{align*}
		W_\cH^+ &= \sum_{n \ge 2} \frac{q_n}{2} s_1^{2n} (a_1 + a_1^2W_\cD^+)^{3n-1} + R_1(s_1,s_2, a_1 + a_1^2 W_\cD^+, a_2 + a_2^2 W_{\cD,2}^+, b_1 + a_2 W_\cD^-) \\
		W_\cH^- &= R_2(s_1,s_2, a_1 + a_1^2 W_\cD^+, a_2 + a_2^2 W_{\cD,2}^+, b_1 + a_2 W_\cD^-)
	\end{align*}
	for uniquely determined series $R_i(s_1, s_2, a_1, a_2, b_1)$. Their coefficients satisfy the bounds
	\begin{align}
		\label{eq:rbound1}
		[x^{2n}y^{3n-1}] R_1(x,x^2, y, y^2, y) \le \frac{q_n}{2}, \\
		\label{eq:rbound2}
		[x^{2n}y^{3n-1}] R_2(x,x^2, y, y^2, y) \le q_n
	\end{align}
	for $n \ge 2$.
	All other coefficients are equal to zero.
	
	By Proposition~\ref{pro:unl3con} the  number of summands $(M, \sigma)$ with $\sigma \ne \id$ in Equations~\eqref{eq:WWWp} and~\eqref{eq:WWWm} is bounded by $(2n)! q_n O(c^n)$ for some constant $0<c<1$ (using the fact that any unlabelled planar map with $2n$ vertices corresponds to at most $(2n)!$ labelled planar maps.) Hence
	\[
	[x^{2n}y^{3n-1}]R_i(x,x^2,y,y^2,y) = O(c^n)  q_n 
	\]
	for $i=1,2$.  This completes the proof.
\end{proof}

\subsubsection{Simple networks}
	\label{sec:simplenetworks}
	A network is simple if it is not a loop network,  not the kind of parallel network illustrated in the top half of Figure~\ref{fi:parallel}, and not a series network whose two components are loop networks. Hence the species $\cN_{\mathrm{s}}$ of simple networks satisfies
\begin{align}
	\label{eq:simpnetwork}
	W_{\cN_{\mathrm{s}}}^+	&= W_\cH^+  + W_{\cP_{\mathrm{s}}}^+  + W_\cI^+ + W_{\cS_{\mathrm{s}}}^+.
\end{align}
Simple networks are precisely connected cubic planar graphs with an oriented root edge.

\section{Unlabelled cubic planar graphs are a critical graph class}

\label{sec:ucpg}

\subsection{Generating functions are analytic and finite at their radius of convergence}

Let $\cR$ denote any of the species of networks $\cN$, $\cN_{\mathrm{s}}$, $\cD$, $\cL$, $\cP$, $\cP_{\mathrm{d}}$, $\cP_{\mathrm{s}}$, $\cI$, $\cH$, $\cS$, $\cS_{\mathrm{d}}$,$\cS_{\mathrm{s}}$. It was shown by~\cite{zbMATH05122852,zbMATH07213288} that the number $r_n^{\mathrm{lab}}$ of networks in $\cR$ with vertex set $[n]$ satisfies
\begin{align}
	\label{eq:labelled}
	r_n^{\mathrm{lab}} \sim c^\mathrm{lab}_\cR n^{-5/2} \rho_{\mathrm{lab}}^{-n} n!
\end{align}
as $n \equiv 0 \mod 2$ becomes large, for a constant $ c^\mathrm{lab}_\cR>0$ that depends on $\cR$ and a constant  $\rho_{\mathrm{lab}}>0$ satisfying
\begin{align}
	\label{eq:rholab}
\rho_{\mathrm{lab}} = 0.319224606195452700761429068280\ldots.
\end{align}
Let  $\rho_\cR$ denote the radius of convergence of the ordinary generating series
\[
	\cR(x) := \cR(x,1) = W_{\cR}^+(x, x^2, \ldots; 1, 1, \ldots; 1,1, \ldots).
\]
We also set $\cR^-(x) := \cR^-(x,1) $ and
\[
	\cR^{\mathrm{lab}}(x) = \sum_{n \ge 1} \frac{r_n^{\mathrm{lab}}}{n!} x^n.
\]
For $k \ge 2$ we also write
\[
	\cR_k(x) = \cR(x^k) \qquad \text{and} \qquad \cR_k^-(x) = \cR^-(x^k).
\]

Our first observation is a rough bound on $\rho_\cR$ that is based on a bound for the number of automorphisms in terms of the maximal degree~\cite{zbMATH01838946}.
\begin{proposition}
	\label{pro:rho}
	We have
	\[
	\rho_{\mathrm{lab}} / 2 \le \rho_\cR \le \rho_{\mathrm{lab}}.
	\]
\end{proposition}
\begin{proof}
	For any connected simple graph $C$ with $n$ vertices and maximum degree $\Delta(C)$, the number of automorphisms $\aut(C)$ satisfies
	\[
	\aut(C) \le n (\Delta(C))!(\Delta(C)-1)^{n - \Delta(C) - 1}
	\]
	by~\cite[Thm. 2]{zbMATH01838946}.  Hence, if $C$ has maximum degree $3$, then
	\begin{align}
		\label{eq:aut}
		1 \le \aut(C) \le \frac{3}{8} n 2^{n}.
	\end{align}
	It is clear that this bound still holds if $C$ is a network from the class $\cR$ and we let $\aut(C)$ denote the number of automorphisms that fix the ends $u$ and $v$ of the oriented root edge, as this may only decrease the number of automorphisms. (Also, if the network is not simple because the root edge is a loop or because there is a non-root edge between $u$ and $v$, then any permutation of vertices that fixes $u$ and $v$ is an automorphism of the network if and only if it is an automorphism of the simple graph obtained by deleting the root edge.)

	By Burnside's relation we have
	\[
	r_n^{\mathrm{lab}} = \sum_C \frac{n!}{\aut(C)}
	\]
	with the sum index $C$ ranging over all  unlabelled $n$-vertex networks in $\cR$. By~\eqref{eq:aut}, it follows that the number $r_n$ of unlabelled $n$-vertex networks in $\cR$ satisfies
	\[
	\frac{8}{3} n^{-1}2^{-n} r_n n! \le r_n^{\mathrm{lab}} \le  r_n n!.
	\]
	By~\eqref{eq:labelled}, this yields
	\[
	c^\mathrm{lab}_\cR n^{-5/2} \rho_{\mathrm{lab}}^{-n}(1+o(1)) \le r_n \le \frac{3}{8} c^\mathrm{lab}_\cR n^{-3/2} (\rho_{\mathrm{lab}}/2)^{-n}(1+o(1)).
	\]
	This completes the proof.
\end{proof}

An important consequence of Proposition~\ref{pro:rho} is that $\rho_\cR <1$. In particular, for any $i \ge 2$ we have that $\cR(x^i)$ has radius of convergence $\rho_\cR^{1/i} > \rho_\cR$. This allows us to determine that all species of networks considered here have the same growth constant:

\begin{corollary}
	\label{cor:rho}
	There exists a constant $0<\rho\le\rho_{\mathrm{lab}}$ with 
	\[
		\rho_\cR= \rho
	\]
	for each species $\cR \in \{ 
	\cN,\cN_{\mathrm{s}}, \cD, \cL, \cP, \cP_{\mathrm{d}}, \cP_{\mathrm{s}}, \cI, \cH, \cS, \cS_{\mathrm{d}},\cS_{\mathrm{s}} \}$.
\end{corollary}
\begin{proof}
	Equation~\eqref{eq:argc} yields \[
	\rho_\cI= \rho_\cL.
	\]
	 Equation~\eqref{eq:S} implies 
	 \[
	 \cS(x) = (\cD(x) - \cS(x)) \cD(x).
	 \]
	 Clearly $\cD(x) - \cS(x)$ has nonnegative coefficients and is dominated coefficient-wise by $\cD(x)$, hence it follows that
	 \[
	 	\rho_\cS = \rho_\cD.
	 \]
	 Equation~\eqref{eq:sdplus} yields
	 \[
	 	\rho_{\cS_\mathrm{d}} = \rho_\cL.
	 \]	 
	 By Equation~\eqref{eq:pdplus}
	 \[
	 	\rho_{\cP_{\mathrm{d}}} = \rho_\cD.
	 \]
	 Equation~\eqref{eq:rhopsplus} implies
	 \[
	 	\cP_{\mathrm{s}}(x) = \frac{x^2}{2}(\cD(x)^2 + \cD(x^2)).
	 \]
	 Since $\rho_\cD<1$ we know that $\cD(x^2)$ has radius of convergence $\sqrt{\rho_\cD} > \rho_\cD$. Hence
	 \[
	 	\rho_{\cP_\mathrm{s}} = \rho_\cD.
	 \]
	 Since $\cP(x) = \cP_{\mathrm{s}}(x) + \cP_{\mathrm{d}}(x)$ it follows that
	 \[
	 	\rho_\cP= \rho_\cD.
	 \]
	 The series composition of two parallel networks is always a simple network, so we have
	 \[
	 	\cS(x) \ge \cS_{\mathrm{s}}(x) \ge \cP(x)^2.
	 \]
	 Using $\rho_\cS= \rho_\cP= \rho_\cD$, it follows that
	 \[
	 	\rho_{\cS_\mathrm{s}} =  \rho_\cD.
	 \]
	 By Equation~\eqref{eq:D} and Proposition~\ref{pro:H} we have
	 \[
	 	\cD(x) \ge \cH(x) \ge \frac{q_2}{2} x^4(1 + \cD(x))^5.
	 \]
	 Hence
	 \[
	 	\rho_\cH = \rho_\cD.
	 \]
	 Summing up, we know that 
	 \[
	 	\rho_\cL= \rho_\cI = \rho_{\cS_{\mathrm{d}}}
	 \]
	 and
	 \[
	 	\rho_\cD= \rho_\cS = \rho_{\cS_{\mathrm{s}}} = \rho_{\cP} = \rho_{\cP_\mathrm{s}} = \rho_{\cP_\mathrm{d}} = \rho_\cH.
	 \]
	 By~\eqref{eq:D} we have $\cD(x) = \cL(x) + \cP(x) + \cS(x) + \cH(x)$, and hence
	 \[
	 	\rho_\cL \ge \rho_{\cD}.
	 \]
	 On the other hand, by~\eqref{eq:ldec} and~\eqref{eq:D} 
	 we have
	 \[
	 	\cL(x) = \frac{x^2}{2}( (\cN -\cL)(x) + (\cN - \cL)^-(x)).
	 \]
	 The ordinary generating series $(\cN - \cL)^-(x)$ of pole-symmetric  networks in $\cN - \cL$ is dominated coefficient-wise by $(\cN - \cL)(x)$. Hence $\cL(x)$ has the same radius of convergence as
	 \[
	 	(\cN - \cL)(x) = \cI(x) + \cS(x) + \cP(x) + \cH(x).
	 \]
	 As $\rho_\cI=\rho_\cL$, this radius of convergence is given by $\min(\rho_\cL, \rho_\cD)$, in other words
	 \[
	 	\rho_\cL \le \rho_\cD.
	 \]
Using $\rho_\cL \ge \rho_\cD$ it follows that
\[
	\rho_\cL= \rho_\cD.
\]
As $(\cN - \cL)(x)$ and $\cL(x)$ both have radius of convergence $\rho_\cL$, it also follows that
\[
	\rho_\cN= \rho_\cL.
\]
By~\eqref{eq:simpnetwork} we also get
\[
	\rho_{\cN_{\mathrm{s}}} = \rho_\cD.
\]
This completes the proof.
\end{proof}

\begin{corollary}
		For each species $\cR \in \{ 
		\cN,\cN_{\mathrm{s}}, \cD, \cL, \cP, \cP_{\mathrm{d}}, \cP_{\mathrm{s}}, \cI, \cH, \cS, \cS_{\mathrm{d}},\cS_{\mathrm{s}} \}$ we have $\cR(\rho) < \infty$. 
\end{corollary}
\begin{proof}
	By Proposition~\ref{pro:H} we have that $\cH(x)$ dominates coefficient-wise the series
	\[
		\sum_{n \ge 2} \frac{q_n}{2} x^{2n} (1 +  \cD(x))^{3n-1}.
	\]
	Recall that $\sum_{n \ge 2} q_n x^n$ has radius of convergence $27/256$. Suppose that we have $\rho^2(1 + \cD(\rho))^3 > 27/256$. Since $\cD$ has nonnegative coefficients, it follows that there exists $\epsilon>0$ small enough with  $(\rho-\epsilon)^2(1 + \cD(\rho  - \epsilon) )^3 > 27/256$. Hence the series $(1 + \cD(x))\cH(x)$ which dominates $\sum_{n \ge 2} \frac{q_n}{2} x^{2n} (1 +  \cD(x))^{3n}$ coefficient-wise has radius of convergence at most $\rho -  \epsilon$,  a contradiction to Corollary~\ref{cor:rho}. Consequently,
	\begin{align}
		\label{eq:almostrho}
		\rho^2(1 + \cD(\rho))^3 \le 27/256.
	\end{align}
	
	This entails $\cD(\rho)<\infty$. The series $\cD(x)$ dominates $\cR(x)$ coefficient-wise (implying $\cR(\rho) \le \cD(\rho)< \infty$) in all cases except (possibly) $\cR\in\{\cI,\cN,\cN_{\mathrm{s}}\}$. For $\cR = \cI$, it follows from~\eqref{eq:argc} that $\cI(\rho) = \rho^{-2} \cL(\rho)^2 <\infty$. This implies $\cN_{\mathrm{s}}(\rho) \le \cN(\rho) < \infty$ and hence completes the proof.
\end{proof}

\subsection{Criticality and an equation for the growth constant}

From the equations of Walsh index series in Section~\ref{sec:dec} we immediately obtain
\begin{align*}
	\cD &= \cL + \cS + \cP + \cH,
\end{align*}
and
\begin{align*}
	\cL &= \frac{x^2}{2}( \cD  + \cI + \cD^- + \cI^- - 2\cL), & \cL^- &= \cL \\
	\cI &= \frac{1}{x^2} \cL^2, & \cI^- &= \frac{1}{x^2} \cL_2 \\
	\cS &= (\cD - \cS) \cD, & \cS^- &= (1 + \cD^-) (\cD_2 - \cS_2) \\
	\cP &= x^2\cD + \frac{x^2}{2}(\cD^2 + \cD_2), & \cP^- &= x^2 \cD^- + \frac{x^2}{2}( (\cD^-)^2 + \cD_2) 
\end{align*}
as well as
\begin{align*}
	\cH &= M(x, 1 + \cD) + R_1(x,x^2, 1+ \cD, 1 + \cD_2, 1 + \cD^-),  \\
	\cH^- &= R_2(x,x^2, 1+ \cD, 1 + \cD_2, 1 + \cD^-)
\end{align*}
for
\[
	M(x,y) := \sum_{n \ge 2} \frac{q_n}{2} x^{2n} y^{3n-1}.
\]
The equation for $\cS$ may be reformulated by
\begin{align}
	\label{eq:sthat}
	\cS= \frac{\cD^2}{1 +
		 \cD}.
\end{align}
Inserting the equation for $\cI$ and $\cI^-$ into the equation for $\cL$ yields
\begin{align}
	\label{eq:biermoe}
	0 = \cL^2 - 2(1 + x^2) \cL + x^2(\cD + \cD^-) + \cL_2.
\end{align}
In the ring of formal power series, this quadratic equation in $\cL$ has two solutions (if we treat $\cL_2$ as a constant). Since $\cL(0) = 0$ we may rule out one of the two, yielding
\begin{align}
	\label{eq:lroot}
	\cL = 1 + x^2 - \sqrt{1 + 2 x^2 + x^4 - \cL_2 - x^2(\cD + \cD^- )}.
\end{align}

\begin{lemma}
	\label{le:analyticsum}
	The functions
	\[
		(x, z, z^-) \mapsto 1 + x^2 - \sqrt{1 + 2 x^2 +  x^4 - \cL_2(x) - x^2(z + z^-)}
	\]
	and
	\[
		(x, z, z^-) \mapsto R_i(x,x^2,1 + z, 1 + \cD_2(x), 1 + z^-)
	\]
	for $i=1,2$
	are analytic at the point $(\rho, \cD(\rho), \cD^-(\rho))$. 
\end{lemma}
\begin{proof}
	By Equation~\eqref{eq:almostrho} we know \[
	\rho^2(1 + \cD(\rho))^3 \le 27/256.
	\] This entails by Proposition~\ref{pro:rho} and Equation~\eqref{eq:rholab} that
	\begin{align}
		\label{eq:drhobound}
		\cD(\rho) \le 3 (\rho^{2} 256)^{-1/3} -1 \le 3 (\rho_{\mathrm{lab}}^{2} 64)^{-1/3} -1 \le 0.61.
	\end{align}
	The smallest loop network has $6$ vertices. Since $\rho<1$, it follows by Proposition~\ref{pro:rho} and Equations~\eqref{eq:rholab} and~\eqref{eq:drhobound} that 
	\begin{align}
		\label{eq:l2rhobound}
		\cL_2(\rho) = \cL(\rho^2) \le \rho^6 \cL(\rho)  \le \rho_{\mathrm{lab}}^6 \cD(\rho) \le 0.00065.
	\end{align}
	Since $\cD^-$ is dominated coefficient-wise by $\cD$, it follows that
	\begin{align}
		\label{eq:lllbound}
	\left|2 \rho^2 + \rho^4 - \cL_2(\rho) - \rho^2(\cD(\rho) + \cD^-(\rho)) \right| \le 0.34.
	\end{align}	
	Since $\sqrt{1 +x}$ is analytic for $|x|<1$ (and  $\cL_2$ is analytic at $\rho$ because of $\rho<1$), it follows that the expression of $\cL$ in~\eqref{eq:lroot} in terms of the variables $x$, $\cD$ and $\cD^-$ is analytic at the point $(\rho, \cD(\rho), \cD^-(\rho))$.

	Furthermore, since $\rho<1$ we have for $\epsilon>0$ sufficiently small that $(\rho + \epsilon)^2 \le \rho$ and hence
	\[
		1 + \cD_2(\rho +\epsilon) \le 1 + \cD(\rho) \le (1 + \cD(\rho))^2.
	\]
	The series $\cD^-(x)$ is dominated coefficient-wise by $\cD(x)$. Thus, for integers $a,b,c \ge 0$ with $a+2b+c = 3n-1$ we have
	\[
		(1 + \cD(\rho) + \epsilon)^a (1 + \cD_2(\rho + \epsilon))^b (1 + \cD^-(\rho) + \epsilon)^c \le (1 + \cD(\rho) + \epsilon)^{3n-1}.
	\]
	 By Proposition~\ref{pro:H} and Equation~\eqref{eq:almostrho} it follows that for $\epsilon>0$ sufficiently small
	\begin{align*}
		&R_i(\rho+\epsilon,(\rho+\epsilon)^2,1 + \cD(\rho) + \epsilon, 1 + \cD_2(\rho + \epsilon), 1 + \cD^-(\rho) + \epsilon)  \\
		&\qquad\le \sum_{n \ge 2} q_n O(c^n) (\rho + \epsilon)^{2n}(1 + \cD(\rho) + \epsilon)^{3n-1} \\
		&\qquad< \infty
	\end{align*}
	for $i=1,2$. Hence  $R_i(x,x^2,1 + \cD, 1 + \cD_2(x), 1 + \cD^-)$ interpreted as a series in $x$, $\cD$ and $\cD^-$ is analytic at the point $(\rho, \cD(\rho), \cD^-(\rho))$.
\end{proof}

We may summarise that
\begin{align}
	\label{eq:done}
	\cD &= 1 + x^2 - \sqrt{1 + 2 x^2 + x^4 - \cL_2 - x^2(\cD + \cD^- )} + \frac{\cD^2}{1 + \cD} + x^2\cD \\ 
		&\quad\, + \frac{x^2}{2}(\cD^2 + \cD_2) + M(x, 1 + \cD) + R_1(x,x^2, 1+ \cD, 1 + \cD_2, 1 + \cD^-) \nonumber
\end{align}		
and
\begin{align}
	\label{eq:dtwo}
	\cD^- &= 1 + x^2 - \sqrt{1 + 2 x^2 + x^4 - \cL_2 - x^2(\cD + \cD^- )} + (1 + \cD^-) (\cD_2 - \cS_2) \\
	&\quad\, + x^2 \cD^- + \frac{x^2}{2}( (\cD^-)^2 + \cD_2)  + R_2(x,x^2, 1+ \cD, 1 + \cD_2, 1 + \cD^-). \nonumber
\end{align}
If we regard~\eqref{eq:done} and~\eqref{eq:dtwo} as equations in the indeterminates  $x$, $\cD$ and $\cD^-$, then by Lemma~\ref{le:analyticsum} the only summand on the right-hand side that might not be analytic at $(\rho, \cD(\rho), \cD^{-}(\rho))$ is $M(x, 1 + \cD)$. We will see below that, in fact, it isn't. As a preparation for this, we need to improve our crude bounds:

\begin{lemma}
	\label{le:prenum}
	We have
	\begin{align*}
		\cD(\rho) &\le 0.15724 \\
		\cL(\rho) &\le 0.092223 \\
		\cS(\rho) &\le 0.021364 \\
		\cP(\rho) &\le 0.024118 \\
		\cH(\rho) &\le \frac{5}{256} \\
		\cD_2(\rho) &\le 0.0016329 \\
		\cL_2(\rho) &\le 0.000097593.
	\end{align*}
\end{lemma}
\begin{proof}
	Recall from  Proposition~\ref{pro:rho} and Equations~\eqref{eq:drhobound} and~\eqref{eq:l2rhobound} we have
	\begin{align*}
		\cD^{-}(\rho) &\le \cD(\rho) \le 0.61, \\
		\cL_2(\rho) & \le 0.00065, \\
		\rho_{\mathrm{lab}} / 2 &\le \rho \le \rho_{\mathrm{lab}}.
	\end{align*}
	It will be convenient to use the notation $\rho_\mathrm{l} = \rho_{\mathrm{lab}} / 2$ and $\rho_\mathrm{u}= \rho_{\mathrm{lab}}$ for the lower and upper bound. 
	
	It follows that
	\begin{align}
		\label{eq:dfirst}
		\cL(\rho) \le 1 + \rho_\mathrm{u}^2 - \sqrt{1 + 2 \rho_{\mathrm{l}}^2 + \rho_{\mathrm{l}}^4 - \cL_2(\rho) - \rho_{\mathrm{u}}^2 2 \cD(\rho) } \le  0.14.
	\end{align}
	and, like in~\eqref{eq:l2rhobound}, 
	\begin{align}
	\cL_2(\rho) = \cL(\rho^2) \le \rho_{\mathrm{u}}^6 \cL(\rho)  \le 0.000149.
	\end{align}
	
	The function $x^2 / (1+x)$ is increasing for $x>0$, hence
	\begin{align}
		\cS(\rho) = \cD(\rho)^2/(1+\cD(\rho)) \le 0.232.
	\end{align}
		The smallest $\cD$-network has four vertices, hence
	\begin{align}
	\cD_2(\rho) \le \rho^4 \cD(\rho) \le \rho_{\mathrm{u}}^4 \cD(\rho) \le 0.0064.
	\end{align}
	Hence
	\begin{align}
		\label{eq:dlast}
		\cP(\rho) \le \rho_\mathrm{u}^2 \cD(\rho) + \frac{\rho_\mathrm{u}^2}{2}(\cD^2(\rho) + \cD_2(\rho)) \le 0.0936.
	\end{align}
	Using Equation~\eqref{eq:rbound1} and the fact that $\cD(x)$ dominates $\cD^-(x)$ coefficient-wise, as well as the fact that $(1 + \cD(x))^2$ dominates $1 + \cD_2(x)$ coefficient-wise, it follows that 
	\[
		\cH(\rho) \le 2 M(\rho,1+\cD(\rho)).
	\]
	Because of $q_1 = 1$ the series $\cQ(x) - x$ is increasing in $x$. By Equations~\eqref{eq:qeval} and~\eqref{eq:almostrho}, it follows that
	\begin{align*}
	M(\rho,1+\cD(\rho)) &\le \frac{1}{2} (\cQ(x) -x )\Big\vert_{x = \rho^2(1 + \cD(\rho))^3} \\
	&\le  \frac{1}{2} (\cQ(x) -x )\Big\vert_{x = 27/256} \\
	&= \frac{1}{2}(\frac{1}{8} - \frac{27}{256}) \\
	&= \frac{5}{512}.
	\end{align*}
	Hence
	\[
		\cH(\rho) \le \frac{5}{256}.
	\]
	With this, we arrive at
	\begin{align}
		\label{eq:final1}
		\cD(\rho) = \cL(\rho) + \cS(\rho) + \cP(\rho) + \cH(\rho) \le 0.14 + 0.232 + 0.0936 + \frac{5}{256} \le 0.49.
	\end{align}
	These improved bounds on $\cD(\rho)$ also improve the bounds in~\eqref{eq:dfirst}--\eqref{eq:dlast} and hence also the bound~\eqref{eq:final1}. In other words we may iterate.

	That is, we keep  $\rho_\mathrm{l}$,   $\rho_\mathrm{u}$ and the upper bound for $\cH(\rho)$ fixed, since the stronger upper bound on $\cD(\rho)$ does not improve them. We take the upper bounds for $\cL(\rho)$, $\cL_2(\rho)$, $\cS(\rho)$, $\cD_2(\rho)$, $\cP(\rho)$ and $\cD(\rho)$ from Equations~\eqref{eq:dfirst}--\eqref{eq:dlast} and ~\eqref{eq:final1} as a $6$-dimensional starting state. In each iteration, we apply the same numerical evaluations and arguments as in~\eqref{eq:dfirst}--\eqref{eq:dlast} and~\eqref{eq:final1} to obtain improved upper bounds which become the new state. 
	
	After thirty iterations using the computer algebra system Mathematica, we arrive at the bounds in the statement of the lemma.
\end{proof}

\begin{lemma}
	\label{le:keynumeric}
	The radius of convergence $\rho$ satisfies
	\[
	\rho^2 (1 + \cD(\rho))^3 = 27/256.
	\]
\end{lemma}
\begin{proof}
	By Equation~\eqref{eq:almostrho} we know \[
	\rho^2(1 + \cD(\rho))^3 \le 27/256.
	\]
	Suppose that $\rho^2(1 + \cD(\rho))^3 < 27/256$. Then $M(x, 1 + z)$ is analytic at $(\rho, \cD(\rho))$. Additionally using Lemma~\ref{le:analyticsum}, it follows that~\eqref{eq:done} and~\eqref{eq:dtwo} yield an equation
	\[
		(\cD, \cD^-) = (\phi_1(x, \cD, \cD^-), \phi_2(x, \cD, \cD^-))
	\]
	with $\phi_1$ and $\phi_2$ analytic at $(\rho, \cD(\rho), \cD^{-}(\rho))$. We have
	\begin{align*}
		\frac{\partial \phi_1}{\partial \cD} &= \frac{x^2}{2\sqrt{1 + 2 x^2 + x^4 - \cL_2 - x^2(\cD + \cD^- )}} + \frac{\cD^2 + 2\cD}{(1 + \cD)^2} + x^2 + x^2\cD  \\
		&\quad\,+ \partial_2 M(x, 1 + \cD) + \partial_3 R_1(x, x^2, 1+ \cD, 1 + \cD_2, 1 + \cD^-) \\
		\frac{\partial \phi_1}{\partial \cD^-} &= \frac{x^2}{2\sqrt{1 + 2 x^2 + x^4 - \cL_2 - x^2(\cD + \cD^- )}} + \partial_5 R_1(x, x^2, 1+ \cD, 1 + \cD_2, 1 + \cD^-) \\
		\frac{\partial \phi_2}{\partial \cD} &= \frac{x^2}{2\sqrt{1 + 2 x^2 + x^4 - \cL_2 - x^2(\cD + \cD^- )}} + \partial_3 R_2(x, x^2, 1+ \cD, 1 + \cD_2, 1 + \cD^-) \\
		\frac{\partial \phi_2}{\partial \cD^-} &= \frac{x^2}{2\sqrt{1 + 2 x^2 + x^4 - \cL_2 - x^2(\cD + \cD^- )}} + \cD_2 - \cS_2 + x^2 + x^2 \cD^- \\
		&\quad\,+ \partial_5 R_2(x, x^2, 1 + \cD, 1 + \cD_2, 1 + \cD^-).
	\end{align*}		
	Using~\eqref{eq:qeval} it follows that
	\begin{align*}
	\partial_2 M(\rho, 1 + \cD(\rho)) 	&= \sum_{n \ge 2} \frac{q_n}{2} \rho^{2n}(1 + \cD(\rho))^{3n-2} (3n-1) \\
										&\le  \frac{1}{2}  \sum_{n \ge 2} q_n (\rho^2(1 + \cD(\rho))^3)^{n} (3n-1) \\
										&\le \frac{3}{2} \sum_{n \ge 2} q_n (27/256)^{n} n - \frac{1}{2}\sum_{n \ge 2} q_n (27/256)^{n}  \\
										&= \frac{3}{2} (27/256)(\cQ'(27/256) - 1) - \frac{1}{2}(\cQ(27/256) - 27/256)\\
										&= \frac{3}{2}  (27/256)(16/9 - 1) - \frac{1}{2}(1/8 - 27/256) \\
										&= \frac{29}{256}.
	\end{align*}
	By~\eqref{eq:rbound1} and $1 + \cD_2(\rho) \le (1+ \cD(\rho))^2$ and $\cD^-(\rho) \le \cD(\rho)$ we get
	\begin{align*}
		\partial_3 R_1(\rho, \rho^2, 1 + \cD(\rho), 1 + \cD_2(\rho), 1 + \cD^-(\rho)) &\le \partial_2 M(\rho, 1 + \cD(\rho)), \\
		\partial_5 R_1(\rho, \rho^2, 1 + \cD(\rho), 1 + \cD_2(\rho), 1 + \cD^-(\rho)) &\le \partial_2 M(\rho, 1 + \cD(\rho)).
	\end{align*}
	Likewise,~\eqref{eq:rbound2} yields
	\begin{align*}
		\partial_3 R_2(\rho, \rho^2, 1 + \cD(\rho), 1 + \cD_2(\rho), 1 + \cD^-(\rho))  &\le 2 \partial_2 M(\rho, 1 + \cD(\rho)) \\
		\partial_5 R_2(\rho, \rho^2, 1 + \cD(\rho), 1 + \cD_2(\rho), 1 + \cD^-(\rho)) &\le 2 \partial_2 M(\rho, 1 + \cD(\rho)).
	\end{align*}
	Lemma~\ref{le:prenum} and Proposition~\ref{pro:rho} yield
	\begin{align*}
		&\frac{\rho^2}{2\sqrt{1 + 2\rho^2 + \rho^4 - \cL_2(\rho) - \rho^2(\cD(\rho) + \cD^-(\rho) )}}\\ &\le \frac{\rho_{\mathrm{lab}}^2}{2\sqrt{1 + 2 (\rho_{\mathrm{lab}}/2)^2 + (\rho_{\mathrm{lab}}/2)^4 - \cL_2(\rho) - \rho_{\mathrm{lab}}^22 \cD(\rho  )}} \\
		&\le 0.05047.
	\end{align*}
	The function $(x^2 + 2x)/(1+x)^2$ is increasing in $x$. Hence
	\[
		\frac{\cD(\rho)^2 + 2\cD(\rho)}{(1 + \cD(\rho))^2} + \rho^2 + \rho^2\cD(\rho) \le  0.37121.
	\]
	Moreover,
	\[
		\cD_2(\rho) - \cS_2(\rho) + \rho^2 + \rho^2 \cD^-(\rho) \le \cD_2(\rho)  + \rho^2 + \rho^2 \cD(\rho) \le 0.11957
	\]
	So we arrive at
	\begin{align*}
		\frac{\partial \phi_1}{\partial \cD}(\rho, \cD(\rho), \cD^-(\rho)) &\le 0.05047 + 0.37121 + \frac{29}{128} &\le 0.64825 \\
		\frac{\partial \phi_1}{\partial \cD^-}(\rho, \cD(\rho), \cD^-(\rho)) &\le 0.05047 + \frac{29}{256} &\le 0.16376 \\
		\frac{\partial \phi_2}{\partial \cD}(\rho, \cD(\rho), \cD^-(\rho)) &\le 0.05047 + \frac{29}{128} &\le 0.27704 \\
		\frac{\partial \phi_2}{\partial \cD^-}(\rho, \cD(\rho), \cD^-(\rho)) &\le 0.05047 + 0.11957 + \frac{29}{128} &\le 0.39661.
	\end{align*}
	It follows that
	\begin{align*}
	&\det \begin{pmatrix}
		1- \frac{\partial \phi_1}{\partial \cD} & -\frac{\partial \phi_1}{\partial \cD^-}  \\
		-\frac{\partial \phi_2}{\partial \cD}  & 1 - \frac{\partial \phi_2}{\partial \cD^-} 
	\end{pmatrix}(\rho, \cD(\rho), \cD^-(\rho)) \\
	&\ge (1 - 0.64825)(1 - 0.39661) - 0.16376 \cdot 0.27704 \\
	&\ge 0.166.
	\end{align*}
	Since this value is non-zero the implicit function theorem applies, yielding that $\cD(x)$ is analytic in a neighbourhood of $x=\rho$. But this is a violation of Pringsheim's theorem~\cite[Thm. IV.6]{MR2483235}. Consequently, it must hold that\[
	\rho^2(1 + \cD(\rho))^3 = 27/256.
	\]
\end{proof}
Lemma~\ref{le:keynumeric} justifies calling unlabelled cubic planar graphs a critical graph class.
\begin{corollary}
	\label{co:mcomp}
	We have
	\begin{align*}
		M(\rho, 1 + \cD(\rho)) &= \frac{1}{1 +\cD(\rho)} \frac{5}{512},\\
		\partial_1 M(\rho, 1 + \cD(\rho)) &= \frac{1}{\rho(1+ \cD(\rho))} \frac{21}{256},\\
		\partial_2 M(\rho, 1 + \cD(\rho)) &= \frac{1}{(1 +\cD(\rho))^2} \frac{29}{256}.
	\end{align*}
\end{corollary}
\begin{proof}
	We have by Lemma~\ref{le:keynumeric} and~\eqref{eq:qeval}
	\begin{align*}
		M(\rho, 1+ \cD(\rho)) 	&= \frac{1}{1 +\cD(\rho)} (\cQ(27/256) - 27/256)/2 \\
		&=  \frac{1}{1 +\cD(\rho)} \frac{5}{512}
	\end{align*}
	and
	\begin{align*}
		\partial_1 M(\rho, 1 + \cD(\rho)) 	&= \sum_{n \ge 2} q_n n \rho^{2n-1}(1 + \cD(\rho))^{3n-1} \\
		&= \frac{1}{\rho(1+ \cD(\rho))} \sum_{n \ge 2} q_n n (27/256)^n \\
		&= \frac{(27/256)}{\rho(1+ \cD(\rho))} (\cQ'(27/256) - 1) \\
		&= \frac{(21/256)}{\rho(1+ \cD(\rho))}.
	\end{align*}
	Moreover, 
	\begin{align*}
		&(1 + \cD(\rho))^2 \partial_2 M(\rho, 1 + \cD(\rho)) \\
		&=  \sum_{n \ge 2} \frac{q_n}{2} \rho^{2n}(1 + \cD(\rho))^{3n} (3n-1) \\
		&=  \sum_{n \ge 2} \frac{q_n}{2} (27/256)^{n} (3n-1) \\
		&= \frac{3}{2} (27/256)(\cQ'(27/256) - 1) - \frac{1}{2}(\cQ(27/256) - 27/256)\\
		&= \frac{3}{2}  (27/256)(16/9 - 1) - \frac{1}{2}(1/8 - 27/256) \\
		&= \frac{29}{256}.
	\end{align*}
\end{proof}
We will require the following sharpened bounds:
\begin{corollary}
	\label{co:sharpy}
	We have
	\begin{align*}
		\cD(\rho) &\le 0.038 \\
		\cL(\rho) &\le 0.012 \\
		\cS(\rho) &\le 0.0014 \\
		\cP(\rho) &\le 0.0058 \\
		\cH(\rho) &\le \frac{5}{256} \\
		\cD_2(\rho) &\le 0.0004 \\
		\cL_2(\rho) &\le 0.000012.
	\end{align*}
	Furthermore,
	\[
					\rho \ge 0.307 
	\]
	and
	\begin{align*}
		\partial_2 M \left( \rho, 1 + \cD(\rho)  \right) &\le \frac{29}{256} \\
				\partial_3 R_1(\rho, \rho^2, 1 + \cD(\rho), 1 + \cD_2(\rho), 1 + \cD^-(\rho)) &\le \frac{29}{256} \\
				\partial_5 R_1(\rho, \rho^2, 1 + \cD(\rho), 1 + \cD_2(\rho), 1 + \cD^-(\rho)) &\le \frac{29}{256} \\
				\partial_3 R_2(\rho, \rho^2, 1 + \cD(\rho), 1 + \cD_2(\rho), 1 + \cD^-(\rho))  &\le \frac{29}{128} \\
				\partial_5 R_2(\rho, \rho^2, 1 + \cD(\rho), 1 + \cD_2(\rho), 1 + \cD^-(\rho)) &\le \frac{29}{128}.
	\end{align*}
	Finally,
	\[
		\Delta := \frac{\rho^2}{2\sqrt{1 + 2\rho^2 + \rho^4 - \cL_2(\rho) - \rho^2(\cD(\rho) + \cD^-(\rho) )}} \le 0.047.
	\]
\end{corollary}
\begin{proof}
	For any real number $x$ with $\cD(\rho) \le x$ it follows by Lemma~\ref{le:keynumeric}  that
	\[
		\rho = \sqrt{(27/256) / (1 + \cD(\rho))^3} \ge \sqrt{(27/256) / (1 + x)^3}.
	\]	
	Thus, we may argue identically as in the proof of Lemma~\ref{le:prenum}, but after each iteration we may replace the lower bound $\rho_{\mathrm{l}}$ (which got initialized with $\rho_{\mathrm{lab}}/2$ in the first round and stayed constant in the proof of Lemma~\ref{le:prenum}) by $\sqrt{(27/256) / (1 + x)^3}$ for $x$ the current upper bound for $\cD(\rho)$ in that iteration.
	
	This yields the upper bounds for $\cD(\rho)$, $\cL(\rho)$, $\cS(\rho)$, $\cP(\rho)$, $\cH(\rho)$, $\cD_2(\rho)$ and $\cL_2(\rho)$. The value of $\rho_{\mathrm{l}}$ in the last iteration becomes the lower bound for $\rho$.
	
	The bounds for the partial derivatives of $M$, $R_1$ and $R_2$ were already observed in the proof of Lemma~\ref{le:keynumeric}.

	Finally, using $\cD^-(\rho) \le \cD(\rho)$, we obtain
	\begin{align*}
	\Delta  &\le \frac{\rho_{\mathrm{lab}}^2}{2\sqrt{1 + 2 \rho_{\mathrm{l}}^2 + \rho_{\mathrm{l}}^4 - \cL_2(\rho) - 2 \rho_{\mathrm{lab}}^2 \cD(\rho  )}} \\
		&\le  0.047.
	\end{align*}
\end{proof}

\section{Counting through local large deviation methods}

\label{sec:countingtl}

\subsection{Symmetry-enriched trees}

Let us call a \emph{symmetry} a pair of a labelled network together with an automorphism that fixes the root edge and either swaps or fixes the endpoints of the root edge. Call the symmetry \emph{positive}, if it fixes the endpoints of the root edge, and \emph{negative} if it swaps the ends. Call it non-trivial if the automorphism of the symmetry is not the identity. The vertex set of a symmetry is the vertex set of the underlying network.

For a species $\cR$ of networks, let $\Sym_{\cR}^{+}[n]$ denote the set of pairs $(N,\sigma)$ such that $N$ is a network in $\cR$ with vertex set $[n]$ and $\sigma\in \Autp(N)$. Likewise, $\Sym_{\cR}^{-}[n]$ denotes the set of pairs $(N,\sigma)$ with $\sigma\in \Autm(N)$.

The equations of Walsh index series in Section~\ref{sec:dec} may be summarised as follows:
\begin{align*}
	W^+_{\cD} &= W^+_{\cL}  + W^+_{\cS}  + W^+_{\cP} + W^+_{\cH} \\
	W^-_{\cD} &= W^+_{\cL}  + W^-_{\cS}  + W^-_{\cP} + W^-_{\cH} \\ 
	W_{\cL}^+ &= \frac{s_1^2 a_1^3}{2}(W^+_{\cS}  + W^+_{\cP} + W^+_{\cH}) + \frac{a_1}{2} \left(W_{\cL}^+\right)^2 \\
	&\quad\,+ \frac{s_1^2a_1 a_2}{2} (W^-_{\cS} + W^-_{\cP} +  W^-_{\cH}) + \frac{s_1^2 a_1 }{2 s_2} W^+_{\cL,2} \\
	W_{\cS}^+ &= (W^+_{\cL}  + W^+_{\cP} + W^+_{\cH}) a_1 W_{\cD}^+ \\
	W_\cS^- &=  (b_1 +   a_2 W_\cD^-)  (W^+_{\cL, 2} + W^+_{\cP, 2} + W^+_{\cH, 2} ) \\
	W_{\cP}^+ &=  s_1^2a_1^3 W_\cD^+ + \frac{s_1^2}{2} \left( a_1^4(W_{\cD}^+)^2 + a_2^2 W_{\cD,2}^+ \right) \\
	W_{\cP}^- &=  s_2 a_2 b_1 W_\cD^- + \frac{s_2}{2} \left( a_2^2(W_{\cD}^-)^2 + a_2^2W_{\cD,2}^+\right) \\
	W_\cH^+ &= \sum_{n \ge 2} \frac{q_n}{2} s_1^{2n} (a_1 + a_1^2W_\cD^+)^{3n-1} + R_1(s_1,s_2, a_1 + a_1^2 W_\cD^+, a_2 + a_2^2 W_{\cD,2}^+, b_1 + a_2 W_\cD^-) \\
	W_\cH^- &= R_2(s_1,s_2, a_1 + a_1^2 W_\cD^+, a_2 + a_2^2 W_{\cD,2}^+, b_1 + a_2 W_\cD^-).
\end{align*}

As detailed in Section~\ref{sec:dec} these equations represent how any symmetry admits a canonical recursive decomposition in terms of other symmetries.  At each recursive step we record the class of the current component  and  obtain a multitype rooted tree  with vertex types in
\[
\Pi
:=
\{\cD,\cL,\cS,\cP,\cH,\cD^-,\cS^-,\cP^-,\cH^-, \cL_{2}, \cP_2, \cH_2,  \cD_2, R_1, \cX,\cY,\cZ\}.
\]
That is, each vertex of the tree is assigned a type from that set. 

The leaves have type $\cX$ or $\cY$, with $\cX$ representing vertices of the underlying network, and leaves of type $\cY$ represent cylindrical $1$-cycles of edges. The auxiliary type $\cZ$ isolates the first summand in the expression for $W_\cH^+$, that is, the case where the $3$-connected core carries the identity automorphism. The reason why we introduce these special vertices is because we will be particularly interested in keeping track of their numbers.

This way, a vertex of type $\cD$ always has a single child, with type in $\{\cL, \cS,\cP, \cH\}$. A vertex of type $\cD^-$ also always has a single child, with type in $\{\cL, \cS^-, \cP^-, \cH^-\}$. For a vertex of type $\cL$ there are four disjoint cases, corresponding to the four summands in the equation for $W_{\cL}^+$. In the first case, corresponding to $\frac{s_1^2 a_1^3}{2}(W^+_{\cS}  + W^+_{\cP} + W^+_{\cH})$, we have two children of type $\cX$, three children of type $\cY$, and one child with type in $\{\cS, \cP, \cH\}$. In the second case, corresponding to $\frac{a_1}{2} \left(W_{\cL}^+\right)^2$,  we have one child of type $\cY$ and two children of type $\cL$. In the third case, corresponding to $\frac{s_1^2a_1 a_2}{2} (W^-_{\cS} + W^-_{\cP} +  W^-_{\cH}) $,  we have two children of type $\cX$, one child of type $\cY$, and one child with type in $\{\cS^-, \cP^-, \cH^-\}$.  In the fourth case, corresponding to $\frac{s_1^2 a_1 }{2 s_2} W^+_{\cL,2}$, we have  one child of type $\cY$, and a single child of type $\cL_2$. The rules for the children of vertices of type $\cS$, $\cS^-$, $\cP$, and $\cP^-$ are determined analogously.

For any vertex $v$ with type  $\cL_2$, $\cP_2$, $\cH_2$, or $\cD_2$, we do not decompose the corresponding symmetry any further, but rather store it as a \emph{local decoration} that gets assigned to the vertex $v$. We add as children of $v$ as many type-$\cX$ leaves as there are vertices in the symmetry. 

A vertex with type $\cH$ has a single child, which has either type $\cZ$ or type $R_1$, depending on whether the symmetry with an underlying $\cH$-network we are decomposing at this step is equipped with the identity automorphism or a non-trivial automorphism. For a vertex  $v$ of type $\cZ$, we add children of type $\cX$ according to the number of vertices of the $3$-connected core (of the symmetry we are decomposing at this step). We store the $3$-connected core as local decoration at $v$. For each non-root edge of the $3$-connected core we either add a child of type $\cY$ to $v$ if we don't insert any component at that edge, or we add two children of type $\cY$ and one child of type $\cD$ to $v$ if we insert a component at that edge. In the latter case we also add the information to the local decoration that this child of type $\cD$ corresponds to this edge of the $3$-connected core.

For vertices with type $R_1$ and $\cH^-$ we proceed analogously: The $\cH$-symmetry we decompose at this step has a $3$-connected core on which it induces an automorphism. We store this $3$-connected core and its automorphism as local decoration. We add children of type $\cX$ according to the number of vertices in the $3$-connected core. For each cylindrical $1$-edge-cycle we either add a child of type $\cY$ if the edge has no component, or we add two children of type $\cY$ plus one child of type $\cD$ if a component gets inserted at that edge. In the latter case, we also add the information that this child of type $\cD$ corresponds to that edge to the local decoration. For each m\"obius $1$-edge-cycle we either add nothing if no component gets inserted at that edge, or we add a child of type $\cD^-$ if a component gets inserted at that edge. In the latter case we also add the information that this child corresponds to that edge in the $3$-connected core to the local decoration. For each cylindrical $2$-edge-cycle we either add nothing if no component gets inserted at that edge, or we add a single child of type $\cD_2$ if a component gets inserted at that cycle. In the latter case, we again store the information that this child corresponds to that cycle as local decoration.

We call this multitype symmetry-enriched tree along with the local decoration a \emph{symmetry-enriched tree}. 

Note that the type of a vertex, together with its children and local decoration, leaves no ambiguity in the type of network decomposition. For example, each of the four cases of adding children to a vertex of type $\cL$ corresponds to a unique set of children that identifies which of the four cases was applied in the decomposition.  Hence a symmetry-enriched tree allows us to reconstruct the symmetry up to labels of vertices. By construction:

\begin{itemize}
	\item The number of type $\cX$ leaves  in the tree equals the number of vertices of the symmetry. 
	\item The number of type $\cY$-vertices in the tree equals the number of cylindrical $1$-edge-cycles in the symmetry.
	\item The local decorations store (among other things) the $3$-connected components.
\end{itemize}

Note that the possible local decorations of a vertex in the tree are only constrained by the set of children of that vertex. For example, if a vertex $v$ of type $\cZ$ has exactly $k$ children of type $\cX$, then we know that the associated $3$-connected core has $k$ vertices. Furthermore, among the networks sharing the same undecorated tree, any $3$-connected core with $k$ vertices is equally likely to occur as decoration of $v$.

\subsection{Branching mechanism}

We use the equations of Walsh index series to define a multi-type branching process that will help us enumerate unlabelled cubic planar networks.  Each type of vertex produces offspring of different types according to its own distribution. The result is a multi-type plane tree.

We denote the offspring distribution of a vertex with type $\cR \in \Pi$ by $\xi_\cR$. We specify the distribution of $\xi_\cR$ through its probability generating series $\phi_\cR$ expressed as a power series in indeterminates $(x_{\cR})_{\cR \in \Pi}$, such that

\begin{align*}
	\phi_{\cD} &= \frac{\cL(\rho)}{\cD(\rho)} x_{\cL}  + \frac{\cS(\rho)}{\cD(\rho)} x_{\cS}  + \frac{\cP(\rho)}{\cD(\rho)} x_{\cP} + \frac{\cH(\rho)}{\cD(\rho)} x_{\cH} \\
	\phi_{\cD^-} &= \frac{\cL(\rho)}{\cD^-(\rho)} x_{\cL}  + \frac{\cS^-(\rho)}{\cD^-(\rho)} x_{\cS^-}  + \frac{\cP^-(\rho)}{\cD^-(\rho)} x_{\cP^-} + \frac{\cH^-(\rho)}{\cD^-(\rho)} x_{\cH^-}
\end{align*}
and
\begin{align*}
	\phi_{\cL} &= \frac{\rho^2 x_\cX^2 x_{\cY}^3}{2}\left(\frac{\cS(\rho)}{\cL(\rho)} x_{\cS}  + \frac{\cP(\rho)}{\cL(\rho)}x_{\cP} + \frac{\cH(\rho)}{\cL(\rho)}x_{\cH}\right) + \frac{x_{\cY}}{2} \cL(\rho) x_{\cL}^2 \\
	&\quad\,+ \frac{\rho^2 x_{\cX}^2x_{\cY} }{2} \left(\frac{\cS^-(\rho)}{\cL(\rho)}x_{\cS^-} + \frac{\cP^-(\rho)}{\cL(\rho)} x_{\cP^-} +  \frac{\cH^-(\rho)}{\cL(\rho)} x_{\cH^-}\right) + \frac{ x_{\cY} }{2 } \frac{x_{\cL_2} \cL_2(\rho)}{\cL(\rho)} \\
	\phi_\cS &= \left(\cL(\rho) x_{\cL}  + \cP(\rho)  x_{\cP} + \cH(\rho) x_{\cH}\right) x_{\cY} \frac{\cD(\rho)}{\cS(\rho)} x_{\cD} \\
	\phi_{\cS^-} &=  \frac{1}{\cS^-(\rho)}(1 +    \cD^{-}(\rho) x_{\cD^-})  (\cL_2(\rho) x_{\cL_2} + \cP_2(\rho)x_{\cP_2} + \cH_2(\rho )x_{\cH_2} ) \\
	\phi_{\cP} &= \frac{1}{\cP(\rho)} \left( \rho^2 x_{\cX}^2x_{\cY}^3 \cD(\rho) x_\cD + \frac{\rho^2 x_{\cX}^2}{2} \left( x_{\cY}^4\cD(\rho)^2 x_{\cD}^2  +  \cD_2(\rho)x_{\cD_2} \right) \right) \\
	\phi_{\cP^-} &=  \frac{1}{\cP^-(\rho)} \left( \rho^2x_{\cX}^2 \cD^-(\rho) x_{\cD^-} + \frac{\rho^2x_{\cX}^2}{2} \left( \cD^-(\rho)^2 x_{\cD^-}^2  + \cD_2(\rho)x_{\cD_2}\right) \right) 
\end{align*}
and
\begin{align*}
	\phi_\cH &= \frac{1}{\cH(\rho)}  \frac{5}{512 (1 + \cD(\rho))} x_{\cZ} \\
	&+ \frac{1}{\cH(\rho)} R_1(\rho, \rho^2, 1 + \cD(\rho) , 1 +  \cD_2(\rho), 1 +  \cD^-(\rho) ) x_{R_1} \\
	\phi_{\cH^-} &= \frac{1}{\cH^-(\rho)} R_2(\rho x_{\cX},(\rho x_{\cX})^2, x_{\cY} + x_{\cY}^2 \cD(\rho) x_\cD , 1 +  \cD_2(\rho)x_{\cD_2}, 1 +  \cD^-(\rho) x_{\cD^-}) \\
	\phi_{\cZ} &= \frac{512 (1 + \cD(\rho))}{5} \sum_{n \ge 2} \frac{q_n}{2} (\rho x_{\cX})^{2n} (x_{\cY} + x_{\cY}^2 \cD(\rho) x_\cD)^{3n-1} \\
		\phi_{R_1} &= \frac{R_1(\rho x_{\cX}, (\rho x_{\cX})^2, x_{\cY} + x_{\cY}^2 \cD(\rho) x_\cD , 1 +  \cD_2(\rho)x_{\cD_2}, 1 +  \cD^-(\rho) x_{\cD^-})}{R_1(\rho, \rho^2, 1 + \cD(\rho) , 1 +  \cD_2(\rho), 1 +  \cD^-(\rho) )}
	\end{align*}
and
	\begin{align*}
	\phi_{\cD_2} &= \frac{\cD_2(\rho x_\cX)}{\cD_2(\rho)} \\
	\phi_{\cP_2} &= \frac{\cP_2(\rho x_\cX)}{\cP_2(\rho)} \\
	\phi_{\cH_2} &= \frac{\cH_2(\rho x_\cX)}{\cH_2(\rho)} \\
	\phi_{\cL_2} &= \frac{\cL_2(\rho x_\cX)}{\cL_2(\rho)} \\
	\phi_{\cX} &= 1 \\
	\phi_{\cY} &= 1.
\end{align*}

Let $\mT$ denote the tree generated by starting this multi-type branching process with a single vertex of type $\cD$. For each type $\cR$ of vertices we let $\#_\cR \mT$ denote the number of vertices with that type in $\mT$. We defined the offspring distribution precisely so that
\[
	\Exb{ x_1^{\#_\cX \mT} x_2^{\#_\cY \mT} } = 	\frac{1}{\cD(\rho)} W_\cD^+(\rho x_1, (\rho x_1)^2, \ldots; x_2, 1,1, \ldots; 1,1, \ldots).
\]

We form a symmetry-enriched tree from $\mT$ by adding uniform random local decorations at each vertex of $\mT$ with type in $\{\cZ, R_1, \cH^-, \cD_2, \cP_2, \cH_2, \cL_2\}$: 

If a vertex $v$ has type $\cZ$ with $k$ children of type $\cX$ and $\ell$ children of type $\cD$, we assign to it as local decoration a uniformly at random selected $3$-connected cubic planar graph with an oriented root edge and vertex set $[k]$. We match the $\ell$ children of type $\cD$  in a uniformly at random selected way to $\ell$ non-root edges of this $3$-connected cubic planar graph. 

If $v$ has type $R_1$ with $k$ children of type $\cX$, $\ell$ children of type $\cD$, $\ell^-$ children of type $\cD^-$, $m$ children of type $\cD_2$, then we endow $v$ with a local decoration chosen uniformly from all compatible positive symmetries of a $3$-connected cubic planar graph with an oriented root edge on the vertex set $[k]$, together with bijections from the children of $v$ of type $\cD$, $\cD^-$, and $\cD_2$  to the corresponding cylindrical $1$-cycles, m\"obius $1$-cycles, and cylindrical $2$-cycles. If $v$ has type $\cH^-$ we proceed analogously, just with a negative symmetry instead.

If $v$ has type $\cR_2$ for $\cR \in \{\cD, \cP, \cH, \cL \}$ with $m$ children of type $\cX$, we store a local decoration at $v$ given by a uniformly at random selected  positive $\cR$-symmetry generated by first selecting a symmetry  with vertex set $[m/2]$ and forming the cycle composition with an identical copy of itself, resulting after a uniform relabelling in a symmetry of size $m$.

The symmetry-enriched tree resulting from adding these local decorations to the tree $\mT$ corresponds to a symmetry, whose vertices we label uniformly at random with labels in the set $[\#_{\cX} \mT]$, the integers from $1$ to the number of type $\cX$ vertices in $\mT$. By construction, conditional on a given number $n$ of type $\cX$ vertices the symmetry is uniformly distributed on $\Sym^+_{\cD}[n]$.

\subsection{Branching structure}

Let $\mS$ denote the result of pruning away all descendants of non-root type $\cD$ vertices in $\mT$. We may view $\mS$ as a branching process like~$\mT$, except that non-root type $\cD$ vertices are infertile. Let $\mU$ denote the result of pruning away all descendants of vertices of type $\cZ$ in $\mS$.

\begin{lemma}
	\label{le:le1}
	The total number of vertices in $\mU$ has finite exponential moments.
\end{lemma}
\begin{proof}
	By construction, the branching process $\mT$ terminates almost surely. All types are reachable from the root, hence the spectral radius of the mean matrix $M$ of the offspring distribution is less than or equal to $1$~\cite[Thm. 10.1]{zbMATH03190745}.
	
	If we order the types in  $\Pi$ so that $\cD$, $\cZ$ come first and $\cD_2, \cP_2, \cH_2, \cL_2, \cY, \cX$ come last, then the mean matrix $M$ (with the coefficient in the $i$th row and $j$th column denoting the average number of children with the $j$th type of a vertex with the $i$th type) of the offspring distribution has the shape
	\[
		M= \begin{pmatrix}
			A & B \\ 0 & C 
		\end{pmatrix}
	\]
	with $A$ an irreducible matrix and $C$ a $6 \times 6$ matrix with the first $4$ columns equal to zero and the last two rows equal to zero. In particular, $0$ is the only eigenvalue of $C$.
	
	The tree $\mU$ is the result of modifying the offspring distribution so that vertices of type $\cD$ (other than the root vertex) and of type $\cZ$ become infertile. The corresponding mean matrix $\bar{M}$ is obtained from $M$ by  setting the first two rows to zero. Since $A$ is irreducible, it follows by the monotonicity property~\cite[Satz III]{zbMATH03054135} that the result of setting its first two rows to zero has spectral radius strictly less than the spectral radius of $A$, i.e. it is strictly smaller than $1$. Consequently, the spectral radius of $\bar{M}$ is strictly smaller than $1$. Hence this modified branching process is subcritical. 
	
	Since we modified the type $\cZ$ to become infertile, by Proposition~\ref{pro:H}, this modified branching process (which starts with a  single random  root vertex with random type chosen according to the probability generating series $\phi_\cD$) has an offspring distribution with finite exponential moments.
	
	For a subcritical branching process with finitely many types whose offspring distribution has finite exponential moments the total population has finite exponential moments as well. (This is a straightforward application of the implicit function theorem.) Hence the number of vertices in $\mU$ has finite exponential moments and the proof is complete.
\end{proof}

We set 
\begin{align*}
	(X,Y,L) &:= (\#_{\cX} \mS, \#_{\cY} \mS, \#_{\cD} \mS-1), \\
	(X_0, Y_0, L_0, Z) &:= (\#_{\cX} \mU, \#_{\cY} \mU, \#_{\cD} \mU-1, \#_{\cZ} \mU).
\end{align*}
By Lemma~\ref{le:le1}, $(X_0,  Y_0, L_0, Z)$ has finite exponential moments. We let $(\bar{X}, \bar{Y}, \bar{L})$ denote a triple of  nonnegative random variables with distribution given by the trivariate probability generating series $\phi_{\cZ}(x_{\cX}, x_{\cY}, x_{\cD})$. 	Note that
\begin{align}
	\label{eq:bary}
\bar{Y}= \frac{3}{2}\bar{X} -1 + \bar{L}.
\end{align}
We also set
\begin{align*}
	\mu &:= \frac{\cD(\rho)}{1 + \cD(\rho)}.
\end{align*}

\begin{proposition}
	\label{pro:expect}
	We have
	\begin{align*}
		\Exb{\bar{X}} &= \frac{42}{5},  \\
		\Exb{\bar{Y}} &= \frac{58}{5}(1 +  \mu), \\
		\Exb{\bar{L}} &= \frac{58}{5} \mu, \\
		\Exb{X} &= \Exb{X_0} + \Ex{Z} \Exb{\bar{X}}  < \infty, \\
		\Exb{Y} &= \Exb{Y_0} + \Ex{Z} \Exb{\bar{Y}}  < \infty, \\
		\Exb{L} &= \Exb{L_0} + \Ex{Z} \Exb{\bar{L}}  < \infty.
	\end{align*}
\end{proposition}
\begin{proof}
	From
	\[
		\phi_{\cZ} = \frac{512 (1 + \cD(\rho))}{5} \sum_{n \ge 2} \frac{q_n}{2} (\rho x_{\cX})^{2n} (x_{\cY} + x_{\cY}^2 \cD(\rho) x_\cD)^{3n-1}
	\]
	we obtain by Lemma~\ref{le:keynumeric} and~\eqref{eq:qeval} that
	\begin{align*}
		\Exb{\bar{L}} &= \frac{256}{5} \mu \sum_{n \ge 2} q_n \left( \frac{27}{256}\right)^{n}(3n-1) \\
		&= \frac{256}{5} \mu \left(\frac{81}{256} \cQ'(27/256) - \cQ(27/256) - \frac{54}{256} \right) \\
		&= \frac{58}{5} \mu
	\end{align*}
	and
	\begin{align*}
		\Exb{\bar{X}} &= \frac{512}{5}  \sum_{n \ge 2}  q_n \left( \frac{27}{256} \right)^{n}n \\
		&= \frac{512}{5}  \left( \frac{27}{256} \cQ'(27/256)  - \frac{27}{256}\right) \\
		&= \frac{42}{5} .
	\end{align*}
	The equation for $\Exb{\bar{Y}}$ now follows from~\eqref{eq:bary}.
	
	With $(\bar{X}_i, \bar{Y}_i, \bar{L}_i)_{i \ge 1}$ denoting independent copies of $(\bar{X}, \bar{Y}, \bar{L})$ that are also independent of  $\mU$, we have by construction
	\begin{align}
		(X, Y, L) \eqdist (X_0, Y_0, L_0) + \sum_{i=1}^Z (\bar{X}_i, \bar{Y}_i, \bar{L}_i).
	\end{align}
	As noted above, $(X_0, Y_0, L_0, Z)$ have finite exponential moments by Lemma~\ref{le:le1}. The expressions for and finiteness of the moments of $X$, $Y$ and $L$ follow by Wald's formula.
\end{proof}

\begin{lemma}
	\label{le:le3}
	With \[
		\vartheta= (\rho,\rho^2,1+\cD(\rho),1+\cD_2(\rho),1+\cD^-(\rho))
	\]
	and $\Delta$ as defined in Corollary~\ref{co:sharpy} set
	\begin{align*}
		a_{11} &=  1 - \left(\Delta  + 
		\frac{  \cD(\rho) }{ 1 +  \cD(\rho) } \right), \\
		a_{12} &= - \Delta - \partial_5 R_1(\vartheta), \\
		a_{21} &= -\Delta,  \\
		a_{22} &= 1 - \left(\Delta + \cD_2(\rho) - \cS_2(\rho) + \rho^2 (1 +  \cD^-(\rho)) +  \partial_5 R_2(\vartheta) \right), \\
		b_1 &=  \frac{ \cD(\rho) }{ \left( 1 +  \cD(\rho) \right)^2 } + \rho^2(1 +  \cD(\rho) ) +  \frac{29}{256(1 +\cD(\rho))^2} 
		+  \partial_3 R_1(\vartheta), \\
		b_2 &=   \partial_3 R_2(\vartheta), \\
		\tilde{b}_1 &= \frac{1}{\cD(\rho)(1 +\cD(\rho))} \frac{5}{512}
	\end{align*}
	and
	\begin{align*}
		\hat{b}_1 	&= \frac{1}{\cD(\rho)} \Big(\frac{\rho^2 \left(\cD(\rho) + \cD^-(\rho) - 2\cL(\rho) \right) + \frac{\rho}{2} \cL_2'(\rho)}{1 + \rho^2 - \cL(\rho)} + 2 \rho^2 \cD(\rho) + \rho^2 \left(\cD(\rho)^2 + \cD_2(\rho)\right) \\
		&\quad+ \frac{\rho^3}{2} \cD_2'(\rho) + \frac{21}{256 (1+\cD(\rho))} + \rho \partial_1 R_1(\vartheta) +2 \rho^2 \partial_2 R_1(\vartheta) + \rho \cD_2'(\rho) \partial_4 R_1(\vartheta) \Big),\\
		\hat{b}_2 	&= \frac{1}{\cD(\rho)} \Big(\frac{\rho^2 \left(\cD(\rho) + \cD^-(\rho) - 2\cL(\rho)\right) + \frac{\rho}{2} \cL_2'(\rho)}{1 + \rho^2 - \cL(\rho)} + (1 + \cD^-(\rho))\rho \left(\cD_2'(\rho) - \cS_2'(\rho)\right)\\
		&\quad + 2 \rho^2 \cD^-(\rho) + \rho^2 \left(\cD^-(\rho)^2 + \cD_2(\rho)\right) + \frac{\rho^3}{2} \cD_2'(\rho) + \rho \partial_1 R_2(\vartheta) + 2 \rho^2 \partial_2 R_2(\vartheta) \\
		&\quad+ \rho \cD_2'(\rho) \partial_4 R_2(\vartheta)\Big), 
	\end{align*}
	and
	\begin{align*}
		\check{b}_1 	&= \frac{1}{\cD(\rho)} \Big(\frac{\rho^2 \left(3\cD(\rho) + \cD^-(\rho) - 4\cL(\rho) \right) + \cL(\rho)^2 + \cL_2(\rho)}{2(1 + \rho^2 - \cL(\rho))} + \frac{\cD(\rho)^2}{\left(1+\cD(\rho)\right)^2}\\
		&\quad + \rho^2 \cD(\rho)\left(3 + 2\cD(\rho)\right) + \left(1 + 2\cD(\rho)\right)\left(\frac{29}{256(1+\cD(\rho))^2} + \partial_3 R_1(\vartheta)\right) \Big) \\
		\check{b}_2 &= \frac{1}{\cD(\rho)} \Big(\frac{\rho^2 \left(3\cD(\rho) + \cD^-(\rho) - 4\cL(\rho)\right) + \cL(\rho)^2 + \cL_2(\rho)}{2(1 + \rho^2 - \cL(\rho))} + \left(1 + 2\cD(\rho)\right)\partial_3 R_2(\vartheta)\Big).
	\end{align*}
	We have
	\begin{align*}
		\Exb{L} &= \frac{a_{22}b_1 - a_{12}b_2  } {a_{11}a_{22} - a_{21}a_{12}} < 1, \\
		\Ex{Z} &=  \frac{a_{22}\tilde{b}_1  } {a_{11}a_{22} - a_{21}a_{12}}, \\
		\Ex{X} &= \frac{a_{22}\hat{b}_1 - a_{12}\hat{b}_2  } {a_{11}a_{22} - a_{21}a_{12}}, \\
		\Ex{Y} &=  \frac{a_{22}\check{b}_1 - a_{12}\check{b}_2 } {a_{11}a_{22} - a_{21}a_{12}}.
	\end{align*}
\end{lemma}
\begin{proof}
	For a vertex type $\cR \in \{\cD, \cD^-, \cL, \cS, \cS^-, \cP, \cP^-, \cH, \cH^-\}$ let $\mS^\cR$  be defined analogously to $\mS$ except that we start with a  root vertex of type $\cR$. Let $N_\cR$ denote the number of non-root type $\cD$ vertices in~$\mS^\cR$.
	For $0 \le u \le 1$ set
	\[
	r_\cR(u) := \cR(\rho)\Exb{u^{N_\cR}}.
	\]
	It is elementary that
	\begin{align*}
		\Exb{L} &= \cD(\rho)^{-1} \frac{\partial r_\cD}{\partial u}(1).
	\end{align*}
	By construction of the offspring laws we have
	\begin{align*}
		r_\cD(u) &= r_{\cL}(u) + r_{\cS}(u) + r_{\cP}(u) + r_{\cH}(u),\\
		r_{\cD^-}(u) &= r_{\cL}(u) + r_{\cS^-}(u) + r_{\cP^-}(u) + r_{\cH^-}(u), \\
		r_{\cL}(u) &= \frac{\rho^2}{2} \left(r_{\cS}(u) + r_{\cP}(u) + r_{\cH}(u)\right) + \frac{r_{\cL}(u)^2}{2} \\
		&\quad\,+\frac{\rho^2}{2} \left(r_{\cS^-}(u) + r_{\cP^-}(u) + r_{\cH^-}(u)\right) + \frac{\cL_2(\rho)}{2},\\
		r_{\cS}(u) &= u\cD(\rho)\left(r_{\cL}(u) + r_{\cP}(u) + r_{\cH}(u)\right),\\
		r_{\cS^-}(u) &= (1+r_{\cD^-}(u))(\cD_2(\rho) - \cS_2(\rho)),\\
		r_{\cP}(u) &= \rho^2\cD(\rho)u + \frac{\rho^2}{2} \left(\cD(\rho)^2 u^2 + \cD_2(\rho)\right),\\
		r_{\cP^-}(u) &= \rho^2r_{\cD^-}(u) + \frac{\rho^2}{2} \left(r_{\cD^-}(u)^2 + \cD_2(\rho)\right),\\
		r_{\cH}(u) &= M(\rho,1+\cD(\rho)u)+R_1(\rho,\rho^2,1+\cD(\rho)u,1+\cD_2(\rho),1+r_{\cD^-}(u)),\\
		r_{\cH^-}(u) &= R_2(\rho,\rho^2,1+\cD(\rho)u,1+\cD_2(\rho),1+r_{\cD^-}(u)).
	\end{align*}
	Analogously as for~\eqref{eq:lroot} and~\eqref{eq:sthat} it follows that
	\begin{align*}
		r_{\cL}(u) &= 1 + \rho^2 - \sqrt{1 + 2\rho^2 + \rho^4 - \cL_2(\rho) - \rho^2 (r_\cD(u) + r_{\cD^-}(u))}, \\
		r_{\cS}(u) &= \frac{u\cD(\rho) r_\cD(u)}{1 + u\cD(\rho)}.
	\end{align*}
	Setting
	\[
		\lambda(u) :=(\rho,\rho^2,1+\cD(\rho)u,1+\cD_2(\rho),1+r_{\cD^-}(u))
	\]
	we hence obtain
	\begin{align*}
		r_{\cL}'(u) &= \frac{\rho^2(r_\cD'(u) + r_{\cD^-}'(u))}{2 \sqrt{1 + 2\rho^2 + \rho^4 - \cL_2(\rho) - \rho^2 (r_\cD(u) + r_{\cD^-}(u))}}, \\
		r_\cS'(u) 
		&= \frac{ \cD(\rho) r_\cD(u) }{ \left( 1 + u \cD(\rho) \right)^2 }
		+
		\frac{ u \cD(\rho) }{ 1 + u \cD(\rho) } r_\cD'(u), \\
		r_{\cS^-}'(u) &= r_{\cD^-}'(u)(\cD_2(\rho) - \cS_2(\rho)),\\
		r_{\cP}'(u) &= \rho^2\cD(\rho)(1 +  \cD(\rho) u), \\
		r_{\cP^-}'(u) &= \rho^2r_{\cD^-}'(u)(1 +  r_{\cD^-}(u)), \\
		r_{\cH}'(u) &=	\cD(\rho) \partial_2 M \left( \rho, 1 + \cD(\rho) u \right)
		+ \cD(\rho) \partial_3 R_1(\lambda(u))
		+ \partial_5 R_1(\lambda(u)) r_{\cD^-}'(u), \\
		r_{\cH^-}'(u) &=
		\cD(\rho) \partial_3 R_2(\lambda(u)) + \partial_5 R_2(\lambda(u)) r_{\cD^-}'(u).
	\end{align*}
	Note that $r_\cD(1) = \cD(\rho)$ and $r_{\cD^-}(1) = \cD^-(\rho)$. Using Corollary~\ref{co:mcomp} and $\vartheta = \lambda(1)$ it follows that
	\[
		\begin{pmatrix} a_{11} & a_{12} \\ a_{21} & a_{22} \end{pmatrix}  \begin{pmatrix} r_\cD'(1) \\ r_{\cD^-}'(1) \end{pmatrix} = \cD(\rho) \begin{pmatrix} b_1 \\ b_2  \end{pmatrix}.
	\]
The function $x \mapsto \frac{x}{1+x}$ is increasing for positive $x$, hence by Corollary~\ref{co:sharpy} we have
\begin{align*}
	 \Delta + \frac{\cD(\rho)}{1 + \cD(\rho)} &\le 0.047 + \frac{0.038}{1 + 0.038} \\&\le 0.084.
\end{align*}
and, using $\cD^-(\rho) \le \cD(\rho)$ (and the fact that $\cD_2(\rho) \ge \cS_2(\rho)$), 
\begin{align*}
	0 &\le \Delta + \cD_2(\rho) - \cS_2(\rho) + \rho^2 (1 +  \cD^-(\rho)) +  \partial_5 R_2(\vartheta) \\
	&\le 0.047 + 0.0004 + \rho_{\mathrm{lab}}^2(1 + 0.038) + \frac{29}{128} \\
	&\le 0.38.
\end{align*}
Furthermore,
\begin{align*}
	|a_{12}| &\le 0.047 + \frac{29}{256} \le 0.161, \\
	|a_{21}| &\le 0.047, \\
	b_2 &\le \frac{29}{128}.
\end{align*}
Furthermore, the function $x \mapsto \frac{x}{(1+x)^2}$ is increasing for $0<x<1$, hence
\begin{align*}
	b_1 &\le \frac{0.038}{(1 + 0.038)^2} + (1+0.038)\rho_{\mathrm{lab}}^2 + 29/256 + 29/256 \\
	&\le 0.368.
\end{align*}
Thus, the determinant of $(a_{ij})_{1 \le i,j \le 2}$  satisfies
\begin{align}
	\label{eq:deti}
	a_{11}a_{22} - a_{21}a_{12} \ge (1 - 0.084)(1- 0.38) - 0.161 \cdot 0.047 \ge 0.56.
\end{align}
Hence $(a_{ij})_{1 \le i,j \le 2}$ is invertible, yielding
\begin{align*}
	\Ex{L} &= \cD(\rho)^{-1} r_{\cD}'(1) = \frac{a_{22}b_1 - a_{12}b_2  } {a_{11}a_{22} - a_{21}a_{12}} 
\end{align*}
with an upper bound
\begin{align*}
\Ex{L} \le \frac{1 \cdot 0.368 + 0.161 \cdot 29 / 128  }{0.56} \le 0.73.
\end{align*}
Next, let $\tilde{N}_\cR$ denote the number of type $\cZ$ vertices in~$\mS^\cR$.
For $0 \le u \le 1$ set
\[
\tilde{r}_\cR(u) := \cR(\rho)\Exb{u^{\tilde{N}_\cR}}.
\]
This way
\begin{align*}
	\Exb{Z} &= \cD(\rho)^{-1} \frac{\partial\tilde{r}_\cD}{\partial u}(1).
\end{align*}
By construction of the offspring laws we have
\begin{align*}
	\tilde{r}_{\cS}(u) &= \cD(\rho)\left(\tilde{r}_{\cL}(u) + \tilde{r}_{\cP}(u) + \tilde{r}_{\cH}(u)\right),\\
	\tilde{r}_{\cP}(u) &= \rho^2\cD(\rho) + \frac{\rho^2}{2} \left(\cD(\rho)^2  + \cD_2(\rho)\right),\\
	\tilde{r}_{\cH}(u) &= M(\rho,1+\cD(\rho))u+R_1(\rho,\rho^2,1+\cD(\rho),1+\cD_2(\rho),1+\tilde{r}_{\cD^-}(u)),\\
	\tilde{r}_{\cH^-}(u) &= R_2(\rho,\rho^2,1+\cD(\rho),1+\cD_2(\rho),1+\tilde{r}_{\cD^-}(u)).
\end{align*}
For all other types the expression of $\tilde{r}_\cR$ is identical as before, just with $r$ replaced by $\tilde{r}$ on the right-hand side. We obtain analogously as before
\[
	\tilde{r}_{\cS}(u) = \frac{\cD(\rho) \tilde{r}_\cD(u)}{1 + \cD(\rho)}.
\]
Setting
\[
\tilde{\lambda}(u) :=(\rho,\rho^2,1+\cD(\rho),1+\cD_2(\rho),1+\tilde{r}_{\cD^-}(u))
\]
it follows that
\begin{align*}
	\tilde{r}_\cS'(u) &= \frac{\cD(\rho) \tilde{r}'_\cD(u)}{1 + \cD(\rho)}, \\
	\tilde{r}_{\cP}'(u) &= 0, \\
	\tilde{r}_{\cH}'(u) &=	M \left( \rho, 1 + \cD(\rho)\right)+ \partial_5 R_1(\tilde{\lambda}(u)) \tilde{r}_{\cD^-}'(u), \\
	\tilde{r}_{\cH^-}'(u) &=  \partial_5 R_2(\tilde{\lambda}(u)) \tilde{r}_{\cD^-}'(u).
\end{align*}
Using Corollary~\ref{co:mcomp}, $\tilde{r}_\cD(1) = \cD(\rho)$ and $\tilde{r}_{\cD^-}(1) = \cD^-(\rho)$ and  $\vartheta = \tilde{\lambda}(1)$ it follows that
\[
	\begin{pmatrix} a_{11} & a_{12} \\ a_{21} & a_{22} \end{pmatrix}  \begin{pmatrix} \tilde{r}_\cD'(1) \\ \tilde{r}_{\cD^-}'(1) \end{pmatrix} = \cD(\rho) \begin{pmatrix} \tilde{b}_1 \\ 0  \end{pmatrix}.
\]
Hence
\begin{align*}
	\Ex{Z} &= \cD(\rho)^{-1} \tilde{r}_{\cD}'(1) = \frac{a_{22}\tilde{b}_1  } {a_{11}a_{22} - a_{21}a_{12}} 
\end{align*}
Next, let $\hat{N}_\cR$ denote the number of type $\cX$ vertices in~$\mS^\cR$.
For $0 \le u \le 1$ set
\[
\hat{r}_\cR(u) := \cR(\rho)\Exb{u^{\hat{N}_\cR}}.
\]
This way
\begin{align*}
	\Exb{X} &= \cD(\rho)^{-1} \frac{\partial\hat{r}_\cD}{\partial u}(1).
\end{align*}
We have
\begin{align*}
	\hat{r}_{\cL}(u) &= \frac{u^2\rho^2}{2} \left(\hat{r}_{\cS}(u) + \hat{r}_{\cP}(u) + \hat{r}_{\cH}(u)\right) + \frac{\hat{r}_{\cL}(u)^2}{2} \\
	&\quad\,+\frac{u^2\rho^2}{2} \left(\hat{r}_{\cS^-}(u) + \hat{r}_{\cP^-}(u) + \hat{r}_{\cH^-}(u)\right) + \frac{\cL_2(\rho u)}{2},\\
	\hat{r}_{\cS}(u) &= \cD(\rho)\left(\hat{r}_{\cL}(u) + \hat{r}_{\cP}(u) + \hat{r}_{\cH}(u)\right),\\
	\hat{r}_{\cS^-}(u) &= (1+\hat{r}_{\cD^-}(u))(\cD_2(u \rho) - \cS_2(u \rho)),\\
	\hat{r}_{\cP}(u) &= u^2\rho^2\cD(\rho) + \frac{u^2\rho^2}{2} \left(\cD(\rho)^2 + \cD_2(u \rho)\right),\\
	\hat{r}_{\cP^-}(u) &= u^2\rho^2\hat{r}_{\cD^-}(u) + \frac{u^2\rho^2}{2} \left(\hat{r}_{\cD^-}(u)^2 + \cD_2(u\rho)\right),\\
	\hat{r}_{\cH}(u) &= M(u \rho,1+\cD(\rho))+R_1(u\rho,u^2\rho^2,1+\cD(\rho),1+\cD_2(u\rho),1+\hat{r}_{\cD^-}(u)),\\
	\hat{r}_{\cH^-}(u) &= R_2(u \rho,u^2\rho^2,1+\cD(\rho),1+\cD_2(u \rho),1+\hat{r}_{\cD^-}(u)).
\end{align*}
It follows that
\begin{align*}
	\hat{r}_{\cL}(u) &= 1 + u^2\rho^2 - \sqrt{1 + 2u^2\rho^2 + u^4\rho^4 - \cL_2(u\rho) - u^2\rho^2 (\hat{r}_\cD(u) + \hat{r}_{\cD^-}(u))}, \\
	\hat{r}_{\cS}(u) &= \frac{\cD(\rho) \hat{r}_\cD(u)}{1 + \cD(\rho)}.
\end{align*}
Using  Corollary~\ref{co:mcomp}
\begin{align*}
	\hat{r}_\cS'(1) 	&= \frac{\cD(\rho)}{1 + \cD(\rho)} \hat{r}_\cD'(1),\\
	\hat r_{\cS^-}'(1) 	&= (\cD_2(\rho) - \cS_2(\rho))\hat{r}_{\cD^-}'(1) + (1 + \cD^-(\rho))\rho (\cD_2'(\rho) - \cS_2'(\rho)),\\
	\hat r_\cP'(1) 		&= 2 \rho^2 \cD(\rho) + \rho^2 (\cD(\rho)^2 + \cD_2(\rho)) + \frac{\rho^3}{2} \cD_2'(\rho),\\
	\hat{r}_{\cP^-}'(1) &= \rho^2 (1 + \cD^-(\rho)) \hat{r}_{\cD^-}'(1) + 2 \rho^2 \cD^-(\rho) +\rho^2(\cD^-(\rho)^2 + \cD_2(\rho))\\
	&\quad + \frac{\rho^3}{2} \cD_2'(\rho),\\
	\hat{r}_\cH'(1) 		&= \partial_5 R_1(\vartheta) \hat{r}_{\cD^-}'(1) + \frac{1}{1+ \cD(\rho)} \frac{21}{256} + \rho \partial_1 R_1(\vartheta) \\
	&\quad + 2 \rho^2 \partial_2 R_1(\vartheta) + \rho \cD_2'(\rho) \partial_4 R_1(\vartheta),\\
	\hat{r}_{\cH^-}'(1)	&= \partial_5 R_2(\vartheta) \hat{r}_{\cD^-}'(1) + \rho \partial_1 R_2(\vartheta) + 2\rho^2 \partial_2R_2(\vartheta) + \rho \cD_2'(\rho) \partial_4 R_2(\vartheta).
\end{align*}
Recall from~\eqref{eq:lroot}
\[
	\cL(\rho) = 1 + \rho^2 - \sqrt{1 + 2 \rho^2 + \rho^4 - \cL_2(\rho) - \rho^2(\cD(\rho) + \cD^-(\rho) )}.
\]
Hence
\begin{align*}
	\hat{r}_\cL'(1) &= 2\rho^2 - \frac{4\rho^2 + 4\rho^4 - \rho\cL_2'(\rho) - 2\rho^2 (\cD(\rho) + \cD^-(\rho)) - \rho^2 (\hat r_\cD'(1) + \hat r_{\cD^-}'(1))}{2 (1 + \rho^2 - \cL(\rho))}\\
	&= \Delta (\hat{r}_\cD'(1) + \hat{r}_{\cD^-}'(1)) + \frac{\rho^2 (\cD(\rho) + \cD^-(\rho)-2\cL(\rho)) + \frac{\rho}{2} \cL_2'(\rho)}{1 + \rho^2 - \cL(\rho)}.
\end{align*}
\[
\begin{pmatrix} a_{11} & a_{12} \\ a_{21} & a_{22} \end{pmatrix}  \begin{pmatrix} \hat{r}_\cD'(1) \\ \hat{r}_{\cD^-}'(1) \end{pmatrix} = \cD(\rho) \begin{pmatrix} \hat{b}_1 \\ \hat{b}_2  \end{pmatrix}.
\]
Hence
\begin{align*}
	\Ex{X} &= \cD(\rho)^{-1} \hat{r}_{\cD}'(1) = \frac{a_{22}\hat{b}_1 - a_{12}\hat{b}_2 } {a_{11}a_{22} - a_{21}a_{12}}.
\end{align*}

It remains to treat the number of type $\cY$ vertices. Let $\check{N}_\cR$ denote the number of type $\cY$ vertices in~$\mS^\cR$ and set
\[
\check{r}_\cR(u) := \cR(\rho)\Exb{u^{\check{N}_\cR}}
\]
for $0 \le u \le 1$. This way
\begin{align*}
	\Exb{Y} &= \cD(\rho)^{-1} \frac{\partial\check{r}_\cD}{\partial u}(1).
\end{align*}
We have
\begin{align*}
	\check{r}_\cD(u) &= \check{r}_{\cL}(u) + \check{r}_{\cS}(u) + \check{r}_{\cP}(u) + \check{r}_{\cH}(u),\\
	\check{r}_{\cD^-}(u) &= \check{r}_{\cL}(u) + \check{r}_{\cS^-}(u) + \check{r}_{\cP^-}(u) + \check{r}_{\cH^-}(u), \\
	\check{r}_{\cL}(u) &= \frac{u^3\rho^2}{2} \left(\check{r}_{\cS}(u) + \check{r}_{\cP}(u) + \check{r}_{\cH}(u)\right) + \frac{u\check{r}_{\cL}(u)^2}{2} \\
	&\quad\,+\frac{u\rho^2}{2} \left(\check{r}_{\cS^-}(u) + \check{r}_{\cP^-}(u) + \check{r}_{\cH^-}(u)\right) + \frac{u\cL_2(\rho)}{2},\\
	\check{r}_{\cS}(u) &= u\cD(\rho)\left(\check{r}_{\cL}(u) + \check{r}_{\cP}(u) + \check{r}_{\cH}(u)\right),\\
	\check{r}_{\cS^-}(u) &= (1+\check{r}_{\cD^-}(u))(\cD_2(\rho) - \cS_2(\rho)),\\
	\check{r}_{\cP}(u) &= u^3\rho^2\cD(\rho) + \frac{\rho^2}{2} \left(u^4\cD(\rho)^2 + \cD_2(\rho)\right),\\
	\check{r}_{\cP^-}(u) &= \rho^2\check{r}_{\cD^-}(u) + \frac{\rho^2}{2} \left(\check{r}_{\cD^-}(u)^2 + \cD_2(\rho)\right),\\
	\check{r}_{\cH}(u) &= M(\rho,u+u^2\cD(\rho))+R_1(\rho,\rho^2,u+u^2\cD(\rho),1+\cD_2(\rho),1+\check{r}_{\cD^-}(u)),\\
	\check{r}_{\cH^-}(u) &= R_2(\rho,\rho^2,u+u^2\cD(\rho),1+\cD_2(\rho),1+\check{r}_{\cD^-}(u)).
\end{align*}
It follows that
\[
\check{r}_{\cS}(u) = \frac{u\cD(\rho) \check{r}_\cD(u)}{1 + u\cD(\rho)}.
\]
Using Corollary~\ref{co:mcomp} we obtain
\begin{align*}
	\check{r}_\cS'(1) 	&= \frac{\cD(\rho)^2}{(1+\cD(\rho))^2} + \frac{\cD(\rho)}{1 + \cD(\rho)} \check{r}_\cD'(1),\\
	\check r_{\cS^-}'(1) 	&= (\cD_2(\rho) - \cS_2(\rho))\check{r}_{\cD^-}'(1),\\
	\check r_\cP'(1) 		&= \rho^2 \cD(\rho)\left(3 + 2\cD(\rho)\right),\\
	\check{r}_{\cP^-}'(1) &= \rho^2 (1 + \cD^-(\rho)) \check{r}_{\cD^-}'(1),\\
	\check{r}_\cH'(1) 		&= \partial_5 R_1(\vartheta) \check{r}_{\cD^-}'(1) + \left(1 + 2\cD(\rho)\right)\left(\frac{29}{256(1+\cD(\rho))^2} + \partial_3 R_1(\vartheta)\right),\\
	\check{r}_{\cH^-}'(1)	&= \partial_5 R_2(\vartheta) \check{r}_{\cD^-}'(1) + \left(1 + 2\cD(\rho)\right)\partial_3 R_2(\vartheta).
\end{align*}
Differentiating the displayed equation for $\check{r}_{\cL}(u)$ at $u=1$ yields
\begin{align*}
	\check{r}_\cL'(1) &= \Delta \left(\check{r}_\cD'(1) + \check{r}_{\cD^-}'(1)\right) + \frac{\rho^2 \left(3\cD(\rho) + \cD^-(\rho)-4\cL(\rho)\right) + \cL(\rho)^2 + \cL_2(\rho)}{2(1+\rho^2-\cL(\rho))}.
\end{align*}
Consequently,
\[
\begin{pmatrix} a_{11} & a_{12} \\ a_{21} & a_{22} \end{pmatrix}  \begin{pmatrix} \check{r}_\cD'(1) \\ \check{r}_{\cD^-}'(1) \end{pmatrix} = \cD(\rho) \begin{pmatrix} \check{b}_1 \\ \check{b}_2  \end{pmatrix}.
\]
Hence, using~\eqref{eq:deti},
\begin{align*}
	\Ex{Y} &= \cD(\rho)^{-1} \check{r}_{\cD}'(1) = \frac{a_{22}\check{b}_1 - a_{12}\check{b}_2 } {a_{11}a_{22} - a_{21}a_{12}}.
\end{align*}
This completes the proof.
\end{proof}

\begin{proposition}
	\label{pro:linearbound}		
	Almost surely,
	\[
		\#_{\cD} \mT \le \frac{5}{2} \#_{\cX} \mT.
	\]
\end{proposition}
\begin{proof}
	In $\mT$, any vertex $u$ of type $\cD$ has a unique child $v$. The type of $v$ belongs to $\{\cL, \cS, \cP, \cH\}$. If $v$ has type $\cL$, $\cS$, or $\cP$, then $v$ has always at least one child $w$ with type of $w$ in $\cX$ or $\cY$. If $v$ has type $\cH$ then $v$ has a unique child $v'$, and the type of $v'$ is either $\cZ$ or $R_1$. In both cases, $v'$ has always at least one child $w'$ with type $\cX$.
	
	This way, we obtain an injective mapping from the type $\cD$ vertices of $\mT$ to the vertices of $\mT$ with type $\cX$ or $\cY$. In other words
	\[
		\#_\cD \mT \le \#_\cX\mT + \#_\cY \mT.
	\]
	The vertices of type $\cX$ of $\mT$ correspond precisely to the vertices of the network represented by $\mT$ and its decorations. The vertices of type $\cY$ of $\mT$ correspond injectively to edges of this network. In any cubic graph, the number of edges equals $3/2$ times the number of vertices. Hence,
	\[
		\#_\cY \mT \le \frac{3}{2} \#_\cX\mT.
	\]
	It follows that
	\[
		\#_\cD \mT \le \frac{5}{2} \#_\cX\mT.
	\]
\end{proof}

\subsection{Total progeny}

Our aim in this section is to determine the asymptotic probability for $\mT$ to have a given even number of type $\cX$ vertices, and analyse the proportion of other types of vertices on that event.

\subsubsection{Deviation inequalities}

For the reader's convenience, we recall two standard deviation bounds that we will employ in our subsequent analysis of $\mT$. Let $\mathrm{Bin}(n,p)$ denote the binomial distribution with success probability $p$ out of $n$ tries.

\begin{proposition}
	\label{pro:binomial}
	For each $0<p<1$ there exist $C,c>0$ such that for all $n \ge 1$ and  $0 \le k \le n$ we have
	\[
	\Prb{\mathrm{Bin}(n,p) = k} \le  \frac{C}{\sqrt{n}} \exp\left( - c \frac{(k- n p)^2}{n } \right).
	\]
\end{proposition}
\begin{proof}
	It suffices to prove this for $0<k<n$, since the cases $k=0$ and $k=n$ then follow by adjusting the constants $C$ and $c$.
	
	Hence, suppose that $0<k<n$. By~\cite[Prop. 5.4]{zhu2022nearlytightuniversalbounds},  we have 
	\[
	\Prb{\mathrm{Bin}(n,p) = k} \le  \frac{\sqrt{n}}{\sqrt{2 \pi k(n-k)}} \exp\left( - n D(k/n \,\|\,  p) \right)
	\]
	with, for all $0<x<1$, 
	\[
	D(x \,\|\, p) = x \log\frac{x}{p} + (1-x) \log\frac{1-x}{1-p}.
	\]
	By Pinsker's inequality for Bernoulli distributions,
	\[
	D(x \,\|\, p) \ge  2 (x-p)^2.
	\]
	Hence we arrive at
	\[
	\Prb{\mathrm{Bin}(n,p) = k} \le  n^{-1/2} \frac{1}{\sqrt{ 2 \pi (k/n)(1 - k/n) }} \exp\left( - 2 \frac{(k- n p)^2}{n } \right).
	\]
	For $k/n$ bounded away from $0$ and $1$ we have 
	\[
	\frac{1}{\sqrt{ 2 \pi (k/n)(1 - k/n) }} = O(1).
	\]
	For $k/n$  bounded away from $p$ we have
	\[
	\frac{1}{\sqrt{ 2 \pi (k/n)(1 - k/n) }} \exp\left( - 2 \frac{(k- n p)^2}{n } \right) = \exp(-\Theta(n)) = \exp(- \Theta(n)(k/n -p)^2).
	\]
	This completes the proof.
\end{proof}

The following medium deviation inequality for random walks with light-tailed step distribution is stated in \cite[Example 1.4]{RS15}:

\begin{proposition}
	\label{pro:deviation}
	Let $(W_i)_{i \in \ndN}$ be an i.i.d. family of real-valued random variables with $\Ex{W_1} = 0$ and $\Ex{e^{t W_1}} < \infty$ for all $t$ in some open interval containing zero. Then there are constants $\delta, c>0$ such that for all $n\in \ndN$, $x > 0$ and $0 \le\lambda\le\delta$ it holds that \[\Pr{|W_1 + \ldots + W_n| \ge x} \le 2 \exp(c n \lambda^2 - \lambda x).\]
\end{proposition}

\subsubsection{Typed population asymptotics}

Recall that $\mu = \frac{\cD(\rho)}{1 + \cD(\rho)}$. We define several auxiliary constants that we are going to use in the analysis of $\mT$.

\begin{proposition}
	The following constants are positive and finite.
\begin{align*}
	\alpha &= \frac{3 \mu}{ 1 - \Ex{L_0} + \mu \Ex{Z} + \frac{3}{2}\mu \Ex{X_0}} = \frac{3 \mu}{1 - \Ex{L} + \frac{3}{2}\mu \Ex{X}}, \\
	\beta &= \alpha \Exb{Y} + 3 (1+\mu) \left(1 - \frac{\alpha}{2} \Exb{X} \right), \\
	\eta &= 3 -  \alpha ((3/2)\Ex{X_0} + \Ex{Z}), \\
	\nu^2 &= \Vab{\frac{3}{2}\mu X_0 + \mu Z - L_0}, \\
	\sigma^2 &= \mu(1-\mu) = \frac{\cD(\rho)}{(1 + \cD(\rho))^2}. \\
	\chi &= \frac{2(1 -\Ex{L})}{1 - \Ex{L} + \frac{3}{2}\mu \Ex{X}}.
\end{align*}
\end{proposition}
\begin{proof}
By Lemma~\ref{le:le3} and Proposition~\ref{pro:expect}  we have $\Ex{L}<1$ and $\Ex{L_0}<1$ (since $\Ex{L} = \Ex{L_0} + \Ex{Z} \Ex{\bar{L}}$). Hence $\alpha>0$ is well-defined. The equality for $\alpha$  follows from Proposition~\ref{pro:expect}. From $\Ex{L}<1$ it also follows that $\frac{\alpha}{2} \Ex{X}<1$, hence both summands in $\beta>0$ are positive. Likewise, it follows that $\eta>0$. The variance $\nu^2$ is finite and non-zero because $(X_0, L_0, Z)$ have finite exponential moments by~Lemma~\ref{le:le1}, and $\frac{3}{2}\mu X_0 + \mu Z - L_0$ is not degenerate. It is clear that $\sigma^2$ and $\chi$ are non-zero.
\end{proof}

We are going to determine the probability for the event $\#_\cX\mT= 2n$. Conditionally on that event we verify that the subpopulation $\#_\cD\mT$ concentrates around $\alpha n$, and that the subpopulation $\#_\cY\mT$ concentrates around $\beta n$. It is important to note that $\#_\cY\mT$ does \emph{not} concentrate at $\alpha \Ex{Y}n$: The local large deviation asymptotics by~\cite{MR2440928} which we will apply describe a single big jump scenario. Hence we get the contribution $\alpha \Ex{Y}n$ from non-maximal jumps and an additional contribution $3 (1+\mu) \left(1 - \frac{\alpha}{2} \Exb{X}\right)n $ from the unique maximal jump.

\begin{theorem}
	\label{te:le4}
	As $n \to \infty$
	\[
		\Pr{ \#_{\cX} \mT = 2n } \sim \frac{8 \sqrt{6}}{5 \sqrt{\pi}} \Exb{Z} 
		\frac{(1 - \Exb{L} + \frac{3}{2} \mu \Exb{X})^{3/2}}{(1-\Exb{L})^{5/2}} n^{-5/2}.
	\]
	Uniformly for all integers $\ell$ with $| \ell - \alpha n| = O(\sqrt{n})$
	\begin{multline*}
		\Pr{ (\#_{\cX} \mT, \#_{\cD} \mT ) = (2n,\ell) } \\
		\sim  \Ex{Z} \frac{8 \sqrt{3} (1 - \Ex{L} + \frac{3}{2}\mu \Ex{X})^{5/2}}{5 \pi\sqrt{\sigma^2 \eta + \nu^2 \alpha} (1 - \Ex{L})^{5/2} }   n^{-3} 
		\exp\left( - \frac{((3\mu/\alpha) (\ell - \alpha n))^2} {2(\sigma^2  \eta +  \nu^2 \alpha)n} \right).
	\end{multline*}
 	Furthermore, 
	\begin{align*}
		\Pr{ |\#_{\cD} \mT  - \alpha n| \ge \sqrt{n} \log n \mid \#_\cX \mT = 2n } \le n^{-\Theta(\log n)}, \\
		\Pr{ |\#_{\cY} \mT  - \beta n| \ge \sqrt{n} \log n \mid \#_\cX \mT = 2n } \le n^{-\Theta(\log n)}.
	\end{align*}
\end{theorem}
\begin{proof}
	We may view $\mT$ as the result of starting with a single vertex of type $\cD$, and then in each step choose a leaf $v$ of  type $\cD$ of the current tree at random, take a fresh independent copy of $\mS$ and glue it to $v$ by identifying its root vertex with $v$. Each step we occupy a pre-existing ``free'' leaf of type $\cD$, and generate a nonnegative number of new free type $\cD$ leaves. This sequence of numbers is subject to the conditions of the outdegree sequence of a tree in an explorative order (for example via depth-first search). That is, with $d_i$ denoting the number of type $\cD$ leaves of the $i$th copy of $\mS$ we use, we have at step $k$ that $\sum_{i=1}^k(d_i -1) \ge -1$ and if at any step the sum assumes the minimal value $-1$ then there are no more free leaves of type $\cD$ in the current tree and the process stops. The cycle lemma~\cite{MR0138139} states that for any sequence of integers $x_1, \ldots, x_k \ge -1$ with $\sum_{i=1}^k x_i  = -1$ there exists a unique cyclic shift $\upsilon: [k] \to [k]$ such that $\sum_{i=1}^{k'} x_{\upsilon(i)}  \ge 0$ for all $k'<  k$.
	
	Let $1 \le \ell \le 5n$ be given. With $(X_i, Y_i, L_i)_{i \ge 1}$ denoting independent copies of $(X,Y,L)$, it follows that
	\begin{align}
		\label{eq:mot}
		\Pr{ (\#_{\cX} \mT, \#_{\cD} \mT ) = (2n,\ell) }  = \frac{1}{\ell} \Prb{ \sum_{i=1}^\ell (X_i, L_i) = (2n,\ell-1)}. 
	\end{align}
	With $(\bar{X}_i, \bar{Y}_i, \bar{L}_i)_{i \ge 1}$ denoting independent copies of $(\bar{X}, \bar{Y}, \bar{L})$ that are also independent of  $\mU$, we have by construction
	\begin{align}
		(X,L) \eqdist (X_0, L_0) + \sum_{i=1}^Z (\bar{X}_i, \bar{L}_i).
	\end{align}
	With $(X_0(j), L_0(j), Z(j))_{j \ge 1}$ denoting independent copies of $(X_0, L_0, Z)$ and setting $S_\ell^{X_0} = \sum_{j=1}^\ell X_0(j)$, $S_\ell^{L_0} = \sum_{j=1}^\ell L_0(j)$ and $S_\ell^{Z} = \sum_{j=1}^\ell Z(j)$ it follows that
	\begin{align}
		\label{eq:myw}
		\sum_{i=1}^\ell (X_i, L_i) \eqdist (S_\ell^{X_0}, S_\ell^{L_0}) + \sum_{i=1}^{S_\ell^Z} (\bar{X}_i, \bar{L}_i).
	\end{align}
	By~\eqref{eq:enumsimple} and Lemma~\ref{le:keynumeric} it follows that
	\begin{align}
		\label{pro:barxdensity}
		\Prb{\bar{X}=2n} &= \frac{256}{5}  q_n \left(\frac{27}{256}\right)^n \\
		&= \frac{8 \sqrt{6}}{5 \sqrt{\pi}} n^{-5/2} (1 + O(1/n)). \nonumber
	\end{align}
	By  the definition of $\phi_{\cZ}$ we have
	\begin{align}
		\label{eq:sona}
		\Prb{ (\bar{X}, \bar{L}) = (2n, d) } = \Prb{\bar{X}=2n} \Prb{\mathrm{Bin}\left(3n-1, \mu\right) = d}.
	\end{align}
	Note that for $k \ge 1$ and integers $n_1, \ldots, n_k \ge 1$ with $\sum_{i=1}^k n_i = n$ we have conditional on $\bar{X_i} = 2n_i$, $i=1, \ldots, k$ that $\bar{L}_1, \ldots, \bar{L}_k$ are conditionally independent with conditional distribution 
	\[
	(\bar{L}_i \mid (\bar{X}_i)_{1 \le i \le k} = (2n_i)_{1 \le i \le k} ) \eqdist \mathrm{Bin}(3n_i-1, \mu),
	\]
	yielding
	\begin{align}
		\label{eq:thattoo}
		\left(\sum_{i=1}^k \bar{L}_i \,\Bigg \vert\, (\bar{X}_i)_{1 \le i \le k} = (2n_i)_{1 \le i \le k} \right) \eqdist \mathrm{Bin}(3n-k, \mu).
	\end{align}
	Setting $S_k^{\bar{X}} = \sum_{i=1}^k \bar{X}_i$, it follows by~\eqref{eq:myw} and~\eqref{eq:thattoo} that
	\begin{align}
		\label{eq:til1}
		&\Prb{ \sum_{i=1}^\ell (X_i, L_i) = (2n,\ell-1)} \\
		&= \sum_{a,b,k \ge 0} \Prb{ (S_\ell^{X_0}, S_\ell^{L_0}, S_\ell^Z) = (2a,b,k)} \Prb{\sum_{i=1}^k (\bar{X}_i, \bar{L}_i) = (2(n-a), \ell-b-1) } \nonumber\\
		&= \sum_{a,b,k \ge 0} \Prb{ (S_\ell^{X_0}, S_\ell^{L_0}, S_\ell^Z) = (2a,b,k)} \cdot \nonumber\\
		&\qquad \qquad \,\Prb{S_k^{\bar{X}} = 2(n-a)}  \Prb{\mathrm{Bin}(3(n-a)-k, \mu) = \ell-b-1}. \nonumber
	\end{align}
	By Proposition~\ref{pro:binomial} there exist $C,c>0$ such that for all $0 \le k < 3(n-a)$ and all $\ell \ge 0$ we have
	\begin{multline}
		\label{eq:til2}
	\Prb{\mathrm{Bin}(3(n-a)-k,\mu) = \ell-b-1} \\ \le  \frac{C}{\sqrt{3(n-a)-k}} \exp\left( - c \frac{((\ell-b-1)- (3(n-a)-k) \mu)^2}{3(n-a)-k } \right).
	\end{multline}
	By Proposition~\ref{pro:deviation}, it follows that there exists a constant $c_0>0$ such that for all sufficiently small $\epsilon>0$, all $n \ge 1$  and all $1 \le \ell \le 5 n$ we have
	\begin{align}
		\label{eq:conen}
		\Prb{|S_\ell^{X_0} - \Ex{X_0}\ell| \ge  \epsilon n } &\le 2 \exp(-c_0 \epsilon^2 n), \\
		\Prb{ |S_\ell^{L_0} - \Ex{L_0}\ell| \ge \epsilon n } &\le 2 \exp(-c_0 \epsilon^2 n), \nonumber \\
		\Prb{ |S_\ell^{Z} - \Ex{Z}\ell| \ge  \epsilon n } &\le 2 \exp(-c_0 \epsilon^2 n). \nonumber
	\end{align}
	(With the notation $\delta, c, \lambda$ used in  Proposition~\ref{pro:deviation} we set $\lambda= \epsilon/ (10 c)$, such that $c \ell \lambda^2 - \lambda \epsilon n \le c 5n \lambda^2 - \lambda \epsilon n = -n  \epsilon^2  / (20c)$. Since $c$ is fixed, we have $\lambda < \delta$ for all sufficiently small $\epsilon$.)
	
	We emphasise that $O(\cdot)$-terms are allowed to be negative. For $2a = \Ex{X_0}\ell + O(\epsilon n)$, $b =  \Ex{L_0}\ell + O(\epsilon n)$ and $k = \Ex{Z} \ell + O(\epsilon n)$ we have uniformly in $n$ and $\epsilon>0$
	\begin{align*}
		(\ell-b-1)- (3(n-a)-k) \mu &= \ell \left(1 - \Ex{L_0} + \frac{3}{2} \mu \Ex{X_0} + \mu \Ex{Z} \right) - 3 \mu n -1 + O(\epsilon n) \\
		&= 3 \mu (\ell / \alpha -  n) + O(\epsilon n) - 1
	\end{align*}
	Let $C_0>0$ denote a constant so that the $O(\epsilon n)$ term in the preceding equation is bounded in absolute value by $C_0 \epsilon n$. It follows using~Proposition~\ref{pro:linearbound} and by~\eqref{eq:til1} and~\eqref{eq:til2} that for all $n \ge 1$,  $\delta>2/n$ and 
	\[
	0<\epsilon_0 \le \frac{3 \mu}{2 \alpha C_0} \delta
	\] we have by~\eqref{eq:mot}
	\begin{align}
		\label{eq:mw}
		&\Pr{ \#_{\cX} \mT = 2n, |\#_{\cD} \mT - n \alpha| \ge \delta n} \\
		&\qquad= \sum_{\substack{1 \le \ell \le 5n \\ |\ell - n \alpha| \ge \delta n}} \frac{1}{\ell} \Prb{ \sum_{i=1}^\ell (X_i, L_i) = (2n,\ell-1)}  \nonumber \\
		&\qquad\le O(n)  \exp(-c_0 \epsilon_0 ^2 n) \nonumber \\
		&\qquad\quad+ \sum_{\substack{1 \le \ell \le 5n \\ |\ell - n \alpha| \ge \delta n}} \sum_{\substack{a,b,k \ge 0 \\ 
		|2a - \Ex{X_0} \ell| \le \epsilon_0  n \\
		|b - \Ex{L_0} \ell| \le \epsilon_0  n \\
		|k - \Ex{Z} \ell| \le \epsilon_0  n}}\exp\left( - c \frac{((\ell-b-1)- (3(n-a)-k) \mu)^2}{3n } \right) \nonumber \\
		&\qquad\le O(n) \exp(-c_0 \epsilon_0 ^2 n) + O(n^4) \exp\left(-c \frac{(3 \mu \delta n / \alpha - C_0 \epsilon_0 n - 1)^2}{3n} \right). \nonumber  \\
		&\qquad\le O(n) \exp(-c_0 \epsilon_0 ^2 n) + O(n^4) \exp\left(- \frac{3 c \mu^2}{4 \alpha^2}n\delta^2 \right). \nonumber
	\end{align}
	By Proposition~\ref{pro:expect} and Lemma~\ref{le:le3} we have
	\[
		\Ex{L_0} < \Ex{L} < 1,
	\]
	hence
	\begin{align*}
		\alpha \left( \frac{3}{2} \Ex{X_0} + \Ex{Z}\right) = \frac{3 \mu(\Ex{Z} + \frac{3}{2} \Ex{X_0})}{ 1 - \Ex{L_0} + \mu (\Ex{Z} + \frac{3}{2} \Ex{X_0})} < 3.
	\end{align*}
	Hence
	\begin{align}
		3(n-a)-k = n \left(3 - \alpha \left( \frac{3}{2} \Ex{X_0} + \Ex{Z}\right) \right) + O(\delta n) .
 	\end{align}
	concentrates around a positive multiple of $n$, yielding
	\[
		\frac{1}{\sqrt{3(n-a)-k}} = O\left(\frac{1}{\sqrt{n}}\right)
	\]
	uniformly for $\delta$ below a fixed threshold.
	Note that by Proposition~\ref{pro:expect} and Lemma~\ref{le:le3}
	\begin{align*}
		2(n-a) - k \Ex{\bar{X}} &= 2n - \ell (\Ex{X_0} + \Ex{Z}\Ex{\bar{X}})+ O(\epsilon_0 n) \\
		&= 2n - \alpha n\Ex{X} + O(\delta n)  \\
		&= 2n \left(1 - \frac{\frac{3}{2} \mu\Ex{X}}{1 - \Ex{L} + \frac{3}{2}\mu \Ex{X}} \right) + O(\delta n) 
	\end{align*}
	concentrates around a positive multiple of $n$.
	Hence  it follows by the local large deviation result \cite[Cor. 2.1]{MR2440928} that uniformly for $\delta$  below a fixed threshold
	\begin{align}
		\label{eq:mw2}
		&\Prb{S_k^{\bar{X}} = 2(n-a)} \\
		&\sim k \Prb{\frac{\bar{X}}{2} = \lfloor (n-a) -k \Ex{\bar{X}}/2 \rfloor } \nonumber \\
		&\sim (\ell \Ex{Z} + O(\epsilon_0 n) ) \frac{8 \sqrt{6}}{5 \sqrt{\pi}} \left(n \left(\frac{1 - \Ex{L} }{1 - \Ex{L} + \frac{3}{2}\mu \Ex{X}} \right) + O(\delta n) \right)^{-5/2}. \nonumber
	\end{align}
	By~\eqref{eq:mw} it follows that
	\begin{align}
		\label{eq:toget}
		\Pr{ \#_{\cX} \mT = 2n, |\#_{\cD} \mT - n \alpha| \ge \sqrt{n} \log n} = n^{- \Theta(\log n)}.
	\end{align}
	For brevity, let us set
	\begin{align*}
		\Delta_\ell &= \ell - \alpha n, \\
		B &= B(n, a, b, k, \ell) = \Prb{\mathrm{Bin}(3(n-a)-k, \mu) = \ell-b-1}.
	\end{align*}
	Suppose that $\ell \ge 1$ satisfies 
	\[
	|\Delta_\ell| \le  \sqrt{n}\log n.
	\]
	By~\eqref{eq:mot},~\eqref{eq:til1} and \eqref{eq:conen} it follows that 
	\begin{align}
		\label{eq:j0}
		&\Pr{ (\#_{\cX} \mT, \#_{\cD} \mT ) = (2n,\ell) }  \\
		 &= \frac{1}{\ell}  \sum_{a,b,k \ge 0} \Prb{ (S_\ell^{X_0}, S_\ell^{L_0}, S_\ell^Z) = (2a,b,k)} \Prb{S_k^{\bar{X}} = 2(n-a)}  B. \nonumber \\
		&= n^{-\Theta(\log n)} + \frac{1}{\ell}  \quad \sum_{ \mathclap{ \substack{a,b,k \ge 0 \\ 
				|a - (1/2)\Ex{X_0} \ell| \le \sqrt{n} \log n \\
				|b - \Ex{L_0} \ell| \le \sqrt{n} \log n \\
				|k - \Ex{Z} \ell| \le \sqrt{n} \log n}}} \quad \Prb{ (S_\ell^{X_0}, S_\ell^{L_0}, S_\ell^Z) = (2a,b,k)} \Prb{S_k^{\bar{X}} = 2(n-a)}  B. \nonumber
	\end{align}
	By~\eqref{eq:mw2} it follows that uniformly for 
	\begin{align*}
		a &= (1/2)\Ex{X_0} \ell + x_a \\
		b &=  \Ex{L_0} \ell + x_b \\
		k &= \Ex{Z} \ell + x_k 
	\end{align*}
	with $|x_a|, |x_b|, |x_k| \le \sqrt{n} \log n$
	\begin{align}
		\label{eq:jump}
		\frac{1}{\ell} \Prb{S_k^{\bar{X}} = 2(n-a)} \sim  \Ex{Z} \frac{8 \sqrt{6}}{5 \sqrt{\pi}}  \left(\frac{1 - \Ex{L} }{1 - \Ex{L} + \frac{3}{2}\mu \Ex{X}}  \right)^{-5/2} n^{-5/2}.
	\end{align}
	Moreover,
	\begin{align*}
		2\pi(3(n-a)-k)\sigma^2	&= 2\pi \sigma^2(3(n-(1/2)\Ex{X_0} \ell + x_a) - (\Ex{Z}\ell + x_k)) \\
								&= 2\pi \sigma^2n (3 -  \alpha ((3/2)\Ex{X_0} + \Ex{Z}))   + O(\sqrt{n} \log n).
	\end{align*}
	and by Lemma~\ref{le:le3} and Proposition~\ref{pro:expect} 
	\[
		\alpha ((3/2)\Ex{X_0} + \Ex{Z})  < 3.
	\]
	Moreover, 
	\begin{align*}
		&(3(n-a)-k)\mu  - (\ell - b -1 ) \\
		&= 3\mu n - \mu ((3/2)\Ex{X_0} +\Ex{Z})\ell - 3 \mu x_a -  \mu x_k - \ell(1- \Ex{L_0})  + x_b + 1 \\
		&=  - (3 \mu / \alpha)\Delta_\ell - 3 \mu x_a -  \mu x_k + x_b + 1
	\end{align*}
	and hence
	\begin{align*}
		\frac{((3(n-a)-k)\mu  - (\ell - b -1 ))^2}{2(3(n-a)-k)\sigma^2  } &= \frac{((3\mu/\alpha)\Delta_\ell + 3 \mu x_a +  \mu x_k - x_b )^2}{ 2 \sigma^2n (3 -  \alpha ((3/2)\Ex{X_0} + \Ex{Z}))} + o(1).
	\end{align*}
	By the local limit theorem for binomial distributions~\cite[Thm. 1.1]{szewczak2022classicalsurelocallimit} (which is valid for deviations from the mean up to any given sequence $o(n^{1/2 + 1/6})$, and in our case the deviation is $O(\sqrt{n} \log n)$ uniformly for all considered $\ell$, $a$, $b$ and $k$) it follows that
	\begin{align}
		\label{eq:audio}
		B &= \frac{1 + o(1)}{\sqrt{2\pi \sigma^2n \eta }} \exp\left(- \frac{((3\mu/\alpha)\Delta_\ell + 3 \mu x_a +  \mu x_k - x_b )^2}{ 2 \sigma^2n \eta }  \right).
	\end{align}
	Set
	\begin{align*}
		W_\ell &= \ell^{-1/2}\left(\frac{3}{2}\mu S_\ell^{X_0} + \mu S_\ell^{Z} - S_\ell^{L_0} - \ell \left(\frac{3}{2}\mu \Ex{X_0} + \mu \Ex{Z} - \Ex{L_0}\right)\right), \\
		f_{n,\ell}(x) &=  \exp\left(- \frac{((3\mu/\alpha)\Delta_\ell + x \sqrt{\ell} )^2}{ 2 \sigma^2n \eta }  \right)  .
	\end{align*}
	This way, $\Va{W_1} = \nu^2$. Since $f_{n,\ell}$ is bounded, it follows by~\eqref{eq:conen} and~\eqref{eq:audio} that
	\begin{multline*}
		\qquad\,\,\,\,\sum_{ \mathclap{ \substack{a,b,k \ge 0 \\ 
					|a - (1/2)\Ex{X_0} \ell| \le \sqrt{n} \log n \\
					|b - \Ex{L_0} \ell| \le \sqrt{n} \log n \\
					|k - \Ex{Z} \ell| \le \sqrt{n} \log n}}} \quad \Prb{ (S_\ell^{X_0}, S_\ell^{L_0}, S_\ell^Z) = (2a,b,k)}   B = n^{-\Theta(\log n)} + \frac{(1+o(1))}{\sqrt{2\pi \sigma^2n \eta }}\Exb{f_{n,\ell}(W_\ell )  }.
	\end{multline*}
	By~\eqref{eq:j0} and~\eqref{eq:jump} we arrive at
	\begin{multline}
		\label{eq:ep1}
		\Pr{ (\#_{\cX} \mT, \#_{\cD} \mT ) = (2n,\ell) }  \\ = n^{-\Theta(\log n)} + \Ex{Z} \frac{8 \sqrt{6}}{5 \sqrt{\pi}}  \left(\frac{1 - \Ex{L} }{1 - \Ex{L} + \frac{3}{2}\mu \Ex{X}}  \right)^{-5/2} n^{-5/2} 
		 \frac{(1+o(1))}{\sqrt{2\pi \sigma^2n \eta }}\Exb{f_{n,\ell}(W_\ell )  }
	\end{multline}
	uniformly for $|\Delta_\ell| \le \sqrt{n} \log n$. By Proposition~\ref{pro:deviation}, 
	\[
		\Prb{|W_\ell| \sqrt{\ell} \ge (3\mu/\alpha)|\Delta_\ell| /2} \le 2\exp(-\Theta(|\Delta_\ell^2|/n)).
	\]
	In the event that $|W_\ell|\sqrt{\ell} \le (3\mu/\alpha)|\Delta_\ell| /2$ we have
	\[
	f_{n,\ell}(W_\ell) =  \exp\left(- \frac{((3\mu/\alpha)\Delta_\ell + W_\ell \sqrt{\ell} )^2}{ 2 \sigma^2n \eta }  \right)  \le \exp(-\Theta(|\Delta_\ell|^2/n)).
	\]
	Since $f_{n,\ell}$ is bounded, it follows that
	\begin{align*}
	\Exb{f_{n,\ell}(W_\ell )  } \le  O(1) \exp(-\Theta(|\Delta_\ell|^2/n)).
	\end{align*}
	It follows that
	\begin{align}
		\Pr{ (\#_{\cX} \mT, \#_{\cD} \mT ) = (2n,\ell) } &\le O(n^{-3}) \exp(-\Theta(|\Delta_\ell|^2/n))
	\end{align}
	uniformly for $|\Delta_\ell| \le \sqrt{n} \log n$.  
	Together with~\eqref{eq:toget} it follows by Riemann summation that there exists a function $g: [0, \infty[ \to [0, \infty[$ with $\lim_{x \to \infty} g(x)=0$ such that for all $n \ge 1$ and $K \ge 1$
	\begin{align}
		\label{eq:emi2}
		\Pr{ \#_{\cX} \mT = 2n, |\#_{\cD} \mT - \alpha n| \ge K \sqrt{n} } &\le n^{-5/2} g(K).
	\end{align}
	We now consider the case
	\[
		|\Delta_\ell| \le K \sqrt{n}
	\]
	for $K>0$ constant.
	By an elementary substitution 
	\begin{align*}
		\int_{\ndR} |f_{n,\ell}'(x)|\,\mathrm{d}x &= 2.
	\end{align*}
	Hence, with $N$ denoting a centred normal distribution with variance $ \nu^2$, it follows by the Berry--Esseen inequality
	\begin{align}
		\label{eq:ep2}
			|\Exb{f_{n,\ell}(W_\ell )  } - \Exb{f_{n,\ell}(N)}| &\le  \int_{\ndR}| (\Pr{ W_\ell \ge x} - \Pr{N \ge x}) f_{n,\ell}'(x)| \,\mathrm{d}x \\
			&\le 2 \sup_{x \in \ndR} | \Pr{ W_\ell \ge x} - \Pr{N \ge x}| \nonumber\\
			&= O(1/\sqrt{\ell})\nonumber \\
			&= O(1/\sqrt{n}). \nonumber
	\end{align}
	For any $\tau>0$ let $\phi_{\tau^2}(x) = (2 \pi \tau^2)^{-1/2} \exp(-x^2 / (2 \tau^2))$ denote the density of a centred normal distribution with variance $\tau^2$. Since $\sqrt{\ell}N \eqdist \cN(0,   \nu^2\ell)$, it follows that
	\begin{align}
		\label{eq:ep3}
		\frac{1}{\sqrt{2\pi \sigma^2n \eta }}\Exb{f_{n,\ell}(N)  } &= \int_\ndR \phi_{\sigma^2 n \eta}( (3\mu/\alpha) \Delta_\ell - x) \phi_{ \nu^2 \ell}(x)\, \mathrm{d}x \\
		&= \phi_{\sigma^2 n \eta +  \nu^2 \ell}((3\mu/\alpha) \Delta_\ell) \nonumber \\
		&\sim \phi_{(\sigma^2  \eta +  \nu^2 \alpha )n}((3\mu/\alpha) \Delta_\ell). \nonumber
	\end{align}
	Since $\Delta_\ell = O(\sqrt{n})$ we have that $\sqrt{n} \phi_{(\sigma^2  \eta +  \nu^2 \alpha )n}((3\mu/\alpha) \Delta_\ell)$ is bounded away from zero. By~\eqref{eq:ep1},~\eqref{eq:ep2} and~\eqref{eq:ep3} it follows that
	\begin{align}
		\label{eq:od}
			&\Pr{ (\#_{\cX} \mT, \#_{\cD} \mT ) = (2n,\ell) } \\
			&\sim \Ex{Z} \frac{8 \sqrt{6}}{5 \sqrt{\pi}}  \left(\frac{1 - \Ex{L} }{1 - \Ex{L} + \frac{3}{2}\mu \Ex{X}}  \right)^{-5/2} n^{-5/2} 
			  \phi_{(\sigma^2  \eta +  \nu^2 \alpha )n}((3\mu/\alpha) \Delta_\ell) \nonumber
	\end{align}
	By Riemann summation
	\begin{align*}
		\sum_{\ell : |\Delta_\ell| \le K \sqrt{n}} \phi_{(\sigma^2  \eta +  \nu^2 \alpha )n}((3\mu/\alpha) \Delta_\ell) &\sim \int_{-K}^K \phi_{\sigma^2  \eta +  \nu^2 \alpha }((3\mu/\alpha) x) \,\,\mathrm{d}x \\
		&=  \frac{\alpha}{3 \mu} \int_{- (3\mu/\alpha) K}^{(3\mu/\alpha)  K} \phi_{\sigma^2  \eta +  \nu^2 \alpha }(x) \,\,\mathrm{d}x.
	\end{align*}
	This tends to $\alpha / (3\mu)$ as $K \to \infty$. Together with~\eqref{eq:toget},~\eqref{eq:emi2} and~\eqref{eq:od} it follows that
	\begin{align}
		\label{eq:combine}
		\Pr{ \#_{\cX} \mT  = 2n } 
		&\sim \Ex{Z} \frac{8 \sqrt{6}}{5 \sqrt{\pi}}  \left(\frac{1 - \Ex{L} }{1 - \Ex{L} + \frac{3}{2}\mu \Ex{X}}  \right)^{-5/2} \frac{\alpha}{3\mu} n^{-5/2}  \\
		&= \Ex{Z} \frac{8 \sqrt{6}}{5 \sqrt{\pi}}  \frac{(1 - \Ex{L} + \frac{3}{2}\mu \Ex{X})^{3/2}}{ (1 - \Ex{L})^{5/2} }  n^{-5/2}. \nonumber
	\end{align}
	Equation~\eqref{eq:od} may be further simplified to
	\begin{align*}
		&\Pr{ (\#_{\cX} \mT, \#_{\cD} \mT ) = (2n,\ell) } \\
		&\sim \Ex{Z} \frac{8 \sqrt{3} (1 - \Ex{L} + \frac{3}{2}\mu \Ex{X})^{5/2}}{5 \pi\sqrt{\sigma^2 \eta + \nu^2 \alpha} (1 - \Ex{L})^{5/2} }   n^{-3} 
		\exp\left( - \frac{((3\mu/\alpha) \Delta_\ell)^2} {2(\sigma^2  \eta +  \nu^2 \alpha)n} \right). 
	\end{align*}
	Combining~\eqref{eq:combine} with~\eqref{eq:toget} yields
	\begin{align}
		\label{eq:quarter}
		\Pr{ |\#_{\cD} \mT  - \alpha n| \ge \sqrt{n} \log n \mid \#_\cX \mT = 2n } \le n^{-\Theta(\log n)}
	\end{align}
	It remains to verify concentration of $\#_\cY \mT$ around $\beta n$ when conditioning on the event $\#_\cX\mT = 2n$.
	Analogous to~\eqref{eq:mot} we have
	\begin{align}
		\label{eq:prece0}
		\Pr{ (\#_{\cX} \mT, \#_{\cY} \mT, \#_{\cD} \mT ) = (2n, m, \ell) }  = \frac{1}{\ell} \Prb{ \sum_{i=1}^\ell (X_i, Y_i, L_i) = (2n,m, \ell-1)}. 
	\end{align}
	Moreover, analogous to~\eqref{eq:myw} we have
	\begin{align}
		\label{eq:prece}
		\sum_{i=1}^\ell (X_i, Y_i, L_i) \eqdist (S_\ell^{X_0},S_\ell^{Y_0}, S_\ell^{L_0}) + \sum_{j=1}^{S_\ell^Z} (\bar{X}_j, \bar{Y}_j, \bar{L}_j).
	\end{align}
	By~\eqref{eq:bary} we have $\bar{Y}=\frac{3}{2}\bar{X} -1 + \bar{L}$. Hence, on the event that the first coordinate in~\eqref{eq:prece} is equal to $2n$, and the third is equal to $\ell-1$, we have
	\begin{align*}
		S_\ell^{Y_0} + \sum_{j=1}^{S_\ell^Z} \bar{Y}_j &= S_\ell^{Y_0} + \sum_{j=1}^{S_\ell^Z}\left( \frac{3}{2}\bar{X}_j - 1 + \bar{L}_j \right) \\
		&= S_\ell^{Y_0} + \frac{3}{2}(2n - S_\ell^{X_0}) - S_\ell^Z + (\ell -1 - S_\ell^{L_0}) \\
		&= 3n + (\ell -1) + S_\ell^{Y_0} - \frac{3}{2} S_\ell^{X_0} - S_\ell^Z - S_\ell^{L_0} \\
		&=  3n + (\ell -1) + S_\ell^{V}
	\end{align*}
	for $V(i) = Y_0(i) - \frac{3}{2} X_0(i) - Z(i) - L_0(i)$  and $S_\ell^V = \sum_{i=1}^\ell V(i)$. Using~\eqref{eq:bary} and Proposition~\ref{pro:expect} we obtain
	\begin{align*}
		&\beta - (3 + \alpha(1 + \Ex{V(1)})) \\
		&= \alpha \Exb{Y} + 3 (1+\mu) \left(1 - \frac{\alpha}{2} \Exb{X} \right) - (3 + \alpha(1 + \Ex{V(1)})) \\
		&= \alpha \left(\Ex{Y_0} + \Ex{Z}\left(\frac{3}{2}\Ex{\bar{X}} - 1 + \Ex{\bar{L}}\right)\right) + 3 (1+\mu) \left(1 - \frac{\alpha}{2} \Exb{X} \right) \\
		&\quad\, - \left(3 + \alpha\left(1 + \Ex{Y_0} - \frac{3}{2} \Ex{X_0} - \Ex{Z} - \Ex{L_0}\right)\right) \\
		&= \alpha \left( \frac{3}{2} \Ex{X} -1  + \Ex{L} \right) + 3\left( -\frac{\alpha}{2} \Ex{X} - \frac{\alpha}{2} \mu \Ex{X} + \mu \right) \\
		&= 3 \mu - \alpha \left( \frac{3}{2} \mu \Ex{X} +1 - \Ex{L} \right) \\
		&= 0.
	\end{align*}
	It follows that
	\begin{align*}
		S_\ell^{Y_0} + \sum_{j=1}^{S_\ell^Z} \bar{Y}_j - n \beta  &= 3n + (\alpha n + \Delta_\ell -1) + S_\ell^V - n(3 + \alpha + \alpha\Ex{V(1)}) \\
		&= \Delta_\ell - 1 + S_\ell^V - n \alpha \Ex{V(1)}
	\end{align*}
	Set $E=1+|\Ex{V(1)}|$. Analogous to~\eqref{eq:quarter} (using only $E>0$) it follows that
	\[
			\Prb{ |\#_{\cD} \mT  - \alpha n| \ge \frac{1}{4 E}\sqrt{n} \log n \mid \#_\cX \mT = 2n } \le n^{-\Theta(\log n)}.
	\]
	Using~\eqref{eq:combine},~\eqref{eq:prece0} and~\eqref{eq:prece} and Proposition~\ref{pro:deviation}  it follows that
	\begin{align*}
		&\Pr{ |\#_{\cY} \mT  - \beta n| \ge \sqrt{n} \log n \mid \#_\cX \mT = 2n } \\
		&\le n^{-\Theta(\log n)} + O(n^{5/2}) \sum_{\ell : |\Delta_\ell| \le \frac{1}{4E} \sqrt{n} \log n} \frac{1}{\ell} \Prb{|\Delta_\ell - 1 + S_\ell^V - n \alpha \Ex{V(1)}| \ge \sqrt{n} \log n} \\
		&\le  n^{-\Theta(\log n)} + O(n^{3/2}) \sum_{\ell : |\Delta_\ell| \le \frac{1}{4E} \sqrt{n} \log n} \Prb{|S_\ell^V - \ell  \Ex{V(1)}| \ge \frac{1}{2}\sqrt{n} \log n} \\
		&\le n^{-\Theta(\log n)}.
	\end{align*}
	This completes the proof.
\end{proof}

By construction,
\[
\Exb{x^{\#_\cX \mT} y^{\#_\cY \mT}} = \frac{1}{\cD(\rho)} W_\cD^+(\rho x, (\rho x) ^2, \ldots; y, 1, 1, \ldots; 1, 1, \ldots).
\]
Moreover, $(1 + \#_\cY \mT \mid \#_\cX \mT= 2n)$ is distributed like the random number $A_{2n}^\cD$ of cylindrical $1$-edge-cycles in a uniformly at random selected element from $\Sym_{\cD}^{+}[2n]$. The constant summand $1$ accounts for the root edge.
Hence the following is a consequence of Theorem~\ref{te:le4}:

\begin{corollary}
	\label{co:piz}
As $n \to \infty$, 
\begin{align}
	\label{eq:bista}
[x^{2n}] \frac{\cD(\rho x)}{\cD(\rho)} \sim \frac{8 \sqrt{6}}{5 \sqrt{\pi}} \Exb{Z} 
\frac{(1 - \Exb{L} + \frac{3}{2} \mu \Exb{X})^{3/2}}{(1-\Exb{L})^{5/2}} n^{-5/2}
\end{align}
and for any constant $c>0$
\begin{align}
	\label{eq:cono}
	\Prb{ |A_{2n}^\cD - \beta n| \ge c \sqrt{n} \log n} \le n^{-\Theta(\log n)}.
\end{align}
\end{corollary}
The fact that we may multiply $\sqrt{n} \log n$ by any constant $c>0$ in the statement is immediate from the proof of Theorem~\ref{te:le4}.

\subsection{Extremal vertex degrees}

Let 
\[
J_1 \ge J_2 \ge \cdots
\]
denote the ranked points of a Poisson point process $\Xi$ on $]0,\infty[$ with intensity
\[
\frac{32\sqrt{3}}{5\sqrt{\pi}}\alpha\Ex{Z} x^{-5/2}\,\mathrm{d}x.
\]
Let $\mathfrak h$ denote the density of the  $3/2$-stable
law $S$ with Laplace transform
\[
\Exb{\exp(-\lambda S)}=\exp(\lambda^{3/2}), \qquad \lambda \ge 0.
\]
Set
\begin{align}
\omega = \left(\frac{32\sqrt{6}}{15}\right)^{2/3} (\alpha\Ex{Z})^{2/3}.
\end{align}
For any type $\cR$ we let $\Delta_{\cR, \cX}(\mT)$ denote the maximal number of type $\cX$ children of a type $\cR$ vertex in $\mT$.  We also let
\[
\Delta_{\cR, \cX}(\mT) = {}_1\Delta_{\cR, \cX}(\mT) \ge {}_2\Delta_{\cR, \cX}(\mT) \ge {}_3\Delta_{\cR, \cX}(\mT) \ldots
\]
denote the numbers of type $\cX$ children of type $\cR$ vertices in non-increasing order.

The next lemma will aid us in bounding the number of cubic planar graphs that admit an automorphism without cylindrical $1$-cycles of edges.

\begin{lemma}
	\label{le:maxdeg}
	We have
	\begin{align*}
		\Pr{ \Delta_{\cZ, \cX}(\mT) \le  n / \log^2 n \mid \#_\cX \mT = 2n } \le n^{-\Theta(\log n)}.
	\end{align*}
	Let $\mathfrak{E}$ denote the event that there exist (at least) two vertices in $\mT$ of type $\cZ$ with at least $n/ \log^2(n)$ children such that both receive $3$-connected cores as local decorations that are isomorphic as unrooted graphs. Then
	\[
		\Prb{\mathfrak{E} \mid \#_\cX \mT = 2n } \le \exp(- \Theta(n/\log^2 n)).
	\]
	For any  type $\cR \ne \cZ$ we have
	\begin{align*}
		\Pr{ \Delta_{\cR, \cX}(\mT) \ge \log^2 n \mid \#_\cX \mT = 2n } \le n^{-\Theta(\log n)}.
	\end{align*}
	We have uniformly for all even integers $h = \chi n + xn^{2/3}$  that
	\begin{align*}
		\omega n^{2/3} \Pr{ \Delta_{\cZ, \cX}(\mT) = h \mid \#_\cX(\mT) = 2n} = 
		\mathfrak h\left(-\frac{x}{2\omega}\right) + o(1).
	\end{align*}
	Moreover,
	\begin{align*}
		( n^{-2/3} {}_{j}\Delta_{\cZ, \cX}(\mT) \mid \#_{\cX} \mT = 2n)_{j \ge 2} \convd (J_i)_{i \ge 1}.
	\end{align*}
\end{lemma}
\begin{proof}
	Throughout this proof we will always implicitly assume that $h$ is even.
	We use the same notation as in the proof of Theorem~\ref{te:le4}. For
	\[
		|\Delta_\ell| \le \sqrt{n} \log n
	\]
	it follows analogously to~\eqref{eq:j0} that
	\begin{align*}
		&\Pr{ (\#_{\cX} \mT, \#_{\cD} \mT ) = (2n,\ell), \Delta_{\cZ, \cX}(\mT) \le  n / \log^2 n}  \\
	&\qquad\qquad= n^{-\Theta(\log n)} + \frac{1}{\ell}  \quad \sum_{ \mathclap{ \substack{a,b,k \ge 0 \\ 
				|a - (1/2)\Ex{X_0} \ell| \le \sqrt{n} \log n \\
				|b - \Ex{L_0} \ell| \le \sqrt{n} \log n \\
				|k - \Ex{Z} \ell| \le \sqrt{n} \log n}}} \quad \Prb{ (S_\ell^{X_0}, S_\ell^{L_0}, S_\ell^Z) = (2a,b,k)} \\ &\qquad\qquad\quad\,\,\,\Prb{S_k^{\bar{X}} = 2(n-a), \max_{1 \le i \le k} \bar{X}_i \le n/\log^2 n}  B. 
	\end{align*}
	In the sum, we have
	\begin{align*}
		2(n-a) - k \Ex{\bar{X}} &= 2n - \Ex{X_0}\ell - \Ex{Z}\Ex{\bar{X}}\ell + O(\sqrt{n} \log n) \\
		&= 2n - \Ex{X}\alpha n + O(\sqrt{n} \log n) \\
		&= \chi n + O(\sqrt{n} \log n) .
	\end{align*}
	Hence we may apply~\cite[Lem. 2.1]{MR2440928}, yielding that there exists $C>0$ with
	\begin{align*}
		\Prb{S_k^{\bar{X}} = 2(n-a), \max_{1 \le i \le k} \bar{X}_i \le n/\log^2 n} &\le C \exp\left(- \frac{2(n-a)}{n} \log^2 n \right) \\
		&\le n^{- \Theta(\log n)}.
	\end{align*}
	It follows that
	\[
		\Pr{ (\#_{\cX} \mT, \#_{\cD} \mT ) = (2n,\ell), \Delta_{\cZ, \cX}(\mT) \le  n / \log^2 n}  \le n^{- \Theta(\log n)}.
	\]
	By~\eqref{eq:toget} and~\eqref{eq:combine} this yields
	\begin{align}
		\label{eq:10am}
		\Pr{ \Delta_{\cZ, \cX}(\mT) \le  n / \log^2 n \mid \#_\cX \mT = 2n } \le n^{-\Theta(\log n)}.
	\end{align}
	
	Next, let $n / \log^2 n \le h \le 2n$. It follows analogously to~\eqref{eq:j0} that
	\begin{align}
		\label{eq:logo}
		&\Pr{ (\#_{\cX} \mT, \#_{\cD} \mT ) = (2n,\ell), \Delta_{\cZ, \cX}(\mT) =  h}  \\
		&\qquad\qquad= n^{-\Theta(\log n)} + \frac{1}{\ell}  \quad \sum_{ \mathclap{ \substack{a,b,k \ge 0 \\ 
					|a - (1/2)\Ex{X_0} \ell| \le \sqrt{n} \log n \\
					|b - \Ex{L_0} \ell| \le \sqrt{n} \log n \\
					|k - \Ex{Z} \ell| \le \sqrt{n} \log n}}} \quad \Prb{ (S_\ell^{X_0}, S_\ell^{L_0}, S_\ell^Z) = (2a,b,k)} \nonumber \\ &\qquad\qquad\quad\,\,\,\Prb{S_k^{\bar{X}} = 2(n-a), \max_{1 \le i \le k} \bar{X}_i = h}  B. \nonumber 
	\end{align}
	We claim that in the sum we have 
	\begin{align}
		\label{eq:toshow}
		 \Prb{S_k^{\bar{X}} = 2(n-a), \max_{1 \le i \le k} \bar{X}_i = h} = \Prb{S_k^{\bar{X}} = 2(n-a)} \frac{1}{\omega n^{2/3}} \left(\mathfrak{h}(-x/(2 \omega)) + o(1) \right)
	\end{align}
	with a uniform $o(1)$ term. 
	By~\eqref{pro:barxdensity} and  the local limit theorem (noting that $\bar{X}$ takes values in the even integers and in no proper sublattice of the even integers), 
	\begin{align}
		\label{eq:hllt}
		\sup_{t \in 2 \ndZ} \left| a_m \Prb{S_m^{\bar{X}} = t} - \mathfrak{h}((t - m\Ex{\bar{X}})/(2 a_m)) \right| = o(1)
	\end{align}
	for
	\[
		a_m = \left(\frac{32\sqrt{6}}{15}\right)^{2/3}  m^{2/3}.
	\]
	It is elementary that
	\begin{multline*}
		\Prb{S_k^{\bar{X}} = 2(n-a), \max_{1 \le i \le k} \bar{X}_i = h} \\= \sum_{1 \le j \le \min(2(n-a) / h,k)} \binom{k}{j} \Pr{\bar{X}= h}^j \Prb{S_{k-j}^{\bar{X}} = 2(n-a)-jh, \max_{1 \le i \le k-j} \bar{X}_i < h}.
	\end{multline*}
	In the sum we have $j = O(\log^2 n)$ and hence $k-j \sim k$.
	Since $\mathfrak{h}$ is bounded, it follows by~\eqref{eq:hllt} that
	\begin{align*}
		&\sum_{2 \le j \le \min(2(n-a) / h,k)} \binom{k}{j} \Pr{\bar{X}= h}^j \Prb{S_{k-j}^{\bar{X}} = 2(n-a)-jh, \max_{1 \le i \le k-j} \bar{X}_i < h} \\
		&\le \sum_{2 \le j \le \min(2(n-a) / h,k)} k^j h^{-5j/2}k^{-2/3} \\
		&= O(n^{-10/3}).
	\end{align*}
	Hence,
	\begin{multline}
		\label{eq:sunday}
		\Prb{S_k^{\bar{X}} = 2(n-a), \max_{1 \le i \le k} \bar{X}_i = h} \\
		=  O(n^{-10/3}) + k\Pr{\bar{X}=h} \Prb{S_{k-1}^{\bar{X}} = 2(n-a)-h, \max_{1 \le i \le k-1} \bar{X}_i < h}.
	\end{multline}
	Let $0< \delta < \chi$. If $n / \log^2 n \le h \le \delta n$, then
	\[
		2(n-a)- (k-1)\Ex{\bar{X}} -h = \Theta(n),
	\]
	yielding by~\cite[Cor. 2.1]{MR2440928}
	\begin{align*}
		\Prb{S_{k-1}^{\bar{X}} = 2(n-a)-h} &\sim (k-1) \Pr{\bar{X}= 2(n-a)-h - (k-1)\Ex{\bar{X}} } \\
		&= \Theta(n^{-3/2}).
	\end{align*}
	Hence by~\eqref{eq:sunday}
	\begin{align*}
		\Prb{S_k^{\bar{X}} = 2(n-a), \max_{1 \le i \le k} \bar{X}_i = h} 
		&\le O(n^{-10/3}) + O(n) (n/\log^2(n))^{-5/2} n^{-3/2}, \\
		&\le O( n^{-3} (\log n)^5).
	\end{align*}
	Note that by~\cite[Cor. 2.1]{MR2440928}
	\[
		\Prb{S_k^{\bar{X}} = 2(n-a)} \frac{1}{\omega n^{2/3}} = n \Theta(n^{-5/2}) n^{-2/3} = \Theta(n^{-13/6}).
	\]
	It follows that~\eqref{eq:toshow} holds uniformly for $n / \log^2 n \le h \le \delta n$. 
	
	Suppose now that $\delta n \le h \le 2n$. 
	\begin{align*}
		&\Prb{S_{k-1}^{\bar{X}} = 2(n-a)-h, \max_{1 \le i \le k-1} \bar{X}_i \ge h} \\
		&\quad\le (k-1) \sum_{ i \ge h } \Pr{\bar{X}= i} \Prb{S_{k-2}^{\bar{X}} = 2(n-a)-h - i} \\
		&\quad\le O(k) \sup_{i \ge h} \Pr{\bar{X}= i} \\
		&\quad\le O(k)  \Pr{\bar{X}= h} \\
		&\quad= O(n^{-3/2}).
	\end{align*}
	Consequently,
	\[
		 k\Pr{\bar{X}=h} \Prb{S_{k-1}^{\bar{X}} = 2(n-a)-h, \max_{1 \le i \le k-1} \bar{X}_i \ge h} = O(n^{-3}).
	\]
	By~\eqref{eq:sunday} it follows that
	\begin{multline}
		\Prb{S_k^{\bar{X}} = 2(n-a), \max_{1 \le i \le k} \bar{X}_i = h} \\
		=  O(n^{-3}) + k\Pr{\bar{X}=h} \Prb{S_{k-1}^{\bar{X}} = 2(n-a)-h}.
	\end{multline}
	From the local limit theorem~\eqref{eq:hllt} it follows that 
	\begin{align*}
		\Prb{S_{k-1}^{\bar{X}} = 2(n-a) - h} &= \frac{1}{a_{k-1}} \left(\mathfrak{h}\left(\frac{2(n-a)-h- (k-1)\Ex{\bar{X}}}{2a_{k-1}} \right) + o(1) \right) \\
		&= \frac{1}{ \omega n^{2/3}}\left(\mathfrak{h}\left(\frac{- x}{2 \omega  } + o(1) \right) + o(1) \right)  \\
		&= \frac{1}{ \omega n^{2/3}}\left(\mathfrak{h}\left(\frac{- x}{2 \omega  }  \right) + o(1) \right).
	\end{align*}
	By~\cite[Cor. 2.1]{MR2440928}
	\begin{align*}
			\Prb{S_k^{\bar{X}} = 2(n-a)} \sim k \Prb{\bar{X}= 2(n-a)-k \Ex{\bar{X}}} \sim k \Pr{\bar{X} = 2 \lfloor \chi n / 2 \rfloor}.
	\end{align*}	
	Thus, we arrive at
	\begin{align*}
		\frac{\Prb{S_k^{\bar{X}} = 2(n-a), \max_{1 \le i \le k} \bar{X}_i = h}}{\frac{1}{\omega n^{2/3}} \Prb{S_k^{\bar{X}} = 2(n-a)} } &= o(1) + \frac{\Prb{ \bar{X} = h} }{ \Prb{\bar{X} = 2 \lfloor \chi n / 2 \rfloor} } \left(\mathfrak{h}\left(\frac{- x}{2 \omega  }  \right) + o(1)\right).
	\end{align*}
	For $|x| \ge \log n $ the right-hand side simplifies to $o(1)$ since $\delta n \le h \le 2n$, and for $|x| \le \log n$ we have \[
	\frac{\Prb{ \bar{X} = h} }{ \Prb{\bar{X} = 2 \lfloor \chi n / 2 \rfloor} } \sim 1.
	\]
	In both cases,
	\[
			 \Prb{S_k^{\bar{X}} = 2(n-a), \max_{1 \le i \le k} \bar{X}_i = h} = \Prb{S_k^{\bar{X}} = 2(n-a)} \frac{1}{\omega n^{2/3}} \left(\mathfrak{h}(-x/(2 \omega)) + o(1) \right).
	\]
	This completes the verification of~\eqref{eq:toshow}. 
	
	Using~\eqref{eq:toshow},~\eqref{eq:logo} and~\eqref{eq:j0} it follows that
	\begin{multline*}
		\Pr{ (\#_{\cX} \mT, \#_{\cD} \mT ) = (2n,\ell), \Delta_{\cZ, \cX}(\mT) = h } \\
		= \frac{1}{\omega n^{2/3}}\left(\mathfrak{h}\left(\frac{- x}{2 \omega  }  \right) + o(1)\right)\Pr{ (\#_{\cX} \mT, \#_{\cD} \mT ) = (2n,\ell)}
	\end{multline*}
	uniformly for $n / \log^2 n \le h \le 2n$. 
	Summing over $\ell$ and using~\eqref{eq:toget} and~\eqref{eq:10am} it follows that
	\[
		 \Pr{ \Delta_{\cZ, \cX}(\mT) = h \mid \#_\cX(\mT) = 2n} = 
		\frac{1}{\omega n^{2/3}} \left( \mathfrak{h} \left(-\frac{x}{2\omega}\right) + o(1)\right)
	\]
	uniformly for all even integers $1 \le h \le 2n$ (or equivalently, all even integers $h \in \ndZ$).
	
	Next, we study the non-maximal numbers of type $\cX$ children of type $\cZ$ vertices.  For $x>0$ we let $\delta_x$ denote the Dirac measure with mass $1$ at the point $x$, and for $x \le 0$ we let $\delta_x = 0$ the measure with mass zero. We set \[
	\Xi_n = \sum_{j \ge 2} \delta_{n^{-2/3} {}_j\Delta_{\cZ, \cX}(\mT)},
	\] with ${}_j\Delta_{\cZ, \cX}(\mT) := 0$ whenever $j$ exceeds the number of type $\cZ$ vertices in $\mT$. Let $g: \ndR \to \ndR$ denote a continuous function with compact support in $]0,\infty[$. In particular, $g(0)=0$. Let $f$ denote the function on the space of locally finite point measures on $[0,\infty[$ with
	\[
		f(\Pi) = \exp\left(- \int g \,\mathrm{d}\Pi \right).
	\]
	Analogously to~\eqref{eq:j0} we have uniformly for
	\[
		|\Delta_\ell| \le \sqrt{n} \log n
	\]
	that
	\begin{align}
		\label{eq:logo5}
		&\Ex{ f(\Xi_n),  (\#_{\cX} \mT, \#_{\cD} \mT ) = (2n,\ell)}  \\
		&\qquad\qquad= n^{-\Theta(\log n)} + \frac{1}{\ell}  \quad \sum_{ \mathclap{ \substack{a,b,k \ge 0 \\ 
					|a - (1/2)\Ex{X_0} \ell| \le \sqrt{n} \log n \\
					|b - \Ex{L_0} \ell| \le \sqrt{n} \log n \\
					|k - \Ex{Z} \ell| \le \sqrt{n} \log n}}} \quad \Prb{ (S_\ell^{X_0}, S_\ell^{L_0}, S_\ell^Z) = (2a,b,k)} \nonumber \\ &\qquad\qquad\quad\,\,\,\Prb{S_k^{\bar{X}} = 2(n-a)} B \Ex{f(\hat{\Xi}_{n,k}) \mid S_k^{\bar{X}} = 2(n-a)} \nonumber 
	\end{align}
	with
	\[
		\hat{\Xi}_{n,k} = \sum_j \delta_{n^{-2/3} \bar{X}_j }.
	\]
	with $j$ ranging over all $1 \le j \le k$ except the index of the first occurrence of the maximum of $\bar{X}_1, \ldots, \bar{X}_k$.
	Let $\tau$ denote an operator that takes as input a finite sequence of nonnegative numbers and deletes the first occurrence of the maximum of that sequence, yielding a sequence that is one element shorter.
	Since $2(n-a) - k \Ex{\bar{X}} \sim \chi n$ we may apply~\cite[Thm. 1]{MR2775110}, yielding that 
	\begin{align}
		\label{eq:tvap}
		d_{\mathrm{TV}}( ( \tau(\bar{X}_1, \ldots, \bar{X}_k) \mid S_k^{\bar{X}}=2(n-a)), (\bar{X}_1, \ldots, \bar{X}_{k-1}) ) \to 0
	\end{align}
	uniformly as $n \to \infty$ with $a$ and $k$ in the ranges as in the sum in~\eqref{eq:logo5}. Set
	\[
		\tilde{\Xi}_{n,k} =  \sum_{j = 1}^{k-1} \delta_{n^{-2/3} \bar{X}_j }.
	\]
	By~\eqref{eq:tvap} it follows that
	\begin{align*}
		\Ex{f(\hat{\Xi}_{n,k}) \mid S_k^{\bar{X}} = 2(n-a)} &= \Ex{f(\tilde{\Xi}_{n,k}) } + o(1)
	\end{align*}
	with a uniform $o(1)$ term.  By~\eqref{pro:barxdensity} we have for each $t>0$
	\[
		(k-1) \Prb{ n^{-2/3} \bar{X} \ge t} \to \int_t^\infty  \frac{32\sqrt{3}}{5\sqrt{\pi}}\alpha \Ex{Z} x^{-5/2}\,\mathrm{d}x.
	\]
	Hence by  Resnick's point process convergence criterion~\cite[Prop. 3.21]{zbMATH04030574},
	\[
		\Ex{f(\tilde{\Xi}_{n,k}) } \to \Ex{f(\Xi)}
	\]
	uniformly in $k$ in the range appearing in~\eqref{eq:logo5} as $n \to \infty$. (To see that this is uniform in $k$, suppose that it is not. Then there exists a specific sequence $k(n)$ for which the assumptions of the point process convergence criterion are satisfied, but not its conclusion. A clear contradiction.) Thus we arrive at
	\[
		\Ex{f(\hat{\Xi}_{n,k}) \mid S_k^{\bar{X}} = 2(n-a)} = \Ex{f(\Xi)} + o(1)
	\]
	with a uniform $o(1)$ term. Substituting the preceding uniform estimate into~\eqref{eq:logo5} and comparing with~\eqref{eq:j0}, we obtain uniformly for
	\[
		|\Delta_\ell| \le \sqrt n\log n
	\]
	that
	\begin{align*}
		\Ex{ f(\Xi_n) ,    (\#_{\cX} \mT, \#_{\cD} \mT ) = (2n,\ell)}  \sim 
		\Ex{f(\Xi)} \Prb{(\#_{\cX} \mT, \#_{\cD} \mT ) = (2n,\ell)}.
	\end{align*}
	Summing over such $\ell$, using~\eqref{eq:toget}, and dividing by
	$\Pr{\#_{\cX}\mT = 2n}$ we obtain
	\[
		\Ex{ f(\Xi_n) \mid   \#_{\cX} \mT = 2n} \to  \Ex{f(\Xi)}.
	\]
	Consequently,
	\[
		(\Xi_n \mid  \#_{\cX} \mT = 2n) \convv \Xi.
	\]
	This readily yields
	\begin{align*}
		( n^{-2/3} {}_j\Delta_{\cZ, \cX}(\mT) \mid \#_{\cX} \mT = 2n)_{j \ge 2} \convd (J_i)_{i \ge 1}.
	\end{align*}
	
	Next, we consider the event $\mathfrak{E}$. Conditional on the event $\#_\cX \mT= 2n$  there can be at most $2 \log^2 n$ vertices of type $\cZ$ with at least $n / \log^2 n$ type $\cX$ children. Hence there are at most $4 \log^4 n$ pairs of such vertices. Suppose that one of those  pairs receives the same number $2n \ge s \ge n / \log^2 n$ of type $\cX$ children. 	Conditional on the tree and on the sizes of the large type-$\cZ$ vertices, the local decorations at distinct type-$\cZ$ vertices are independent, and a type-$\cZ$ vertex with $s$ type-$\cX$ children
	receives a uniformly selected labelled $3$-connected cubic planar graph
	with an oriented root edge on the vertex set $[s]$. There are $q_{s/2} s! / 2$ such graphs. For any unlabelled unrooted cubic planar graph with $s$ vertices there are at most $(3s)s!$ labelled versions with an oriented root edge. By~\eqref{eq:enumsimple} it follows that the probability for a specific pair of  decorations with the same number $s$ of type-$\cX$ children to receive decorations that are isomorphic as unrooted graphs is bounded by  $\exp(- \Theta(s)) = \exp(- \Theta(n / \log^2n))$. Hence we arrive at
	\[
			\Prb{\mathfrak{E} \mid \#_\cX \mT = 2n } \le   4 \log^4 n \exp(- \Theta(n/ \log^2 n)) = \exp(- \Theta(n/ \log^2 n)).
	\]

	It remains to consider the event $\Delta_{\cR, \cX}(\mT) \ge \log^2 n$ for  a vertex type $\cR \ne \cZ$. The number of type $\cX$ children of a vertex of type $\cR$ has finite exponential moments. Thus the probability for such a vertex to have at least $\log^2(n)$ type $\cX$ children is bounded by $n^{-\Theta(\log n)}$. Moreover, in the event $\#_{\cX} \mT = 2n$ there can be at most $2n$ vertices that have at least one type $\cX$ child. Since the event $\#_\cX \mT= 2n$ has probability $\Theta(n^{-5/2})$, it follows that
	\begin{align*}
		\Pr{ \Delta_{\cR, \cX}(\mT) \ge \log^2 n \mid \#_\cX \mT = 2n } \le O(n^{5/2})  (2n) n^{-\Theta(\log n)} = n^{-\Theta(\log n)}.
	\end{align*}
	This completes the proof.
\end{proof}

\begin{corollary}
	\label{co:deno}
	The probability that a uniformly at random selected unlabelled $\cD$ network with $2n$ vertices admits at least one automorphism $\sigma$ of the underlying unrooted cubic graph with $\cyl_1(\sigma)=0$ is bounded by $n^{-\Theta(\log n)}$.
\end{corollary}
\begin{proof}
	With $\mathfrak{E}$ as in Lemma~\ref{le:maxdeg}, let $\mathfrak{E}_0$ denote the event that the complementary event $\mathfrak{E}^{\mathrm{c}}$ holds jointly with $\#_\cX \mT= 2n$ and 
	\[
		\Delta_{\cZ, \cX}(\mT) \ge  n / \log^2 n 
	\]
	and jointly for any type $\cR \ne \cZ$
	\begin{align*}
		\Delta_{\cR, \cX}(\mT) \le \log^2 n.
	\end{align*}
	By Lemma~\ref{le:maxdeg} we have
	\[
		\Prb{ \mathfrak{E}_0 \mid \#_\cX \mT = 2n} \ge 1 - n^{-\Theta( \log n)}.
	\]
	
	In the event $\mathfrak{E}_0$ all $3$-connected components with more than $\log^2 n$ vertices arise from decorations of vertices of type $\cZ$. (There may be $3$-connected components arising from decorations of vertices of different type, but their size is bounded by $\log^2 n$.)
	
	If we take the conditioned tree  $(\mT \mid \#_\cX \mT= 2n)$, endow it with its independent vertex decorations and distribute labels uniformly at random then the result corresponds to a symmetry from $\Sym_\cD^+([2n])$ with the uniform distribution. The corresponding unlabelled $\cD$-network $\mD_{2n}$ is uniformly distributed among all unlabelled $\cD$-networks with $2n$ vertices.
	
	Any type $\cZ$ vertex with $s \ge n / \log^2n$ type-$\cX$ children receives as decoration a uniformly selected labelled $3$-connected cubic planar graph with an oriented root edge on the vertex set $[s]$. This happens independently from the decorations of other vertices and everything else. There are precisely $q_{s/2} s! / 2$ such graphs.  It follows by Proposition~\ref{pro:unl3con} that there are constants $A,a>0$ (that do not depend on $s$) such that at most $q_{s/2} s! A\exp(-a s)$ of these graphs admit a non-trivial automorphism of the underlying unrooted graph.
	
	Let $\mathfrak{E}_1$ denote the event that  $\mathfrak{E}_0$ holds and jointly all of the at most $2\log^2n$ large $3$-connected components (i.e. those with at least $n / \log^2n$ vertices) in the decorations admit no non-trivial automorphism of the underlying unrooted graph. It follows that
	\[
		\Prb{ \mathfrak{E}_1 \mid \#_\cX \mT = 2n} \ge 1 - n^{-\Theta( \log n)}.
	\]

	In the event $\mathfrak{E}_1$ any automorphism $\sigma$ of the unrooted graph underlying $\mD_{2n}$  restricts to the identity map on a $3$-connected component $C$. If an edge of $C$ did not get substituted by a network, then it is a cylindrical $1$-cycle of $\sigma$. If an edge of $C$ was substituted by a network, then the substituted network is mapped to itself and its two poles are fixed separately, because the two incident vertices in $C$ are fixed. Hence an attachment edge incident to one of the fixed vertices of $C$ is fixed with both endpoints fixed. We arrive at $\mathrm{cyl}_1(\sigma) > 0$. This completes the proof.
\end{proof}

\begin{corollary}
	\label{co:dcomponents}
	For each $i \ge 1$ let $\mathfrak{M}_i^\cD(2n)$ denote the number of vertices of the $i$th largest $3$-connected component in a uniform random $2n$-vertex unlabelled $\cD$-network. 
	We have uniformly for all even integers $h=\chi n + x n^{2/3}$ that
	\[
	\omega n^{2/3} \Pr{\mathfrak{M}_1^\cD(2n) = h} = \mathfrak{h} \left(-\frac{x}{2\omega}\right) + o(1).
	\]
	Moreover,
	\[
	\left(n^{-2/3}\mathfrak{M}_j^\cD(2n)\right)_{j \ge 2} \convd \left(J_i\right)_{i \ge 1}.
	\]
\end{corollary}
\begin{proof}
	The random tree $\mT$, conditioned on $\#_\cX \mT = 2n$, together with its decorations corresponds to a uniformly at random selected $2n$-vertex unlabelled $\cD$-network. Let $\mathfrak{F}$ denote the event that 
	$
	\Delta_{\cZ, \cX}(\mT) \ge  n / \log^2 n 
	$
	and  for any type $\cR \ne \cZ$ it holds that
	$
		\Delta_{\cR, \cX}(\mT) \le \log^2 n.
	$
	By Lemma~\ref{le:maxdeg} we have
	\[
	\Prb{ \mathfrak{F} \mid \#_\cX \mT = 2n} \ge 1 - n^{-\Theta( \log n)}.
	\]
	On the event $\mathfrak{F}$ all $3$-connected components with more than $\log^2 n$ vertices arise from decorations of vertices of type $\cZ$. The number of type $\cX$ children of such a vertex corresponds precisely to the number of vertices in the associated $3$-connected component.
	
	Thus, uniformly for all even integers $h=\chi n + x n^{2/3}$
	\begin{align*}
		&\omega n^{2/3} \Pr{\mathfrak{M}_1^\cD(2n) = h} \\
		&\quad = O(n^{2/3 - \Theta(\log n)}) + \omega n^{2/3} \Pr{ \Delta_{\cZ, \cX}(\mT) = h \mid \#_\cX(\mT) = 2n} \\
		&\quad = \mathfrak{h} \left(-\frac{x}{2\omega}\right) + o(1). 
	\end{align*}
	
	Likewise,  the point process limit for the remaining rescaled extremal degrees under the conditioning $\#_\cX \mT = 2n$ translates to the stated limit
	\[
		\left(n^{-2/3}\mathfrak{M}_j^\cD(2n)\right)_{j \ge 2} \convd \left(J_i\right)_{i \ge 1}.
	\]
	This completes the proof.
\end{proof}

\section{Simple networks and cubic planar graphs}

\label{sec:sicup}

The preceding section concluded with the asymptotic enumeration of $\cD$-networks. From this we will derive the asymptotic number of simple networks, connected unlabelled cubic planar graphs, and unconstrained unlabelled  cubic planar graphs.

\subsection{Subexponential sequences}

We collect several useful results on subexponential sequences following mostly~\cite{MR3097424}.  See also \cite{MR0348393,zbMATH01593195, MR714482,MR772907,zbMATH05346192} for the development of this theory.

\begin{definition}
	Let $d \ge 1$ be an integer. A power series $f(x) = \sum_{n =0}^\infty f_n x^n$ with nonnegative coefficients and radius of convergence $r >0$ belongs to the class $\mathscr{S}_d$, if $f_n=0$ whenever $n$ is not divisible by $d$, and
	\begin{align}
		\label{eq:condition}
		\frac{f_n}{f_{n+d}} \sim r^d, \qquad \frac{1}{f_n}\sum_{i+j=n}f_if_j \sim 2 f(r) < \infty
	\end{align}
	as $n \equiv 0 \mod d$ becomes large.
\end{definition}

The broad scope of this setting is illustrated by the following  observation, which has been noted in various places, see for example \cite{MR772907}.
\begin{proposition}
	\label{pro:easy}
	If  for some constants $r>0$, $b > 1$ and a slowly varying function $h$  we have $f_n = h(n) n^{-b} r^{-n}$ for large enough $n$, then the series $\sum_{n \in d\ndN_0} f_n x^n$ belongs to the class $\mathscr{S}_d$.
\end{proposition}

We make a basic observation:
\begin{proposition}
	\label{pro:negle}
	Let $f(x) = \sum_{n \ge 0} f_n x^n$ belong to $\mathscr{S}_d$ with radius of convergence $r$. Let $g(x)$ denote a power series with nonnegative coefficients such that there exists $s>r$ with $g(s)< \infty$. Then for any $\frac{r}{s}<t <1$ there exists a constant $C_t>0$ such that for all integers $n \ge 0$ with $[x^n]f(x) \ne 0$
	\[
		 [x^n] g(x) \le C_t t^n [x^n] f(x).
	\]
\end{proposition}
\begin{proof}
	Throughout this proof we restrict $n$ to multiples of $d$. By~\eqref{eq:condition} for any $\epsilon>0$ there exists $N \ge 1$ divisible by $d$ such that for  $n \ge N$
	\[
		\frac{f_{n+d}}{f_n} \ge r^{-d} (1-\epsilon)^d.
	\]
	Consequently, for  $n \ge N$ and $k \ge 1$
	\[
		\frac{f_{n+kd}}{f_n} \ge r^{-kd} (1- \epsilon)^{kd}.
	\]
	This entails for  $n \ge N$ 
	\[
		f_n \ge r^{-(n-N)} (1- \epsilon)^{(n-N)}  f_N .
	\]
	On the other hand, for any $s>r$ with $g(s)< \infty$
	we have 
	\[
		[x^n]g(x) = o(s^{-n}).
	\]
	Thus
	\[
		\frac{1}{f_n} [x^n]g(x) \le o(1) \left( \frac{r}{(1-\epsilon)s} \right)^n.
	\]
	In particular, there exists a constant $C>0$ such that for $n \ge N$
	\[
		\frac{1}{f_n} [x^n]g(x) \le C \left( \frac{r}{(1-\epsilon)s} \right)^n.
	\]
	By replacing $C$ with a possibly larger constant this also holds for integers $0 \le n \le N$ with $f_n \ne 0$. As $\epsilon>0$ was arbitrary, this completes the proof.
\end{proof}

The following observation is a variant of \cite[Thm. 4.9]{MR3097424}.

\begin{lemma}
	\label{le:help}
	Let $f(x)$ belong to $\mathscr{S}_d$ with radius of convergence $r$, and $h_1(x), h_2(x)$ be power-series with nonnegative coefficients such that $[x^n]h_i(x)=0$ for $i=1,2$ whenever $n$ is not divisible by $d\ge1$. With $n$ restricted to  multiples of $d$, set
	\begin{align*}
		\overline{o}_1 &= \limsup_{n \to \infty} \frac{[x^n]h_1(x)}{[x^n]f(x)},  &\overline{o}_2 &= \limsup_{n \to \infty}\frac{[x^n]h_2(x)}{[x^n]f(x)}, \\
		\underline{o}_1 &= \liminf_{n \to \infty} \frac{[x^n]h_1(x)}{[x^n]f(x)},   &\underline{o}_2 &= \liminf_{n \to \infty} \frac{[x^n]h_2(x)}{[x^n]f(x)}.
	\end{align*}
	Suppose that $\overline{o}_1, \overline{o}_2 < \infty$.  Then, with $n$ restricted to multiples of $d$
	\begin{align*}
		\limsup_{n \to \infty} \frac{[x^n]h_1(x)h_2(x)}{[x^n]f(x)} &\le  \overline{o}_2h_1(r) + \overline{o}_1h_2(r), \\
		\liminf_{n \to \infty} \frac{[x^n]h_1(x)h_2(x)}{[x^n]f(x)} &\ge  \underline{o}_2h_1(r) + \underline{o}_1h_2(r).
	\end{align*}
\end{lemma}
\begin{proof}
	Write $a_n = [x^n] h_1(x)$ and $b_n = [x^n] h_2(x)$ and $f_n = [x^n]f(x)$. For any constant $K \ge 1$, by~\eqref{eq:condition}
	 \begin{align*}
		\frac{1}{f_n}\sum_{i=0}^{K} a_i b_{n-i} &\le \frac{\overline{o}_2+o(1)}{f_n}\sum_{i=0}^{K} a_i f_{n-i}  \sim \overline{o}_2\sum_{i=0}^{K} a_i r^i
	 \end{align*}
	 and likewise
 	 \begin{align*}
 	 	\frac{1}{f_n}\sum_{i=n-K}^{n} a_i b_{n-i} &\le \frac{\overline{o}_1+ o(1)}{f_n}\sum_{i=n-K}^{n} f_i b_{n-i}  \sim \overline{o}_1 \sum_{i=0}^{K} b_i r^i.
 	 \end{align*}
 	 (Note that for $n$ divisible by $d$ we have $f_n > 0$ for large enough $n$. Furthermore,  $a_if_{n-i}$ and $f_i b_{n-i}$ are  equal to zero when $i$ is not divisible by $d$. Hence we only need to apply~\eqref{eq:condition} to the summands where  $i$ is a multiple of $d$.) Furthermore, for any $\epsilon>0$ we may choose $K$ large enough so that by~\eqref{eq:condition} 
 	 \begin{align*}
 	 	\frac{1}{f_n}\sum_{i=K+1}^{n-K-1} a_i b_{n-i} &\le  (\overline{o}_1 \overline{o}_2 + \epsilon) \frac{1}{f_n}\sum_{i=K+1}^{n-K-1} f_i f_{n-i} \\
 	 	&\sim 2 (\overline{o}_1 \overline{o}_2 + \epsilon) \left(f(r) - \sum_{i=0}^K f_i r^i \right).
 	 \end{align*}
 	 Since we may take $K$ arbitrarily large, it follows that
 	 \[
 	 		\limsup_{n \to \infty} \frac{[x^n]h_1(x)h_2(x)}{[x^n]f(x)} \le  \overline{o}_2h_1(r) + \overline{o}_1h_2(r).
 	 \]
 	 The lower bound for the limit inferior follows analogously.
\end{proof}

The following result describes the behaviour of a one-dimensional random walk with i.i.d. steps that is stopped at a random time that is independent from its steps.

\begin{proposition}[{\cite[Thm. 4.8, 4.30]{MR3097424}}] \label{pro:cosub}
	If $f(x)$ belongs to $\mathscr{S}_d$ with radius of convergence $r>0$, and $g(x)$ is a non-constant power series with nonnegative coefficients that is analytic at $f(r)$, then $g(f(x))$ belongs to $\mathscr{S}_d$ and \[ [x^n] g(f(x)) \sim g'(f(r)) [x^n]f(x), \qquad n \to \infty, \qquad n \equiv 0 \mod d. \]
\end{proposition}
We extend this result to more general compositions:
\begin{lemma}
	\label{le:hela}
	Let $f(x)$ belong to $\mathscr{S}_d$ with radius of convergence $r$. Let $F(x, u)$ denote a power series with nonnegative coefficients such that for some $\epsilon>0$
	\[
		F(r + \epsilon, f(r) + \epsilon) < \infty.
	\]
	Furthermore, suppose that $[x^n]F(x,f(x))=0$ whenever $n$ is not divisible by~$d$.  Then
	\[
		[x^n] F(x, f(x)) \sim \frac{\partial F}{\partial u} (r, f(r)) [x^n] f(x).
	\]
\end{lemma}
\begin{proof}
	We may write
	\[
		F(x,u) = \sum_{k \ge 0} F_k(x) u^k
	\]
	with $F_k(x)$ a power series with nonnegative coefficients. Since we assumed $[x^n]F(x,f(x))=0$ whenever $n$ is not divisible by~$d$, it follows that $[x^n]F_k(x)=0$ whenever $n$ is not divisible by $d$.

	Let us now restrict $n$ to multiples of $d$. We write $f_n = [x^n] f(x)$. By assumption,
	\[
		\sum_{k \ge 0} F_k(r + \epsilon) (f(r) + \epsilon)^k < \infty.
	\]
	This entails $F_k(r + \epsilon)< \infty$  and hence by Proposition~\ref{pro:negle}
	\[
	[x^n]F_k(x) = o(f_n).
	\]	
	Proposition~\ref{pro:cosub} implies 
	\[
		[x^n]f(x)^k \sim k f(r)^{k-1} f_n
	\]
	for any constant $k \ge 0$. By Lemma~\ref{le:help} it follows that
	\[
		[x^n] F_k(x) f(x)^k \sim k f(r)^{k-1}  F_k(r) f_n .
	\]
	Hence, for any $K \ge 0$
	\begin{align}
		\label{eq:gof}
		[x^n] \sum_{k=0}^K F_k(x) f(x)^k \sim f_n \sum_{k=0}^K F_k(r) kf(r)^{k-1} .
	\end{align}
	Set $s=f(r) + \epsilon/2$ and define
	\[
		\Phi_s(x) = \frac{x}{(1-x/s)^2} = \sum_{k \ge 1} \frac{k}{s^{k-1}} x^k.
	\]
	This way, $\Phi_s$ has radius of convergence $s>f(r)$. Hence by 	Proposition~\ref{pro:cosub}
	\[
		[x^n] \Phi_s(f(x)) \sim \Phi_s'(f(r)) f_n.
	\]
	Moreover, coefficient-wise
	\[
		f(x)^k \le_{\mathrm{coeff}} \frac{s^{k-1}}{k} \Phi_s(f(x)).
	\]
	Hence, coefficient-wise
	\[
		\sum_{k \ge K} F_k(x) f(x)^k \le_{\mathrm{coeff}} \Phi_s(f(x))  \sum_{k \ge K} F_k(x) \frac{s^{k-1}}{k}.
	\]
	The series $\sum_{k \ge K} F_k(x) \frac{s^{k-1}}{k}$  is convergent at $x=r + \epsilon$, hence by Proposition~\ref{pro:negle}
	\[
		[x^n]\sum_{k \ge K} F_k(x) \frac{s^{k-1}}{k} = o(f_n).
	\]
	Hence by Lemma~\ref{le:help}
	\[
		[x^n] \Phi_s(f(x)) \sum_{k \ge K} F_k(x) \frac{s^{k-1}}{k} \sim f_n \Phi_s'(f(r))  \sum_{k \ge K} F_k(r) \frac{s^{k-1}}{k}.
	\]
	It follows that
	\[
		\lim_{K \to \infty} \limsup_{n \to \infty} \frac{1}{f_n} [x^n] \sum_{k \ge K} F_k(x) f(x)^k = 0.
	\]
	By~\eqref{eq:gof}, it follows that
	\[
		\frac{1}{f_n} [x^n] \sum_{k=0}^\infty F_k(x) f(x)^k \to  \sum_{k=0}^\infty  F_k(r) k f(r)^{k-1}.
	\]
	This completes the proof.
\end{proof}

\begin{lemma}
	\label{le:help2}
		Let $f(x)$ belong to $\mathscr{S}_d$ with radius of convergence $r$. Let $h_1(x), \ldots, h_m(x)$ denote power-series with nonnegative coefficients such that $[x^n]h_i(x)=0$ whenever $n$ is not divisible by $d\ge1$. With $n$ restricted to multiples of $d$, set
	\begin{align*}
		\overline{o}_i &= \limsup_{n \to \infty} \frac{[x^n]h_i(x)}{[x^n]f(x)}, \\
		\underline{o}_i &= \liminf_{n \to \infty} \frac{[x^n]h_i(x)}{[x^n]f(x)},
	\end{align*}
	for all $1 \le i \le m$.
	Suppose that $\overline{o}_i < \infty$ for all $i$.
	Let $F(x, u_1, \ldots, u_m)$ denote a power-series with nonnegative coefficients such that for some $\epsilon>0$ we have
	\[
		F(r + \epsilon, h_1(r) + \epsilon, \ldots,  h_m(r) + \epsilon) < \infty.
	\]
	Furthermore, suppose that $[x^n]F(x, h_1(x), \ldots, h_m(x))=0$ whenever $n$ is not divisible by~$d$.
	Then, with $n$ restricted to multiples of $d$
	\begin{align*}
		\limsup_{n \to \infty} \frac{[x^n] F(x, h_1(x), \ldots, h_m(x))}{[x^n] f(x)}   \le \sum_{i=1}^m \overline{o}_i \frac{\partial F}{\partial u_i} (r, h_1(r), \ldots, h_m(r))
	\end{align*}
	and
	\begin{align*}
	\liminf_{n \to \infty} \frac{[x^n] F(x, h_1(x), \ldots, h_m(x))}{[x^n] f(x)}   \ge \sum_{i=1}^m \underline{o}_i \frac{\partial F}{\partial u_i} (r, h_1(r), \ldots, h_m(r)).
	\end{align*}
\end{lemma}
\begin{proof}
Without loss of generality we may assume that $h_i(x)$ is not a constant for all $1 \le i \le m$.
Let $f_n = [x^n]f(x)$ and $h_{i,n} = [x^n]h_i(x)$. We may write
\[
F(x, u_1, \ldots, u_m) = \sum_{k_1,\ldots,k_m \ge 0} F_{k_1,\ldots,k_m}(x) u_1^{k_1} \cdots u_m^{k_m}.
\]
For each $k_1, \ldots, k_m \ge 0$ the coefficients of $F_{k_1,\ldots,k_m}(x)$ are nonnegative. Since for all $1 \le i \le m$ we have that $h_i(x)$ is a power series in $x^d$ with nonnegative coefficients, the assumption that
\[
[x^n]F(x,h_1(x),\ldots,h_m(x)) = 0
\]
whenever $n$ is not divisible by $d$ implies that  \[
[x^n]F_{k_1,\ldots,k_m}(x) = 0
\]
whenever $n$ is not a multiple of $d$. Let us now restrict $n$ to multiples of $d$. Let $\epsilon>0$ and set
\[
C_i = \overline{o}_i + \epsilon
\]
for $1 \le i \le m$. There exists $N \in d\ndN$ large enough so that
\[
h_{i,n} \le C_i f_n
\]
for all $1 \le i \le m$ and all $n \ge N$. We set
\begin{align*}
p_i(x) = \sum_{n \le N} h_{i,n} x^n
\end{align*}
and
\[
	q(x) = \sum_{n > N} f_n x^n.
\]
Then $q(x)$ belongs to $\mathscr{S}_d$, since removing finitely many coefficients does not affect~\eqref{eq:condition}. Moreover, coefficient-wise
\[
h_i(x) \le_{\mathrm{coeff}} p_i(x) + C_i q(x)
\]
for all $1 \le i \le m$. Note that 
\begin{align}
	\label{eq:capN}
	\lim_{N \to \infty}(p_i(r) + C_i q(r)) =  h_i(r).
\end{align}
Hence, by increasing $N$ if necessary, we may assume that
\[
p_i(r) + C_i q(r) < h_i(r) + \epsilon/2
\]
for all $1 \le i \le m$. Since each $p_i$ is a polynomial it follows that there exists $\delta>0$ such that
\[
	p_i(r+\delta) + C_i (q(r) + \delta) < h_i(r) + \epsilon, \qquad 1 \le i \le m.
\]
We set
\[
	G(x,u) = F(x,p_1(x) + C_1 u, \ldots, p_m(x) + C_m u).
\]
Then $G(x,u)$ has nonnegative coefficients, satisfies
\[
	G(r+\delta,q(r)+\delta)<\infty,
\]
and $[x^n]G(x,q(x))=0$ whenever $n$ is not divisible by $d$. By Lemma~\ref{le:hela} it follows that
\[
	[x^n]G(x,q(x)) \sim \frac{\partial G}{\partial u} (r,q(r)) [x^n]q(x).
\]
Since $[x^n]q(x) = f_n$ for all $n > N$, and since coefficient-wise
\[
	F(x,h_1(x),\ldots,h_m(x)) \le_{\mathrm{coeff}} G(x,q(x)),
\]
it follows that
\[
	\limsup_{n \to \infty} \frac{[x^n]F(x,h_1(x),\ldots,h_m(x))}{f_n} \le \frac{\partial G}{\partial u}(r,q(r)).
\]
We have
\[
	\frac{\partial G}{\partial u}(r,q(r)) = \sum_{i=1}^m C_i \frac{\partial F}{\partial u_i} (r,p_1(r)+C_1q(r),\ldots,p_m(r)+C_mq(r)).
\]
By~\eqref{eq:capN}, letting $N$ tend to infinity yields
\[
\limsup_{n \to \infty} \frac{[x^n]F(x,h_1(x),\ldots,h_m(x))}{f_n} \le \sum_{i=1}^m C_i \frac{\partial F}{\partial u_i} (r, h_1(r), \ldots, h_m(r)).
\]
Letting $\epsilon \to 0$ we obtain
\begin{align*}
	\limsup_{n \to \infty} \frac{[x^n] F(x, h_1(x), \ldots, h_m(x))}{[x^n] f(x)}   \le \sum_{i=1}^m \overline{o}_i \frac{\partial F}{\partial u_i} (r, h_1(r), \ldots, h_m(r)).
\end{align*}
The asserted bound for the limit inferior follows analogously.
\end{proof}

\begin{remark}
	Lemma~\ref{le:help2} may be applied to situations where  $F$ has coefficients with mixed signs in the following way. We simply need to write $F = F^+ - F^-$ for power series $F^+$ and $F^-$ with positive coefficients and apply Lemma~\ref{le:help2} to each, if they satisfy the required conditions.
\end{remark}

\subsection{Simple networks}

By Equation~\eqref{eq:simpnetwork}, simple networks are enumerated by
\begin{align*}
	\cN_{\mathrm{s}}	= \cH  + \cP_{\mathrm{s}}  + \cI + \cS_{\mathrm{s}}.
\end{align*}
By~\eqref{eq:D},~\eqref{eq:sdplus},~\eqref{eq:pdplus} and \eqref{eq:argc} this may be expressed by
\begin{align}
	\label{eq:gotcha}
	\cN_{\mathrm{s}} 	&= \cD + \cI - \cL -  \cP_{\mathrm{d}}  - \cS_{\mathrm{d}} \\
			&= \cD + \frac{1}{x^2}\cL^2 - \cL -x^2 \cD -\cL^2 \nonumber \\
			&= (1 - x^2)\cD + \left(\frac{1}{x^2} - 1\right) \cL^2 - \cL. \nonumber
\end{align}
Recall by~\eqref{eq:biermoe},~\eqref{eq:lroot} and~\eqref{eq:dtwo} that
\begin{align*}
	2(1 + x^2) \cL &= \cL^2  + x^2(\cD + \cD^-) + \cL_2, \\
	\cL &= 1 + x^2 - \sqrt{1 + 2 x^2 + x^4 - \cL_2 - x^2(\cD + \cD^- )}, \\
	\cD^- &= \cL + (1 + \cD^-) (\cD_2 - \cS_2) + x^2 \cD^- + \frac{x^2}{2}( (\cD^-)^2 + \cD_2) \\
	&\quad\,  + R_2(x,x^2, 1+ \cD, 1 + \cD_2, 1 + \cD^-). 
\end{align*}
Recall from Corollary~\ref{co:sharpy} that $0 \le \Delta \le 0.047$ for
\begin{align*}
	\Delta &= \frac{\rho^2}{2\sqrt{1 + 2\rho^2 + \rho^4 - \cL_2(\rho) - \rho^2(\cD(\rho) + \cD^-(\rho) )}} \\
			&= \frac{\rho^2}{2(1 + \rho^2 - \cL(\rho))}.
\end{align*}
 For notational convenience, we set
\[
d_n = [x^n] \cD(x).
\]
By~\eqref{eq:bista} we have that
\begin{align}
	\label{eq:bista2}
d_n \sim \cD(\rho) \frac{64 \sqrt{3}}{5 \sqrt{\pi}} \Exb{Z} 
\frac{(1 - \Exb{L} + \frac{3}{2} \mu \Exb{X})^{3/2}}{(1-\Exb{L})^{5/2}}  n^{-5/2} \rho^{-n}
\end{align}
as $n \in 2\ndN$ tends to infinity.

\begin{lemma}
	\label{le:blo1}
	Let
	\[
		\vartheta = (\rho,\rho^2, 1 + \cD(\rho), 1 + \cD_2(\rho), 1 + \cD^-(\rho))
	\]
	and set
	\begin{align*}
		c_{\cD^-} 	&= \frac{\Delta + \partial_3 R_2(\vartheta)}{1 - \Delta - \cD_2(\rho) + \cS_2(\rho) - \rho^2(1 + \cD^-(\rho)) - \partial_5 R_2(\vartheta)}, \\
		c_\cL		&= \Delta(1 + c_{\cD^-}), \\
		c_{\cN_{\mathrm{s}}}	&= 1 - \rho^2 - c_\cL + 2(\rho^{-2} - 1)c_\cL \cL(\rho).
	\end{align*}
	The three constants are finite, positive, and as $n \in 2 \ndN$ tends to infinity,
	\begin{align*}
		[x^n] \cD^-(x) &\sim c_{\cD^-} d_n, \\
		[x^n] \cL(x) &\sim c_\cL d_n, \\
		[x^n] \cN_{\mathrm{s}}(x) &\sim c_{\cN_{\mathrm{s}}} d_n.
	\end{align*}
\end{lemma}

\begin{proof}
	Throughout this proof, $n$ is restricted to positive even integers. Since $\cD^-(x)$ and $\cL(x)$ are dominated coefficient-wise by $\cD(x)$, the quantities
	\begin{align*}
		\overline{\delta} &= \limsup_{n \to \infty} \frac{[x^n]\cD^-(x)}{d_n}, &
		\underline{\delta} &= \liminf_{n \to \infty} \frac{[x^n]\cD^-(x)}{d_n}, \\
		\overline{\lambda} &= \limsup_{n \to \infty} \frac{[x^n]\cL(x)}{d_n}, &
		\underline{\lambda} &= \liminf_{n \to \infty} \frac{[x^n]\cL(x)}{d_n}
	\end{align*}
	are finite. By Proposition~\ref{pro:easy}, $\cD(x)$ belongs to $\mathscr{S}_2$. Hence
	\[
		\frac{d_{n-2}}{d_n} \to \rho^2.
	\]
	It follows that
	\begin{align*}
		\limsup_{n \to \infty} \frac{[x^n](1+x^2)\cL(x)}{d_n} &= (1 + \rho^2) \overline{\lambda}, \\
		\liminf_{n \to \infty} \frac{[x^n](1+x^2)\cL(x)}{d_n} &= (1 + \rho^2) \underline{\lambda}.
	\end{align*}
	Moreover, $\rho<1$ implies that $\cL_2(x)$, $\cD_2(x)$ and $\cS_2(x)$ are analytic at $\rho$. By Proposition~\ref{pro:negle} it follows that their coefficients grow at speed $o(d_n)$.
	By Equation~\eqref{eq:biermoe},
	\[
	2(1 + x^2)\cL = x^2(\cD + \cD^-) + \cL^2+ \cL_2.
	\]
	By Lemma~\ref{le:help} it follows that
	\begin{align*}
		2(1+\rho^2)\overline{\lambda}
		&\le \rho^2(1 + \overline{\delta}) + 2 \overline{\lambda}\cL(\rho), \\
		2(1+\rho^2)\underline{\lambda}
		&\ge \rho^2(1 + \underline{\delta}) + 2 \underline{\lambda}\cL(\rho).
	\end{align*}
	Hence
	\begin{align}
		\overline{\lambda} &\le \Delta(1 + \overline{\delta}), \label{eq:blo1d} \\
		\underline{\lambda} &\ge \Delta(1 + \underline{\delta}). \label{eq:blo1e}
	\end{align}
	
	Next, let
	\[
	\Psi(x,u,v) = R_2(x,x^2, 1 + u, 1 + \cD_2(x), 1 + v).
	\]
	Since $R_2$ has nonnegative coefficients, so does~$\Psi$. By Lemma~\ref{le:analyticsum}, $\Psi$ is analytic at the point $(\rho,\cD(\rho),\cD^-(\rho))$. It follows by Lemma~\ref{le:help2} 
	\begin{align*}
		\limsup_{n \to \infty} \frac{[x^n]\Psi(x,\cD(x),\cD^-(x))}{d_n}
		&\le \partial_3 R_2(\vartheta) + \overline{\delta}\partial_5 R_2(\vartheta),\\
		\liminf_{n \to \infty} \frac{[x^n]\Psi(x,\cD(x),\cD^-(x))}{d_n}
		&\ge \partial_3 R_2(\vartheta) + \underline{\delta}\partial_5 R_2(\vartheta).
	\end{align*}
	Equation~\eqref{eq:dtwo} may be written as
	\[
		\cD^- = \cL + (1 + \cD^-) (\cD_2 - \cS_2) + x^2 \cD^- + \frac{x^2}{2}( (\cD^-)^2 + \cD_2)  + \Psi(x,\cD(x),\cD^-(x)). 
	\]
	By Lemma~\ref{le:help} it follows that
	\begin{align*}
		\overline{\delta}
		&\le \overline{\lambda}
		+ (\cD_2(\rho)-\cS_2(\rho))\overline{\delta}
		+ \rho^2\overline{\delta}
		+ \rho^2\cD^-(\rho)\overline{\delta}
		+ \partial_3 R_2(\vartheta)
		+ \overline{\delta}\,\partial_5 R_2(\vartheta), \\
		\underline{\delta}
		&\ge \underline{\lambda}
		+ (\cD_2(\rho)-\cS_2(\rho))\underline{\delta}
		+ \rho^2\underline{\delta}
		+ \rho^2\cD^-(\rho)\underline{\delta}
		+ \partial_3 R_2(\vartheta)
		+ \underline{\delta}\,\partial_5 R_2(\vartheta).
	\end{align*}
	Using~\eqref{eq:blo1d} and~\eqref{eq:blo1e}, this yields
	\begin{align*}
		\overline{\delta}
		&\le \Delta + \partial_3 R_2(\vartheta) + \left(\Delta + \cD_2(\rho)-\cS_2(\rho) + \rho^2(1+\cD^-(\rho)) + \partial_5 R_2(\vartheta)\right)\overline{\delta}, \\
		\underline{\delta}
		&\ge \Delta + \partial_3 R_2(\vartheta)  + \left(\Delta + \cD_2(\rho)-\cS_2(\rho) + \rho^2(1+\cD^-(\rho)) + \partial_5 R_2(\vartheta)\right)\underline{\delta}.
	\end{align*}
	By Corollary~\ref{co:sharpy} and $\cD^-(\rho)\le \cD(\rho)$, we have
	\begin{align*}
		&\Delta + \cD_2(\rho)-\cS_2(\rho) + \rho^2(1+\cD^-(\rho)) + \partial_5 R_2(\vartheta) \\
		&\le 0.047 + 0.0004 + \rho_{\mathrm{lab}}^2(1 + 0.038) + \frac{29}{128} \\
		&\le 0.38.
	\end{align*}
	Since this is less than $1$, it follows that
	\begin{align*}
		\overline{\delta} &\le \frac{\Delta + \partial_3 R_2(\vartheta)}{1 - \left(\Delta + \cD_2(\rho)-\cS_2(\rho) + \rho^2(1+\cD^-(\rho)) + \partial_5 R_2(\vartheta)\right) } = c_{\cD^-}, \\
		\underline{\delta} &\ge c_{\cD^-}.
	\end{align*}
	In other words,
	\[
		\overline{\delta} = \underline{\delta} =  c_{\cD^-}.
	\]
	Hence
	\[
		[x^n]\cD^-(x) \sim c_{\cD^-} d_n.
	\]
	Using Corollary~\ref{co:sharpy}, we obtain
	\[
		0 <	c_{\cD^-} < \frac{0.047 + 29/128}{1 - 0.38} \le 0.442.
	\]	
	From~\eqref{eq:blo1d} and~\eqref{eq:blo1e}, it follows that
	\[
		\overline{\lambda}  = \underline{\lambda} =  \Delta (1 + \overline{\delta}) = c_{\cL}.
	\]
	Hence
	\[
		[x^n]\cL(x) \sim c_\cL d_n.
	\]
	By Corollary~\ref{co:sharpy},
	\[
		0 < c_\cL < 0.047 (1 + 0.442) \le 0.0678.
	\]
	By Equation~\eqref{eq:gotcha}
	\begin{align*}
		\cN_{\mathrm{s}} 
		= (1 - x^2)\cD + \left(\frac{1}{x^2} - 1\right) \cL^2 - \cL. \nonumber
	\end{align*}
	Hence, by Lemma~\ref{le:help},
	\[
		[x^n]\cN_{\mathrm{s}}(x) \sim \left(1 - \rho^2 - c_\cL + 2(\rho^{-2}-1)c_\cL \cL(\rho)\right) d_n = c_{\cN_{\mathrm{s}}} d_n.
	\]
	By Corollary~\ref{co:sharpy},
	\begin{align*}
		\rho^2 + c_\cL + 2(\rho^{-2} -1)c_{\cL} \cL(\rho) &\le \rho_{\mathrm{lab}}^2 + 0.0678 + 2( 0.307^{-2} -1) \cdot 0.0678 \cdot 0.012 \\
		&\le 0.186.
	\end{align*}
	Hence
	\[
		0.814 < c_{\cN_{\mathrm{s}}} < \infty. 
	\]
	This completes the proof.
\end{proof}

For any species of networks $\cJ$ let $\mA_n^{\cJ}$ denote the sum of $1$ plus the number of  cylindrical $1$-edge-cycles in a uniformly at random selected element from $\Sym_{\cJ}^{+}[n]$ .

\begin{lemma}
	\label{le:concentrate}
	We have
	\begin{align*}
		\Prb{ |\mA_n^{\cN_{\mathrm{s}}} - \beta n/2| \ge \sqrt{n} \log n} \le n^{-\Theta(\log n)}.
	\end{align*}
\end{lemma}
\begin{proof}
	Recall that
	\[
			\cN_{\mathrm{s}}	= \cH  + \cP_{\mathrm{s}}   + \cS_{\mathrm{s}} + \cI.
	\]
	The species of networks $\tilde{\cN} := \cH  + \cP_{\mathrm{s}} + \cS_{\mathrm{s}}$ is a subspecies of $\cD$. From~\eqref{eq:rhopsplus} and Lemma~\ref{le:hela}  it follows that
	\begin{align*}
		\limsup_{n \to \infty} \frac{d_n}{[z^n] \tilde{\cN}(z)} < \infty.
	\end{align*}
	It is elementary (from simply dividing the number of ``bad'' networks by the total number of networks) that
	\begin{align*}
		\Prb{ |\mA_n^{\tilde{\cN}} - \beta n/2| \ge \sqrt{n} \log n} &\le \Prb{ |\mA_n^{\cD} - \beta n/2| \ge \sqrt{n} \log n} \frac{d_n}{[z^n] \tilde{\cN}(z)} \\
		&\le \Prb{ |\mA_n^{\cD} - \beta n/2| \ge \sqrt{n} \log n} \sup_{\substack{m \ge n \\ m \in 2\ndN}} \frac{d_m}{[z^m] \tilde{\cN}(z)}.
	\end{align*}
	By Corollary~\ref{co:piz} it follows that for any $c>0$
	\begin{align}
		\label{eq:par1}
		\Prb{ |\mA_n^{\tilde{\cN}} - \beta n/2| \ge c\sqrt{n} \log n} \le n^{-\Theta(\log n)}.
	\end{align}
	Recall that by Lemma~\ref{le:blo1} 
	\[
		[z^n] \cL(z) \sim c_{\cL} d_n.
	\]
	Analogously as for~\eqref{eq:par1} it follows that
	\begin{align}
		\label{eq:par2}
		\Prb{ |\mA_n^{\cL} - \beta n/2| \ge c\sqrt{n} \log n} \le n^{-\Theta(\log n)}.
	\end{align}
	Since an $n$-vertex $\cL$-network has at most $3n/2$ edges it holds with $C = (3 + \beta)/2$ trivially that
	\[
		|\mA_n^{\cL} - \beta n/2| \le C n
	\]
	for all $n$. Consequently, with $m$ restricted to even integers for which at least one $\cL$-network with $m$ vertices exists,
	\begin{align}
		\label{eq:dm}
		\sup_{m \le n} \Prb{|\mA_m^{\cL} - \beta m/2| \ge c \sqrt{n} \log n} \le n^{-\Theta(\log n)}.
	\end{align}
	Indeed, for $m \le C^{-1} c \sqrt{n} \log n $ the upper bound holds trivially, and for $n \ge m > C^{-1} c \sqrt{n} \log n$ we have $\log m = \Theta(\log n)$ and hence by~\eqref{eq:par2}
	\[
		\Prb{ |\mA_m^{\cL} - \beta m/2| \ge c\sqrt{m} \log m} \le m^{-\Theta(\log m)} \le n^{-\Theta(\log n)}.
	\]
	
	As established in~\eqref{eq:argc}, we have that $\cI(x) = \cL(x)^2 / x^2$ since any $\cI$-network may be constructed from two $\cL$-networks as illustrated in Figure~\ref{fi:isthmus}. 
	
	Hence a uniformly selected $\cI$-symmetry with vertex set $[n]$ is obtained by numbers of vertices $n_1,n_2$ satisfying $n_1+n_2=n+2$ with probability
	proportional to
	\[
		[z^{n_1}]\cL(z)\,[z^{n_2}]\cL(z),
	\]
	and then composing two independent uniformly at random selected symmetries with $n_1$ and $n_2$ vertices (involving a final uniform relabelling step that the vertex set becomes $[n]$).

	This allows us to write
	\[
		\mA_n^{\cI} \eqdist \mA_{n_1}^{\cL,1}+\mA_{n_2}^{\cL,2}-3
	\]
	with $(\mA_{m}^{\cL,1})_m$ and $(\mA_{m}^{\cL,2})_m$ independent from each other and from $(n_1, n_2)$, and $\mA_{m}^{\cL,i}\eqdist \mA_{m}^\cL$ for all $m$ and $i=1,2$. 

	Applying~\eqref{eq:dm} for $c/4$ instead of $c$ it follows that with probability at least $1 - n^{-\Theta(\log n)}$ we have jointly for $i=1,2$
	\[
		|\mA_{n_i}^{\cL,i} - \beta n_i / 2| \le (c/4) \sqrt{n} \log n 
	\]
	and hence
	\[
		|\mA_{n_1}^{\cL,1}+\mA_{n_2}^{\cL,2} - \beta (n+2) / 2| \le (c/2) \sqrt{n} \log n.
	\]
	This yields
	\begin{align}
		\label{eq:parr2}
		\Prb{ |\mA_n^{\cI} - \beta n/2| \ge c\sqrt{n} \log n} \le n^{-\Theta(\log n)}.
	\end{align}
	Combining~\eqref{eq:par1} and~\eqref{eq:parr2} it follows by $\cN_{\mathrm{s}} = \tilde{\cN} + \cI$ that
	\begin{align*}
		\Prb{ |\mA_n^{\cN_{\mathrm{s}}} - \beta n/2| \ge \sqrt{n} \log n} \le n^{-\Theta(\log n)}.
	\end{align*}
\end{proof}

\begin{lemma}
	\label{le:nosym2}
	The probability that a uniformly at random selected unlabelled simple network with $n$ vertices admits at least one automorphism $\sigma$ of the underlying unrooted cubic graph with $\cyl_1(\sigma)=0$ is bounded by $n^{-\Theta(\log n)}$.
\end{lemma}
\begin{proof}
	Let $\mathfrak{B}_n$ denote the set of $n$-vertex unlabelled (unrooted) connected cubic planar graphs that admit an automorphism $\sigma$ with $\mathrm{cyl}_1(\sigma)=0$. 
	For $G \in \mathfrak{B}_n$, let $r_{\cN_{\mathrm{s}}}(G)$ denote the number of distinct oriented-edge rootings of $G$, and let $r_\cD(G)$ denote the number of such rootings for which the root edge is not an isthmus.
	
	Since $G$ is cubic, it has $3n/2$ edges and hence at most $3n$ oriented edges. Hence
	\[
		r_{\cN_{\mathrm{s}}}(G) \le 3n.
	\]
	On the other hand, $G$ is not a tree and hence contains a cycle, yielding
	\[
		r_\cD(G) \ge 1.
	\]
	Thus
	\begin{align}
		\label{eq:trivi}
		r_{\cN_{\mathrm{s}}}(G) \le 3n r_\cD(G).
	\end{align}
	
	Let $b_n^\cD$ and $b_n^{\cN_{\mathrm{s}}}$ denote the number of unlabelled $\cD$-networks and $\cN_{\mathrm{s}}$-networks with $n$ vertices whose underlying unrooted graph belongs to $\mathfrak{B}_n$. By Corollary~\ref{co:deno} we have
	\[
		b_n^{\cD} \le d_n n^{-\Theta( \log n )}.
	\]
	By~\eqref{eq:trivi} and Lemma~\ref{le:blo1} it follows that
	\begin{align*}
		b_n^{\cN_{\mathrm{s}}} 
		&= \sum_{G \in \mathfrak{B}_n} r_{\cN_{\mathrm{s}}}(G) \\
		&\le 3n \sum_{G \in \mathfrak{B}_n} r_{\cD}(G) \\
		&\le d_n n^{-\Theta( \log n )} \\
		&\le n^{-\Theta(\log n)} [x^n] \cN_{\mathrm{s}}(x).
	\end{align*}
	This completes the proof.
\end{proof}

\subsection{Connected cubic planar graphs}

Throughout this section $n$ is restricted to even integers. Let $\cC$ denote the species of connected  cubic planar graphs and let $c_n = [x^n]\cC(x)$ denote the number of unlabelled graphs in this class with $n$ vertices.

\begin{theorem}
	\label{te:connected}
	Set
	\[
		\gamma = \frac{c_{\cN_{\mathrm{s}}}}{\beta} \cD(\rho) \frac{64 \sqrt{3}}{5 \sqrt{\pi}} \Exb{Z} \frac{(1 - \Exb{L} + \frac{3}{2} \mu \Exb{X})^{3/2}}{(1-\Exb{L})^{5/2}}.
	\]
	We have
	\[
		c_n \sim \gamma n^{-7/2} \rho^{-n}
	\]
	as $n \in 2\ndN$ tends to infinity.
\end{theorem}
\begin{proof}
We may express  $\Sym_{\cC}[n]$ as a disjoint union
\[
\Sym_{\cC}[n] = \Sym_{\cC}^0[n] \cup  \widetilde{\Sym}_{\cC}[n]
\]
with $\Sym_{\cC}^0[n]$ the subset of symmetries with no cylindrical $1$-edge cycles. The Walsh index series of $\cC$ is related to the Walsh index series of the species of simple networks $\cN_{\mathrm{s}}$ via
\begin{align*}
	W_{\cN_{\mathrm{s}}}^+ = 2 \frac{\partial W_{\cC} }{\partial a_1}.
\end{align*}
The factor $2$ corresponds to the two possible orientations of a specified root edge. In particular, any symmetry $(C, \sigma) \in \widetilde{\Sym}_\cC[n]$  corresponds to precisely $2 \cyl_1(\sigma)$  symmetries from $\Sym_{\cN_{\mathrm{s}}}^+[n]$ that are obtained by declaring one of the $\cyl_1(\sigma)$ edges whose endpoints get fixed by $\sigma$ as the root edge and orienting it in one of the two possible ways. Let $\Pi(C, \sigma)$ denote the collection of these rooted versions of $(C, \sigma)$, so that we have a disjoint union
\[
	\Sym_{\cN_{\mathrm{s}}}^+[n] = \bigcup_{ (C, \sigma) \in \widetilde{\Sym}_\cC[n]} \Pi(C, \sigma).
\]
Let us set $s_n = [x^n]\cN_{\mathrm{s}}(x)$. By~Lemma~\ref{le:concentrate} it follows that
\begin{align}
	\label{eq:ph1}
	&\sum_{\substack{(C, \sigma) \in \widetilde{\Sym}_{\cC}[n] \\ |\cyl_1(\sigma) - \beta n /2| \ge \sqrt{n} \log n }} |\Pi(C, \sigma)| \\
	&= |\{(N, \sigma) \in \Sym_{\cN_{\mathrm{s}}}^+[n] \,\mid\, |1 +\cyl_1(\sigma) - \beta n /2| \ge \sqrt{n} \log n \}| \nonumber \\
	&\le  n^{-\Theta(\log n)} s_n n! . \nonumber
\end{align}
We have $|\Pi(C, \sigma)| \ge 1$ for each $(C, \sigma) \in \widetilde{\Sym}_{\cC}[n]$. It follows that
\begin{align}
	\label{eq:ph2}
	|\{(C, \sigma) \in \widetilde{\Sym}_{\cC}[n] \,\mid\, |\cyl_1(\sigma) - \beta n /2| \ge \sqrt{n} \log n \}| \le n^{-\Theta(\log n)} s_n n! .
\end{align}
Furthermore, since $|\Pi(C, \sigma)| = 2\cyl_1(\sigma)$
\begin{align*}
	&\sum_{\substack{(C, \sigma) \in \widetilde{\Sym}_{\cC}[n] \\ |\cyl_1(\sigma) - \beta n /2| < \sqrt{n} \log n }} |\Pi(C, \sigma)| \\
	&= ( \beta n + O(\sqrt{n} \log n) ) |\{(C, \sigma) \in \widetilde{\Sym}_{\cC}[n] \,\mid\, |\cyl_1(\sigma) - \beta n /2| < \sqrt{n} \log n \}|.
\end{align*}
By~\eqref{eq:ph1} it follows that
\begin{align*}
		|\{(C, \sigma) \in \widetilde{\Sym}_{\cC}[n] \,\mid\, |\cyl_1(\sigma) - \beta n /2| < \sqrt{n} \log n \}| &= \frac{s_n n!(1 - n^{-\Theta(\log n)})}{\beta n + O(\sqrt{n} \log n)}.
\end{align*}
Hence, by~\eqref{eq:ph2} 
\begin{align}
	\label{eq:posd1}
	|\widetilde{\Sym}_{\cC}[n]| = \frac{s_n n!(1 - n^{-\Theta(\log n)})}{\beta n + O(\sqrt{n} \log n)}.
\end{align}

We have that $|\Sym_{\cC}^0[n]|$ is bounded by $n!$ times the number of unlabelled connected cubic planar graphs with $n$ vertices that admit at least one automorphism $\sigma$ with $\cyl_1(\sigma)=0$. Any such graph corresponds to at least one simple network. By Lemma~\ref{le:nosym2} it follows that
\begin{align}
	\label{eq:ume}
	|\Sym_{\cC}^0[n]| &\le n! s_n n^{-\Theta(\log n)}.
\end{align}
Together with~\eqref{eq:posd1} and Lemma~\ref{le:blo1}, it follows that
\begin{align*}
	c_n = \frac{s_n}{n \beta}(1 + O(\log(n) / \sqrt{n})).
\end{align*}
Using~\eqref{eq:bista2}  it follows that
\[
	c_n \sim \gamma n^{-7/2} \rho^{-n}
\]
for
\begin{align*}
	\gamma &= \frac{c_{\cN_{\mathrm{s}}}}{\beta} \cD(\rho) \frac{64 \sqrt{3}}{5 \sqrt{\pi}} \Exb{Z} 
	\frac{(1 - \Exb{L} + \frac{3}{2} \mu \Exb{X})^{3/2}}{(1-\Exb{L})^{5/2}}.
\end{align*}
\end{proof}

\begin{corollary}
	\label{co:tv}
	Let $P_n$ denote the uniform law on $\Sym_{\cC}[n]$. Let $Q_n$ denote the law on $\Sym_{\cC}[n]$ obtained by taking a uniform element of $\Sym_{\cN_{\mathrm{s}}}^{+}[n]$ and forgetting the root edge. Then
	\[
		d_{\mathrm{TV}}(P_n, Q_n) \to 0
	\]
	as $n \in 2\ndN$ tends to infinity. Furthermore, uniformly for all events $B \subset \Sym_{\cC}[n]$ 
	\[
		P_n(B) = Q_n(B)(1+ o(1)) + O(n^{-\Theta(\log n)}).
	\]
\end{corollary}
\begin{proof}
	Each symmetry $(C,\sigma)\in \Sym_{\cC}[n]$ with $\cyl_1(\sigma)>0$ gives rise to exactly $2\cyl_1(\sigma)$ elements of $\Sym^+_{\cN_{\mathrm{s}}}[n]$ by choosing one of the $\cyl_1(\sigma)$ cylindrical $1$-cycles and orienting it in one of the two possible ways. Hence, with $c_n = [x^n]\cC(x)$ and $s_n = [x^n] \cN_{\mathrm{s}}(x)$ as before,
	\begin{align}
		\label{eq:c0}
	\frac{Q_n(\{(C, \sigma)\})}{P_n(\{(C,\sigma)\})} = \frac{2\mathrm{cyl}_1(\sigma) c_n}{s_n}
	\end{align}
	for $\mathrm{cyl}_1(\sigma)>0$, and $Q_n(\{(C,\sigma)\})=0$ if $\mathrm{cyl}_1(\sigma)=0$. 
	
	By Theorem~\ref{te:connected},
	\begin{align}
		\label{eq:c1}
		c_n\sim \frac{s_n}{\beta n}.
	\end{align}
	Moreover, Lemma~\ref{le:concentrate} gives, for the event
	\[
		G_n := \{ (C, \sigma) \in \Sym_{\cC}[n] \,\,\mid\,\, |\mathrm{cyl}_1(\sigma)- \beta n / 2| \le \sqrt n\log n\}
	\]
	that
	\begin{align}
		\label{eq:c2}
		Q_n(G_n^{\mathrm{c}})\le n^{-\Theta(\log n)}.
	\end{align}
	On $G_n$ we have by~\eqref{eq:c0} and~\eqref{eq:c1}
	\[
		\frac{Q_n(\{(C, \sigma)\})}{P_n(\{(C,\sigma)\})} = \frac{2\mathrm{cyl}_1(\sigma)}{\beta n}(1+o(1)) = 1 + o(1).
	\]
	Furthermore, for the event
	\[
		E_n = \{ (C, \sigma) \in \Sym_{\cC}[n] \mid \cyl_1(\sigma) > 0\}
	\]
	we have by~\eqref{eq:c0}, ~\eqref{eq:c1} and~\eqref{eq:c2}
	\[
		P_n(G_n^{\mathrm{c}} \cap E_n) \le \frac{s_n}{2c_n}Q_n(G_n^{\mathrm c}) = n^{-\Theta(\log n)}.
	\]
	By~\eqref{eq:ume} and~\eqref{eq:c1}
	\[
		P_n(E_n^{\mathrm{c}}) \le \frac{s_n}{c_n}  n^{-\Theta(\log n)} = n^{-\Theta(\log n)}.
	\]
	Hence we arrive at
	\[
		P_n(G_n^{\mathrm{c}}) \le n^{-\Theta(\log n)}. 
	\]
	Consequently, uniformly for all $B \subset \Sym_{\cC}[n]$
	\begin{align*}
		P_n(B) 	&= P_n(B \cap G_n) + O(n^{-\Theta(\log n)})\\
				&= Q_n(B \cap G_n)(1+o(1)) + O(n^{-\Theta(\log n)}) \\
				&= Q_n(B)(1+o(1)) + O(n^{-\Theta(\log n)}).
	\end{align*}
	In particular,
	\[
		d_{\mathrm{TV}}(P_n,Q_n) \to 0.
	\]
\end{proof}

For each $i \ge 1$ let $\mathfrak{M}_i(n)$ denote the number of vertices of the $i$th largest $3$-connected component in a uniform random $n$-vertex unlabelled connected cubic planar graph.

\begin{theorem}
	\label{te:components}
	For each $i \ge 1$ let $\mathfrak{M}_i(n)$ denote the number of vertices
	of the $i$th largest $3$-connected component in a uniform random
	$n$-vertex unlabelled connected cubic planar graph. Set
	\[
	\chi_{\cC}:=\frac{\chi}{2},
	\qquad
	\omega_{\cC}:=2^{-2/3}\omega .
	\]
	We have uniformly for all even integers
	$h=\chi_{\cC} n + x n^{2/3}$ that
	\[
	\omega_{\cC} n^{2/3}
	\Pr{\mathfrak{M}_1(n)=h}
	=
	\mathfrak h\left(-\frac{x}{2\omega_{\cC}}\right)+o(1).
	\]
	Moreover,
	\[
	\left(n^{-2/3}\mathfrak{M}_j(n)\right)_{j \ge 2}
	\convd
	\left(2^{-2/3}J_i\right)_{i \ge 1}.
	\]
\end{theorem}
\begin{proof}
	By Corollary~\ref{co:tv} it suffices to prove the corresponding statements for the extremal sizes of $3$-connected components in uniform random $n$-vertex unlabelled simple networks.
	
	Recall that
	\[
	\cN_{\mathrm{s}} = \cH + \cP_{\mathrm{s}} + \cS_{\mathrm{s}} + \cI.
	\]
	Since a uniform random unlabelled simple network is a mixture of uniform random unlabelled networks from each of these four classes, it suffices to prove the corresponding statements for the extremal component sizes of these four classes.
	
	The proof of Lemma~\ref{le:maxdeg} extends in a straightforward manner to four modified branching processes, where the root has type $\cH$ instead of $\cD$, or has one of the three modified offspring distributions:
	\begin{align*}
		\phi_{\cP_{\mathrm{s}}} &= \frac{1}{\cP_{\mathrm{s}}(\rho)}\frac{\rho^2 x_{\cX}^2}{2} \left( x_{\cY}^4(\cD(\rho) x_{\cD})^2 + \cD_{2}(\rho) x_{\cD_2} \right) ,\\
		\phi_{\cS_{\mathrm{s}}} &= 	\frac{1}{\cS_{\mathrm{s}}(\rho)} \Big( \left(\cL(\rho) x_{\cL}  + \cP(\rho)  x_{\cP} + \cH(\rho) x_{\cH}\right) x_{\cY} \\
			&\quad\, \cdot \left( \cL(\rho) x_{\cL}  + \cS(\rho) x_{\cS}  + \cP(\rho) x_{\cP} + \cH(\rho) x_{\cH} \right) - \cL(\rho)^2 x_{\cL}^2 x_{\cY}\Big),  \\
		\phi_{\cI} &=  x_{\cL}^2.
	\end{align*}

	The first three of these four modified decorated branching processes conditioned on having $n$ vertices of type $\cX$  correspond to uniform random symmetries from the three types of networks $\cH$, $\cP_{\mathrm{s}}$, and  $\cS_{\mathrm{s}}$. 
	
	Recall that \[
	W_\cI^+ = \frac{1}{s_1^2a_1^2} (W_\cL^+)^2.
	\]
	Hence in the case where the root offspring is determined by $\phi_{\cI}$ we need to condition on $n+2$ vertices instead and be aware that in this case the number of type $\cY$ vertices in the tree corresponds to the number of cylindrical edge $1$-cycles plus $2$.

	In all four cases it follows by identical arguments as in Corollary~\ref{co:dcomponents} that  the corresponding statements for the extremal sizes of the $3$-connected components hold. This completes the proof.
\end{proof}

\subsection{Unconstrained cubic planar graphs}

An unconstrained unlabelled cubic planar graph is a multiset of connected unlabelled cubic planar graphs. Hence the ordinary generating series $\cG(x)$ for the unlabelled graphs in the species $\cG$ of unconstrained  cubic planar graphs satisfies
\[
\cG(x) = \exp\left(\sum_{i \ge 1} \cC(x^i)/i \right).
\]
By Theorem~\ref{te:connected}, $\cC(z)$ belongs to $\mathscr{S}_2$. Hence we may apply a counting result~\cite[Lem. 3.2]{zbMATH07359155} to obtain
\begin{align*}
	g_n &\sim \exp\left(\sum_{i \ge 1} \cC(\rho^i)/i \right) c_n \sim \kappa n^{-7/2} \rho^{-n}
\end{align*}
for
\[
\kappa = \exp\left(\sum_{i \ge 1} \cC(\rho^i)/i \right) \gamma.
\]
This proves Theorem~\ref{te:main}.

By applying~\cite[Thm. 3.1]{zbMATH07359155} we also obtain that a uniformly at random selected unlabelled cubic planar graph with $n$ vertices has a giant connected component of size $n - O_{\mathrm{p}}(1)$ and the fragment obtained by deleting it approaches in distribution a finite limiting graph as $n$ tends to infinity. This implies that the central limit theorem for the largest $3$-connected components (implied by the local limit theorem) and the joint distributional limits for the remaining $3$-connected components in Theorem~\ref{te:components} also hold for uniform unlabelled unconstrained cubic planar graphs with $n$ vertices. By a similar deviation bound as those used in the proof of Theorem~\ref{te:components} the local limit theorem also holds in the unconstrained case. 

\section{Numerical approximations}

\label{sec:numerics}

\subsection{Bounds on symmetries}

In order to improve our numeric estimates  we are going to make the exponential bounds from Proposition~\ref{pro:H} explicit by using the structural analysis  by Brown~\cite{MR168485}.

\begin{proposition}
	\label{pro:brown}
	We have for $n \ge 2$
	\begin{align*}
		[x^{2n}y^{3n-1}] R_1(x,x^2,y,y^2,y) &\le \epsilon_n/2, \\
		[x^{2n}y^{3n-1}] R_2(x,x^2,y,y^2,y) &\le \epsilon_{n-1} + \epsilon_n + \theta_n,
	\end{align*}
	for
	\begin{align*}
		\epsilon_n &= \frac{(3n)!}{n!(2n+1)!}\\
		\theta_n &= \frac{4 (2n-3)!}{(\lfloor n/2\rfloor - 1)! ( \lfloor 3n/2 \rfloor -1)!}.
	\end{align*}
\end{proposition}
\begin{proof}
	In Proposition~\ref{pro:unl3con} we let $q_n'$ denote the number of rooted achiral $3$-connected unlabelled planar graphs with $2n$ vertices.  By Proposition~\ref{pro:unl3con} these are precisely those with exactly one non-trivial root-preserving graph automorphism and $(2n)!/2$ labelled versions, whereas the remaining rooted chiral graphs each have no non-trivial root-preserving graph automorphism. Hence out of all symmetries in $\Sym_\cH^+[2n]$ there are precisely $q_n' (2n)!/2$ that carry a non-trivial automorphism, and the non-trivial automorphism corresponds to an orientation-reversing homeomorphism of the sphere. Hence
	\begin{align}
		\label{eq:b1}
		[x^{2n}y^{3n-1}] R_1(x,x^2,y,y^2,y) = \frac{q_n'}{2}.
	\end{align}
	By the dual map construction, $q_n'$ is also equal to the number of $(n+2)$-vertex simple triangulations that have an automorphism that swaps the poles of the root edge but fixes each of the two adjacent faces. In the notation of Brown~\cite[Sec. 10]{MR168485} the number of such objects is bounded by $K_{n-1,0}$, which by~\cite[Eq. (14.4)]{MR168485} is bounded by $\epsilon_n$. Hence
	\begin{align}
		\label{eq:b2}
		q_n' \le \epsilon_n.
	\end{align}
	Combining~\eqref{eq:b1} and~\eqref{eq:b2}, we obtain
	\[
	[x^{2n}y^{3n-1}] R_1(x,x^2,y,y^2,y) \le \epsilon_n/2.
	\]
	
	As for $R_2$, recall that $[x^{2n}y^{3n-1}] R_2(x,x^2,y,y^2,y)$ counts pole symmetric $\cH$-networks. For such a network we fix a single rooted cubic planar map corresponding to it and form the dual map. In an injective manner this yields a triangulation $T$ with an oriented root edge $e$ such that it admits an automorphism $\varphi$ that swaps the two faces adjacent to the root edge. Let $u$ denote the start vertex of $e$ and $v$ the target vertex of $e$. Note that $\varphi$ either swaps $u$ and $v$ or it fixes both $u$ and $v$. Let $a$ denote the unique third vertex in the face right to $e$ and $b$ the unique third vertex in the face left to $e$. 
	
	We delete $e$, creating a face of degree four that is left invariant by $\varphi$.  The result is a triangulation $T'$ of a $4$-gon. We declare the directed edge $e'$ from $u$ to $a$ as the new root edge. Note that we may recover $(T,e)$ from $(T', e')$.
	
	If $\varphi$ swaps the ends of $u$ and $v$ then $\varphi$  rotates the boundary of the root quadrangle by two turns. In the notation of Brown~\cite[Sec. 6]{MR168485}, the number of objects like $(T', e')$ is bounded by ${}_2E_{n-2,1}$, which by~\cite[Eq. (8.10)]{MR168485} is given by $\theta_n$.
	
	If $\varphi$ fixes both $u$ and $v$ then $\varphi$ reverses the orientation of the boundary of the quadrangle, and leaves the origin of the root edge $e'$ fixed. In the notation of Brown~\cite[Sec. 10]{MR168485}, the number of objects like $(T', e')$ is bounded by $L_{n-2,1}$.
	
	Hence, with this notation
	\begin{align}
		\label{eq:b3}
		[x^{2n}y^{3n-1}] R_2(x,x^2,y,y^2,y) \le \theta_n + L_{n-2,1}.
	\end{align}
	
	A bound for $L_{n-2,1}$ is not explicitly spelled out in~\cite{MR168485}, but may be obtained from the equations of generating series given there. Specifically~\cite[Sec. 10]{MR168485} defines generating series $K(x,y)$, $L(x,y)$ and $J(x,y)$ such that
	\begin{align*}
		K(x,y) &= \sum_{n,m \ge 0} K_{n,m} x^n y^m, \\
		L(x,y) &= \sum_{n,m \ge 0} L_{n,m} x^n y^m,\\
		J(x,y^2) &= \sum_{n,m \ge 0} K_{n,2m} x^n y^{2m} = \sum_{n,m \ge 0} L_{n,2m} x^n y^{2m}.
	\end{align*}
	Brown uses the notation \[
	L_{.1}(x) = \sum_{n \ge 0} L_{n,1} x^n
	\]
	and 
	\[
	J_0(x) = J(x,0) = K(x,0) = L(x,0).
	\]
	By~\cite[Eq. (12.2)]{MR168485} 
	\[
	K(x,0) = 1 + x(L_{.1}(x) - J(x,0)^2).
	\]
	Therefore $J_0(0) = K(0,0) = 1$ and
	\begin{align}
		\label{eq:b4}
		L_{.1}(x) = (J_0(x) -1)/x + J_0(x)^2.
	\end{align}
	By~\cite[Display before (14.4)]{MR168485} we have coefficient-wise domination
	\begin{align}
		\label{eq:b5}
		J_0(x) \le_\mathrm{coeff} J_0(x)^2 \le_\mathrm{coeff} \sum_{n=0}^\infty \epsilon_{n+1}x^n.
	\end{align}
	Combining~\eqref{eq:b4} and~\eqref{eq:b5} we obtain
	\[
	L_{n,1} \le \epsilon_{n+2} + \epsilon_{n+1}.
	\]
	Hence by~\eqref{eq:b3}
	\[
	[x^{2n}y^{3n-1}] R_2(x,x^2,y,y^2,y) \le \epsilon_{n-1} + \epsilon_n + \theta_n.
	\]
	This completes the proof.
\end{proof}

\subsection{Truncations of generating series}

\label{sec:truncations}

The software \texttt{plantri} (version \texttt{5.8}) by Gunnar Brinkmann, Brendan McKay and Heidi Van den Camp  efficiently generates unlabelled $3$-connected cubic planar graphs with a given size.  See~\cite{zbMATH05382101,zbMATH02244630, zbMATH02245027} for papers related to its algorithmic capabilities.  Specifically, 
\[
\text{\texttt{plantri -m1c3 -d -G \{n\}d}  }
\]
generates all unlabelled $3$-connected cubic planar graphs  with $\{n\}$ vertices and internally computes the automorphism group of their dual triangulation. The action of this group on the dual triangulation may be converted to an action on the cubic graph using the dual map formalism. This enables us to compute arbitrarily large truncations of the Walsh index series $W_3(\bm{s}, \bm{a}, \bm{b})$ of $3$-connected cubic planar graphs. Due to \texttt{plantri}'s res/mod functionality this process may  be parallelised well, allowing us to reach high vertex counts.

The corresponding series of $3$-connected cubic networks are given by
\begin{align*}
	W_{3}^+ &= 2 \frac{\partial W_3}{\partial a_1} \\
	W_{3}^- &= 2 \frac{\partial W_3}{\partial b_1},
\end{align*}
with the factor $2$ corresponding to the two possible orientations of a specified root edge in a labelled $3$-connected cubic planar graph. This way, $R_2 = W_3^-$ and $R_1$ may be obtained from $W_3^+$ by discarding the terms corresponding to identity automorphisms, i.e. those without vertex cycles of length $2$. Plugging these truncations into the recursive system of equations for $\cD$-networks yields truncations of the generating series for networks:

\begin{tcolorbox}
	\tiny
\setlength{\jot}{1pt}
\begin{align*}
	\cD(x) &= x^{4} + 4x^{6} + 19x^{8} + 104x^{10} + 620x^{12} + 4026x^{14} + 27894x^{16} + 203162x^{18} + 1537687x^{20} \\
	&\quad + 11995410x^{22} + 95861608x^{24} + 781347958x^{26} + 6474182078x^{28} + 54397619869x^{30} \\
	&\quad + 462575384197x^{32} + 3974833433257x^{34} + 34470094954668x^{36} + 301371748748567x^{38} \\
	&\quad + 2654131825150422x^{40} + 23527938600570547x^{42} + 209804460441511473x^{44} + O(x^{46}), \\[3pt]
	\cD^{-}(x) &= x^{4} + 4x^{6} + 13x^{8} + 44x^{10} + 182x^{12} + 810x^{14} + 4200x^{16} + 24464x^{18} + 159293x^{20} + 1121984x^{22} \\
	&\quad + 8379336x^{24} + 65164438x^{26} + 521969644x^{28} + 4274197843x^{30} + 35610577435x^{32} \\
	&\quad + 300888332649x^{34} + 2572446598612x^{36} + 22216231046627x^{38} + 193560027328890x^{40} \\
	&\quad + 1699553475117967x^{42} + 15026655868163753x^{44} + O(x^{46}), \\[3pt]
	\cL(x) &= x^{6} + 3x^{8} + 13x^{10} + 62x^{12} + 342x^{14} + 2095x^{16} + 14053x^{18} + 100379x^{20} + 752038x^{22} \\
	&\quad + 5833396x^{24} + 46478053x^{26} + 378202552x^{28} + 3130855522x^{30} + 26292033253x^{32} \\
	&\quad + 223504815431x^{34} + 1920153657373x^{36} + 16649581158677x^{38} + 145554662204890x^{40} \\
	&\quad + 1281802023943294x^{42} + 11362274299117646x^{44} + O(x^{46}), \\[3pt]
	\cS(x) &= x^{8} + 8x^{10} + 53x^{12} + 348x^{14} + 2329x^{16} + 16148x^{18} + 116192x^{20} + 864850x^{22} + 6628261x^{24} \\
	&\quad + 52072566x^{26} + 417712440x^{28} + 3410514978x^{30} + 28269678992x^{32} + 237402111144x^{34} \\
	&\quad + 2016447481073x^{36} + 17299604330959x^{38} + 149741737920037x^{40} + 1306463523050961x^{42} \\
	&\quad + 11480284741511714x^{44} + O(x^{46}), \\[3pt]
	\cS^{-}(x) &= x^{8} + 5x^{12} + 4x^{14} + 35x^{16} + 60x^{18} + 348x^{20} + 1058x^{22} + 5825x^{24} + 28880x^{26} + 184862x^{28} \\
	&\quad + 1240912x^{30} + 9146194x^{32} + 70210146x^{34} + 559121049x^{36} + 4558123213x^{38} + 37869567837x^{40} \\
	&\quad + 319284517553x^{42} + 2725214180998x^{44} + O(x^{46}), \\[3pt]
	\cP(x) &= x^{6} + 4x^{8} + 20x^{10} + 108x^{12} + 649x^{14} + 4206x^{16} + 29120x^{18} + 211644x^{20} + 1598925x^{22} \\
	&\quad + 12451122x^{24} + 99353143x^{26} + 808751290x^{28} + 6693770858x^{30} + 56188508181x^{32} \\
	&\quad + 477404199016x^{34} + 4099237882762x^{36} + 35525787924956x^{38} + 310421127532837x^{40} \\
	&\quad + 2732400454975305x^{42} + 24210327527652583x^{44} + O(x^{46}), \\[3pt]
	\cP^{-}(x) &= x^{6} + 4x^{8} + 14x^{10} + 48x^{12} + 205x^{14} + 906x^{16} + 4652x^{18} + 26574x^{20} + 170119x^{22} + 1181786x^{24} \\
	&\quad + 8743597x^{26} + 67573846x^{28} + 539078604x^{30} + 4402328757x^{32} + 36609144608x^{34} \\
	&\quad + 308902214764x^{36} + 2638209030084x^{38} + 22765349391939x^{40} + 198210140466377x^{42} \\
	&\quad + 1739396012622501x^{44} + O(x^{46}), \\[3pt]
	\cH(x) &= x^{4} + 2x^{6} + 11x^{8} + 63x^{10} + 397x^{12} + 2687x^{14} + 19264x^{16} + 143841x^{18} + 1109472x^{20} \\
	&\quad + 8779597x^{22} + 70948829x^{24} + 583444196x^{26} + 4869515796x^{28} + 41162478511x^{30} \\
	&\quad + 351825163771x^{32} + 3036522307666x^{34} + 26434255933460x^{36} + 231896775333975x^{38} \\
	&\quad + 2048414297492658x^{40} + 18207272598600987x^{42} + 162751573873229530x^{44} + O(x^{46}), \\[3pt]
	\cH^{-}(x) &= x^{4} + 2x^{6} + 5x^{8} + 17x^{10} + 67x^{12} + 259x^{14} + 1164x^{16} + 5699x^{18} + 31992x^{20} + 198769x^{22} \\
	&\quad + 1358329x^{24} + 9913908x^{26} + 76008384x^{28} + 603022805x^{30} + 4907069231x^{32} + 40704162464x^{34} \\
	&\quad + 342831605426x^{36} + 2923882734653x^{38} + 25202146164224x^{40} + 219222026190743x^{42} \\
	&\quad + 1922260342242608x^{44} + O(x^{46}).
\end{align*}
\end{tcolorbox}

We omit the obtained expressions for $R_1$ and $R_2$ since they would fill several pages. As a byproduct, we obtain the following truncation for the ordinary generating series $\cW(x)$ of $3$-connected unlabelled planar graphs:

\begin{tcolorbox}
	\tiny
\begin{align*}
		\cW(x)
		&= x^4 + x^6 + 2x^8 + 5x^{10} + 14x^{12} + 50x^{14}
		+ 233x^{16} + 1249x^{18} 
		+ 7595x^{20} + 49566x^{22} + 339722x^{24} \\
		&\quad+ 2406841x^{26} 
		+ 17490241x^{28} + 129664753x^{30}
		+ 977526957x^{32} 
		+ 7475907149x^{34}
		+ 57896349553x^{36}\\
		&+ 453382272049x^{38} 
		+ 3585853662949x^{40}
		+ 28615703421545x^{42} 
		+ 230214772514389x^{44}
		+ O(x^{46}).
\end{align*}
\end{tcolorbox}

\subsection{Numeric bounds for networks and connected cubic planar graphs}

For each $k \ge 4$, let   \[
\cD^{\langle k \rangle} = \sum_{i=4}^k d_i x^i .
\]  
We will use similar notation for truncations of other power series.
The series $\cH(x)$ dominates coefficient-wise the series $M(x, 1 + \cD^{\langle k \rangle}(x))$. Hence with $\rho_k$ denoting the radius of convergence of the latter, we have
\[
\rho \le \rho_k.
\]
Since $\cQ(z)$ has radius of convergence $27/256$, the radius of convergence $\rho_k$  is the unique positive solution to the equation
\[
x^2(1+\cD^{\langle k \rangle}(x))^3 = 27/256.
\]
Since $\rho^2(1 + \cD(\rho))^3= 27/256$ by Lemma~\ref{le:keynumeric}, it follows that
\[
\cD(\rho)  = \rho^{-2/3} (27/256)^{1/3}-1 \ge \rho_k^{-2/3} (27/256)^{1/3}-1 = \cD^{\langle k \rangle}(\rho_k).
\]
Clearly \[
\cD^{\langle k \rangle}(\rho_k) \ge \cD^{\langle k \rangle}(\rho),
\]
hence it follows that 
\[
\lim_{k \to \infty} \cD^{\langle k \rangle}(\rho_k) = \cD(\rho)
\]
and consequently
\[
\lim_{k \to \infty} \rho_k = \rho.
\]
This yields the following table:
\begin{table}[H]
	\centering
	\begin{tabular}{r|l}
		$k$ & $\rho_k$  \\
		\hline 
		20 & 0.315951739471170 \\
		24 & 0.315867873991947 \\
		28 & 0.315813199766632 \\
		32 & 0.315775205290344 \\
		36 & 0.315747538826359 \\ 
		40 & 0.315726657018910 \\
		44 & 0.315710440451692
	\end{tabular}
\end{table}

Using the rate of decay of these rigorous upper bounds we may employ differential approximants to non-rigorously \emph{guess} that $\rho$ should equal approximately $0.315606$. But the following result gives actual lower and upper bounds.

\begin{theorem}
	\label{te:numeric}
	We have
	\begin{alignat*}{2}
		0.315602932219633 &\le \rho &\le 0.315608352382852, \\
		0.019238001429548 &\le \cD(\rho) &\le 0.019249670988456, \\
		0.015983988309684 &\le \cD^-(\rho) &\le 0.015997675822097, \\
		0.000102531061343 &\le \cD_2(\rho) &\le 0.000102545445050, \\
		0.001593442610593 &\le \cL(\rho) &\le 0.001601495258377, \\
		0.000001006922857 &\le \cL_2(\rho) &\le 0.000001007132574, \\
		0.000363115090375 &\le \cS(\rho) &\le 0.000363551584770, \\
		0.000000010511540 &\le \cS_2(\rho) &\le 0.000000010514491, \\
		0.001939743481319 &\le \cP(\rho) &\le 0.001940995582173, \\
		0.001609918801599 &\le \cP^-(\rho) &\le 0.001611360012121, \\
		0.015341700247260 &\le \cH(\rho) &\le 0.015343628563137, \\
		0.012676467663415 &\le \cH^-(\rho) &\le 0.012680645297464.
	\end{alignat*}
\end{theorem}
\begin{proof}
Recall that
\[
	\vartheta = (\rho,\rho^2, 1 + \cD(\rho), 1 + \cD_2(\rho), 1 + \cD^-(\rho)).
\]
Suppose that at discrete time $i\ge 0$ we are given nonnegative bounds
\begin{alignat*}{2}
	\rho_{\mathrm{l}}[i] &\le \rho &\le \rho_{\mathrm{u}}[i], \\
	D_{\mathrm{l}}[i] &\le \cD(\rho) &\le D_{\mathrm{u}}[i], \\
	D_{\mathrm{l}}^-[i] &\le \cD^-(\rho) &\le D_{\mathrm{u}}^-[i], \\
	D_{\mathrm{l}}^2[i] &\le \cD_2(\rho) &\le D_{\mathrm{u}}^2[i], \\
	L_{\mathrm{l}}[i] &\le \cL(\rho) &\le L_{\mathrm{u}}[i], \\
	L_{\mathrm{l}}^2[i] &\le \cL_2(\rho) &\le L_{\mathrm{u}}^2[i], \\
	S_{\mathrm{l}}[i] &\le \cS(\rho) &\le S_{\mathrm{u}}[i], \\
	S_{\mathrm{l}}^2[i] &\le \cS_2(\rho) &\le S_{\mathrm{u}}^2[i], \\
	P_{\mathrm{l}}[i] &\le \cP(\rho) &\le P_{\mathrm{u}}[i], \\
	P_{\mathrm{l}}^-[i] &\le \cP^-(\rho) &\le P_{\mathrm{u}}^-[i], \\
	H_{\mathrm{l}}[i] &\le \cH(\rho) &\le H_{\mathrm{u}}[i], \\
	H_{\mathrm{l}}^-[i] &\le \cH^-(\rho) &\le H_{\mathrm{u}}^-[i],
\end{alignat*}
as well as
\begin{alignat*}{2}
	R_{\mathrm{l}}^1[i] &\le R_1(\vartheta) &\le R_{\mathrm{u}}^1[i].
\end{alignat*}
We are going to construct the next iteration of bounds. From~\eqref{eq:sthat} we have
\[
	\cS(\rho) = \frac{\cD^2(\rho)}{1 + \cD(\rho)}.
\]
The function $x \mapsto x^2/(1+x)$ is increasing for $x>0$. Hence we may set
\begin{align*}
	S_{\mathrm{l}}[i+1] &=  (D_{\mathrm{l}}[i])^2 / (1 + D_{\mathrm{l}}[i]), \\
	S_{\mathrm{u}}[i+1] &=  (D_{\mathrm{u}}[i])^2 / (1 + D_{\mathrm{u}}[i]).
\end{align*}

We set $k=44$, since we have obtained truncations of generating series up to this order in Section~\ref{sec:truncations}. We have
\[
\cD_2(\rho) = \cD(\rho^2) = \sum_{i=1}^{k} d_i \rho^{2i} + E = \cD^{\langle k \rangle}(\rho^2) + E
\]
with
\[
0 \le E := \rho^{k+1} \sum_{i \ge k+1} d_i \rho^{2i - (k+1) } \le \rho^{k+1} \cD(\rho).
\]
Hence we may set
\begin{align*}
	D_{\mathrm{l}}^2[i+1] &= \cD^{\langle k \rangle}((\rho_{\mathrm{l}}[i])^2), \\
	D_{\mathrm{u}}^2[i+1] &= \cD^{\langle k \rangle}((\rho_{\mathrm{u}}[i])^2) + (\rho_{\mathrm{u}}[i])^{k+1} D_{\mathrm{u}}[i].
\end{align*}
By analogous arguments we may bound $\cS_2(\rho)$ and $\cL_2(\rho)$, allowing us to set
\begin{align*}
	S_{\mathrm{l}}^2[i+1] &= \cS^{\langle k \rangle}((\rho_{\mathrm{l}}[i])^2), \\
	S_{\mathrm{u}}^2[i+1] &= \cS^{\langle k \rangle}((\rho_{\mathrm{u}}[i])^2) + (\rho_{\mathrm{u}}[i])^{k+1} S_{\mathrm{u}}[i+1], \\
	L_{\mathrm{l}}^2[i+1] &= \cL^{\langle k \rangle}((\rho_{\mathrm{l}}[i])^2), \\
	L_{\mathrm{u}}^2[i+1] &= \cL^{\langle k \rangle}((\rho_{\mathrm{u}}[i])^2) + (\rho_{\mathrm{u}}[i])^{k+1} L_{\mathrm{u}}[i].
\end{align*}
By~\eqref{eq:lroot} we have
\[
	\cL(\rho) = 1 + \rho^2 - \sqrt{1 + 2 \rho^2 + \rho^4 - \cL_2 - \rho^2(\cD(\rho) + \cD^-(\rho) )}.
\]
Hence we may set
\begin{align*}
	L_{\mathrm{l}}[i+1] &= 1 + (\rho_{\mathrm{l}}[i])^2 \\
	&\quad- \sqrt{1 + 2 (\rho_{\mathrm{u}}[i])^2 + (\rho_{\mathrm{u}}[i])^4 - L_{\mathrm{l}}^2[i+1] - (\rho_{\mathrm{l}}[i])^2(D_{\mathrm{l}}[i] + D_{\mathrm{l}}^-[i] )},\\
	L_{\mathrm{u}}[i+1] &= 1 + (\rho_{\mathrm{u}}[i])^2 \\
	&\quad- \sqrt{1 + 2 (\rho_{\mathrm{l}}[i])^2 + (\rho_{\mathrm{l}}[i])^4 - L_{\mathrm{u}}^2[i+1] - (\rho_{\mathrm{u}}[i])^2(D_{\mathrm{u}}[i] + D_{\mathrm{u}}^-[i] )}.
\end{align*}
Of course, this is only valid if the argument of the square root stays positive and the resulting bounds stay positive. If any of these required  conditions fails, we abort the iteration.

From
\[
	\cP= x^2\cD + \frac{x^2}{2}(\cD^2 + \cD_2)
\]
it follows that we may set
\begin{align*}
	P_{\mathrm{l}}[i+1] &= (\rho_{\mathrm{l}}[i])^2 D_{\mathrm{l}}[i] + \frac{(\rho_{\mathrm{l}}[i])^2}{2}(  (D_{\mathrm{l}}[i])^2 +  D_{\mathrm{l}}^2[i+1]), \\
	P_{\mathrm{u}}[i+1] &= (\rho_{\mathrm{u}}[i])^2 D_{\mathrm{u}}[i] + \frac{(\rho_{\mathrm{u}}[i])^2}{2}(  (D_{\mathrm{u}}[i])^2 +  D_{\mathrm{u}}^2[i+1]).
\end{align*}
By Proposition~\ref{pro:brown}, Lemma~\ref{le:keynumeric} and the fact that $\cD_2(x) \le_{\mathrm{coeff}} \cD(x)^2$ and $\cD^-(x) \le_{\mathrm{coeff}} \cD(x)$ we may write
\begin{align}
	\label{eq:aous}
	R_1(\vartheta)  =  R_1^{\langle k \rangle}(\vartheta) + E_1
\end{align}
with $E_1 \ge 0$ satisfying
\begin{align*}
	E_1 &\le \frac{1}{1 + \cD(\rho)} \sum_{j \ge \lceil(k+1)/2\rceil} \left(\frac{27}{256} \right)^j [x^{2j}y^{3j-1}] R_1(x,x^2,y,y^2,y) \\
		&\le \frac{1}{1 + \cD(\rho)} \sum_{j \ge \lceil(k+1)/2\rceil}\left(\frac{27}{256} \right)^j \epsilon_j/2.
\end{align*}
Hence we may set
\begin{align*}
	R_{\mathrm{l}}^1[i+1] &= R_1^{\langle k \rangle}(\rho_{\mathrm{l}}[i], (\rho_{\mathrm{l}}[i])^2, 1 + D_{\mathrm{l}}[i], 1 + D_{\mathrm{l}}^2[i+1], 1 + D_{\mathrm{l}}^-[i]), \\
	R_{\mathrm{u}}^1[i+1] &= R_1^{\langle k \rangle}(\rho_{\mathrm{u}}[i], (\rho_{\mathrm{u}}[i])^2, 1 + D_{\mathrm{u}}[i], 1 + D_{\mathrm{u}}^2[i+1], 1 + D_{\mathrm{u}}^-[i])  \\
	&\quad+  \frac{1}{1 + D_{\mathrm{l}}[i]} \sum_{j \ge \lceil(k+1)/2\rceil}\left(\frac{27}{256} \right)^j \epsilon_j/2.
\end{align*}
Recall by Corollary~\ref{co:mcomp}
\begin{align*}
	M(\rho, 1+ \cD(\rho)) =  \frac{1}{1 +\cD(\rho)} \frac{5}{512}.
\end{align*}
This allows us to set
\begin{align*}
	H_{\mathrm{l}}[i+1] &= R_{\mathrm{l}}^1[i+1] +  \frac{1}{1 + D_{\mathrm{u}}[i]} \frac{5}{512}, \\
	H_{\mathrm{u}}[i+1] &= R_{\mathrm{u}}^1[i+1] +  \frac{1}{1 + D_{\mathrm{l}}[i]} \frac{5}{512}.
\end{align*}
We may now set
\begin{align*}
	D_{\mathrm{l}}[i+1] &= L_{\mathrm{l}}[i+1] + S_{\mathrm{l}}[i+1] + P_{\mathrm{l}}[i+1] + H_{\mathrm{l}}[i+1], \\
	D_{\mathrm{u}}[i+1] &= L_{\mathrm{u}}[i+1] + S_{\mathrm{u}}[i+1] + P_{\mathrm{u}}[i+1] + H_{\mathrm{u}}[i+1].
\end{align*}
From
\[
	\rho^2(1 +\cD(\rho))^3 = 27/256
\]
we obtain
\begin{align*}
	\rho_{\mathrm{l}}[i+1] &= \sqrt{\frac{1}{(1 + D_{\mathrm{u}}[i+1])^3} 27/256}, \\
	\rho_{\mathrm{u}}[i+1] &= \sqrt{\frac{1}{(1 + D_{\mathrm{l}}[i+1])^3} 27/256}.
\end{align*}
With
\[
\cP^-(\rho) = \rho^2 \cD^-(\rho) + \frac{\rho^2}{2}( (\cD^-(\rho))^2 + \cD_2(\rho)), 
\]
we set
\begin{align*}
	P_{\mathrm{l}}^-[i+1] &= (\rho_{\mathrm{l}}[i+1])^2 D_{\mathrm{l}}^-[i] + \frac{(\rho_{\mathrm{l}}[i+1])^2}{2}( (D_{\mathrm{l}}^-[i])^2 + D_{\mathrm{l}}^2[i+1]), \\
	P_{\mathrm{u}}^-[i+1] &= (\rho_{\mathrm{u}}[i+1])^2 D_{\mathrm{u}}^-[i] + \frac{(\rho_{\mathrm{u}}[i+1])^2}{2}( (D_{\mathrm{u}}^-[i])^2 + D_{\mathrm{u}}^2[i+1]).
\end{align*}
From
\[
	\cS^-(\rho) = (1 + \cD^-(\rho)) (\cD_2(\rho) - \cS_2(\rho))
\]
we obtain
\begin{align*}
	S_{\mathrm{l}}^-[i+1] &= (1 + D_{\mathrm{l}}^-[i])(D_{\mathrm{l}}^2[i+1] - S_{\mathrm{u}}^2[i+1]), \\
	S_{\mathrm{u}}^-[i+1] &= (1 + D_{\mathrm{u}}^-[i])(D_{\mathrm{u}}^2[i+1] - S_{\mathrm{l}}^2[i+1]).
\end{align*}
If the proposed lower bound becomes negative, we abort the iteration.

Applying Proposition~\ref{pro:brown} and Lemma~\ref{le:keynumeric} in a fully analogous manner as for $H_{\mathrm{l}}[i+1]$ and  $H_{\mathrm{u}}[i+1]$ allows us to set
\begin{align*}
	H_{\mathrm{l}}^-[i+1] &= R_2^{\langle k \rangle}(\rho_{\mathrm{l}}[i+1], (\rho_{\mathrm{l}}[i+1])^2, 1 + D_{\mathrm{l}}[i+1], 1 + D_{\mathrm{l}}^2[i+1], 1 + D_{\mathrm{l}}^-[i]), \\
	H_{\mathrm{u}}^-[i+1] &= R_2^{\langle k \rangle}(\rho_{\mathrm{u}}[i+1], (\rho_{\mathrm{u}}[i+1])^2, 1 + D_{\mathrm{u}}[i+1], 1 + D_{\mathrm{u}}^2[i+1], 1 + D_{\mathrm{u}}^-[i])  \\
	&\quad+  \frac{1}{1 + D_{\mathrm{l}}[i+1]} \sum_{j \ge \lceil(k+1)/2\rceil}\left(\frac{27}{256} \right)^j (\epsilon_{j-1} + \epsilon_j + \theta_j).
\end{align*}
Finally,
\begin{align*}
	D_{\mathrm{l}}^-[i+1] &= L_{\mathrm{l}}[i+1] + S_{\mathrm{l}}^-[i+1] + P_{\mathrm{l}}^-[i+1] + H_{\mathrm{l}}^-[i+1], \\
	D_{\mathrm{u}}^-[i+1] &= L_{\mathrm{u}}[i+1] + S_{\mathrm{u}}^-[i+1] + P_{\mathrm{u}}^-[i+1] + H_{\mathrm{u}}^-[i+1].
\end{align*}
This completes the description of the recursion. 

By Corollary~\ref{co:sharpy} and the fact $\cD^-(\rho) \le \cD(\rho)$, we may initialise this procedure with the start values
\begin{align*}
	\rho_{\mathrm{l}}[0] &=0.307, \\
	\rho_{\mathrm{u}}[0] &=\rho_k, \\
	D_{\mathrm{l}}[0] &= \cD^{\langle k \rangle}(\rho_k), \\
	D_{\mathrm{u}}[0] &= 0.038, \\
	D_{\mathrm{u}}^-[0]&=0.038, \\
	L_{\mathrm{u}}[0]&=0.012.
\end{align*}
The remaining start values for the lower bounds are set to zero. The remaining start values for the upper bounds do not get used by the procedure, and may remain in an arbitrary starting state for $i=0$. (At time $i+1$ these bounds get computed from other time $i$ values then they are used directly before the step $i+1$ of the procedure terminates.) Running the procedure from $i=0$ until $i=200$ yields the stated bounds. 
\end{proof}

\begin{corollary}
	\label{co:padi}
We have
	\begin{alignat*}{2}
		0.037018252397683 &\le \partial_1 R_1(\vartheta) &\le 0.037227709665557, \\
		0.067168778620222 &\le \partial_2 R_1(\vartheta) &\le 0.067500626434639, \\
		0.000079495165998 &\le \partial_3 R_1(\vartheta) &\le 0.000174249520716, \\
		0.012689059335935 &\le \partial_4 R_1(\vartheta) &\le 0.012736757392454, \\
		0.006273159247075 &\le \partial_5 R_1(\vartheta) &\le 0.006368412933331, \\
		0.040518269261126 &\le \partial_1 R_2(\vartheta) &\le 0.040999840264280, \\
		0.223778904704424 &\le \partial_2 R_2(\vartheta) &\le 0.224551628314946, \\
		0.006273159247075 &\le \partial_3 R_2(\vartheta) &\le 0.006492964428157, \\
		0.029525024297009 &\le \partial_4 R_2(\vartheta) &\le 0.029635534579018, \\
		0.007799088274125 &\le \partial_5 R_2(\vartheta) &\le 0.008019097678712.
	\end{alignat*}
\end{corollary}
\begin{proof}
	Set $k=44$. Analogous to~\eqref{eq:aous} we have for $1 \le i \le 5$
	\begin{align*}
		\partial_i R_1(\vartheta)  =  \partial_i R_1^{\langle k \rangle}(\vartheta) + E_1(i)
	\end{align*}
	with $E_1(i) \ge 0$ satisfying by Proposition~\ref{pro:brown} 
	\begin{align*}
			E_1(1) &\le \frac{1}{\rho(1 + \cD(\rho))}  \sum_{j \ge \lceil(k+1)/2\rceil} \left(\frac{27}{256} \right)^j 2j [x^{2j}y^{3j-1}] R_1(x,x^2,y,y^2,y) \\
			&\le \sqrt{(1 + \cD(\rho))27/256} \sum_{j \ge \lceil(k+1)/2\rceil}\left(\frac{27}{256} \right)^{j-1} j \epsilon_j, \\
			E_1(2) &\le (1 + \cD(\rho))^2  \sum_{j \ge \lceil(k+1)/2\rceil} \left(\frac{27}{256} \right)^{j-1} j \epsilon_j/2,
	\end{align*}
	and likewise
	\begin{align*}
		E_1(3) &\le \frac{1}{(1 + \cD(\rho))^2} \sum_{j \ge \lceil(k+1)/2\rceil}\left(\frac{27}{256} \right)^{j} (3j-1)  \epsilon_j / 2,\\
		E_1(4) &\le \frac{1}{(1 + \cD(\rho))^3} \sum_{j \ge \lceil(k+1)/2\rceil}\left(\frac{27}{256} \right)^{j} (3j-1)  \epsilon_j / 4,\\
		E_1(5) &\le \frac{1}{(1 + \cD(\rho))^2} \sum_{j \ge \lceil(k+1)/2\rceil}\left(\frac{27}{256} \right)^{j} (3j-1)  \epsilon_j / 2.
	\end{align*}
	Analogously, for $1  \le i \le 5$
	\begin{align*}
		\partial_i R_2(\vartheta)  =  \partial_i R_2^{\langle k \rangle}(\vartheta) + E_2(i)
	\end{align*}
	with $E_2(i) \ge 0$ satisfying 
	\begin{align*}
		E_2(1) &\le \sqrt{(1 + \cD(\rho))27/256} \sum_{j \ge \lceil(k+1)/2\rceil}\left(\frac{27}{256} \right)^{j-1} 2 j(\epsilon_{j-1} + \epsilon_j + \theta_j), \\
		E_2(2) &\le (1 + \cD(\rho))^2 \sum_{j \ge \lceil(k+1)/2\rceil}\left(\frac{27}{256} \right)^{j-1} j(\epsilon_{j-1} + \epsilon_j + \theta_j), \\
		E_2(3) &\le \frac{1}{(1 + \cD(\rho))^2} \sum_{j \ge \lceil(k+1)/2\rceil}\left(\frac{27}{256} \right)^{j} (3j-1)(\epsilon_{j-1} + \epsilon_j + \theta_j), \\
		E_2(4) &\le \frac{1}{(1 + \cD(\rho))^3} \sum_{j \ge \lceil(k+1)/2\rceil}\left(\frac{27}{256} \right)^{j} (3j-1)(\epsilon_{j-1} + \epsilon_j + \theta_j)/2, \\
		E_2(5) &\le \frac{1}{(1 + \cD(\rho))^2} \sum_{j \ge \lceil(k+1)/2\rceil}\left(\frac{27}{256} \right)^{j} (3j-1)(\epsilon_{j-1} + \epsilon_j + \theta_j).
	\end{align*}
	Plugging in the bounds from Theorem~\ref{te:numeric} yields the stated bounds for the partial derivatives of $R_1$ and $R_2$.
\end{proof}

\begin{corollary}
	\label{co:ssi}
 	We have
	\begin{alignat*}{2}
		0.002653575206648 &\le \cD_2'(\rho) &\le 0.002653905861533, \\
		0.000038688165463 &\le \cL_2'(\rho) &\le 0.000038695587969, \\
		0.002653031142560 &\le \cD_2'(\rho) - \cS_2'(\rho) &\le 0.002653361653329.
	\end{alignat*}
\end{corollary}
\begin{proof}
	Note that $[x^n] \cS(x) \le [x^n] \cD(x)$ for all $n$.
	The lower bounds are obtained by evaluating the derivatives of the truncated generating series by the lower bound of $\rho$ obtained in Theorem~\ref{te:numeric}. 
	
	For the upper bounds, set $k=44$ and $K=k+2$. Let
	\[
		E = \sum_{m \ge K} 2m d_{m} \rho^{2m-1}.
	\]
	This way,
	\[
		\cD_2'(\rho) = \sum_{m=4}^k 2m d_m \rho^{2m-1} + E.
	\]
	By the upper bound for $\rho$ in Theorem~\ref{te:numeric}, we have for $m \ge K$
	\[
		\frac{(m+1)\rho^{m+1}}{m\rho^m} = \left(1+\frac1m\right)\rho \le \left(1+\frac1K\right)\rho < 1.
	\]
	Hence the sequence $m\rho^m$ is decreasing for $m\ge K$. Therefore
	\begin{align*}
		E &= \sum_{m\ge K} 2m d_{m} \rho^{2m-1} \\
		&\le 2 K \rho^{K-1} \sum_{m\ge K} d_{m}\rho^m \\
		&= 2 K \rho^{K-1} \left(\cD(\rho) - \cD^{\langle k \rangle}(\rho) \right).
	\end{align*}
	Hence with $\rho_{\mathrm{l}} \le \rho \le \rho_{\mathrm{u}}$ and $\cD(\rho) \le D_{\mathrm{u}}$ the bounds from Theorem~\ref{te:numeric} we obtain
	\begin{align*}
		\cD_2'(\rho) &\le \sum_{m=4}^k 2m d_m \rho^{2m-1}  + 2 K \rho^{K-1} \left(\cD(\rho) - \cD^{\langle k \rangle}(\rho) \right) \\
		&\le \sum_{m=4}^k 2m d_m \rho_{\mathrm{u}}^{2m-1}  + 2 K \rho_{\mathrm{u}}^{K-1} \left(D_{\mathrm{u}} - \cD^{\langle k \rangle}(\rho_{\mathrm{l}}) \right).
	\end{align*}
	The remaining upper bounds for $\cL_2'(\rho)$ and $\cD_2'(\rho) - \cS_2'(\rho)$ may be obtained analogously.
\end{proof}

\begin{corollary}
	\label{co:constants}
	We have
	\begin{alignat*}{2}
		0.045356936670661 &\le \Delta &\le 0.045358968590721, \\
		0.018874670237461 &\le \mu &\le 0.018886335636483, \\
		0.245963056274748 &\le \Ex{L} &\le 0.246237577296421, \\
		0.533332604696564 &\le \Ex{Z} &\le 0.533960972759706, \\
		6.475011097668019 &\le \Ex{X} &\le 6.491809765761764, \\
		6.854400539306733 &\le \Ex{Y} &\le 6.869460478394766, \\
		0.060370180161665 &\le \alpha &\le 0.060463169887492, \\
		2.870539413038735 &\le \beta &\le 2.874588048549916, \\
		0.803865622664150 &\le \chi/2 &\le 0.804428352762890, \\
		0.061061383492303 &\le c_{\cD^-} &\le 0.061340254776674, \\
		0.048126493974744 &\le c_{\cL} &\le 0.048141299280483, \\
		0.853636457272574 &\le c_{\cN_{\mathrm{s}}} &\le 0.853662171949282, \\
		0.070022335988452 &\le \gamma &\le 0.070409625852348.
	\end{alignat*}
\end{corollary}
\begin{proof}
	The expressions
	\[
		\Delta = \frac{\rho^2}{2(1+\rho^2-\cL(\rho))}
	\]
	and
	\[
		\mu = \frac{\cD(\rho)}{1 + \cD(\rho)}
	\]
	may be bounded directly using the bounds of Theorem~\ref{te:numeric}. Next, we use the linear system from Lemma~\ref{le:le3}. With
	\[
		A=	\begin{pmatrix} a_{11} & a_{12}\\ a_{21} & a_{22} \end{pmatrix}, \qquad J=a_{11} a_{22} - a_{21} a_{12},
	\]
	Lemma~\ref{le:le3} gives
	\begin{align*}
		\Ex{L} &= J^{-1}(a_{22} b_1 - a_{12} b_2) \\
		\Ex{Z} &= J^{-1}a_{22}\tilde{b}_1, \\
		\Ex{X} &= J^{-1}(a_{22}\hat{b}_1 - a_{12} \hat{b}_2), \\
		\Ex{Y} &= J^{-1}(a_{22}\check{b}_1 - a_{12} \check{b}_2).
	\end{align*}
	The coefficients of $A$ and the expressions for $b_1, b_2, \tilde{b}_1, \hat{b}_1, \hat{b}_2, \check{b}_1, \check{b}_2$ may be computed using the bounds from Theorem~\ref{te:numeric}, Corollary~\ref{co:padi} and Corollary~\ref{co:ssi}. This yields the stated intervals for $\Ex{L}$, $\Ex{Z}$, $\Ex{X}$ and $\Ex{Y}$.
	
	This allows us to compute
	\begin{align*}
		\alpha &= \frac{3\mu}{1 - \Ex{L} + \frac{3}{2} \mu \Ex{X}}, \\
		\beta &= \alpha \Ex{Y} + 3(1+\mu)\left(1 - \frac{\alpha}{2} \Ex{X} \right), \\
		\frac{\chi}{2} &= \frac{1-\Ex{L}} {1 - \Ex{L} + \frac{3}{2} \mu \Ex{X}}.
	\end{align*}
	For $\chi/2$, we use monotonicity of the last expression in the ``variables'' $\Ex{L},\mu,$ and $\Ex{X}$.
	
	It is decreasing in each of these variables on the relevant interval, so the lower bound is obtained from the upper endpoints of $\Ex{L},\mu,\Ex{X}$, and the upper bound from their lower endpoints.
		
	Next, we compute the simple-network constant from Lemma~\ref{le:blo1}. Namely,
	\begin{align*}
		c_{\cD^-} &=
		\frac{\Delta + \partial_3R_2(\vartheta)}{1 - \Delta - \cD_2(\rho) + \cS_2(\rho) - \rho^2(1 + \cD^-(\rho)) - \partial_5 R_2(\vartheta)}, \\
		c_{\cL} &= \Delta(1+c_{\cD^-}), \\
		c_{\cN_{\mathrm{s}}} &= 1 - \rho^2 - c_{\cL} + 2(\rho^{-2} - 1) c_{\cL} \cL(\rho).
	\end{align*}
	Substituting the resulting intervals into the expression
	\[
		\gamma = \frac{c_{\cN_{\mathrm{s}}}}{\beta} \cD(\rho) \frac{64 \sqrt{3}}{5 \sqrt{\pi}} \Exb{Z} \frac{(1 - \Exb{L} + \frac{3}{2} \mu \Exb{X})^{3/2}}{(1-\Exb{L})^{5/2}}
	\]
	from Theorem~\ref{te:connected}	yields the stated bounds for $\gamma$ and finalises the proof.
\end{proof}

\subsection{Numeric bounds for unconstrained cubic planar graphs}

In order to obtain numeric estimates for $\kappa$ we need bounds on $\cC(\rho^i)$ for $i \ge 1$. Hence we require truncations of $\cC(x)$.

In order to compute its coefficients, our first approach is to use the mentioned software \texttt{plantri} to iterate over all cubic planar maps and then de-duplicate using \texttt{shortg} from the \texttt{nauty} software suite~\cite{zbMATH06264238} by Brendan McKay and Adolfo Piperno. Specifically, 
\[
	\text{\texttt{plantri -c1m3 -d -s \{n\}d | shortg -u }  }
\]
outputs the number $c_n$ of unlabelled connected cubic planar graphs with $n$ vertices. This process may be parallelised using \texttt{plantri}'s res/mod functionality by storing the resulting chunks of graphs in a first stage and deduplicating graphs from different chunks in a second stage.

Regardless of whether we parallelise or not, the bottleneck is the de-duplication stage. At the very least we need to store and search all the non-isomorphic graphs found during program execution. This poses a major challenge already when $n>30$, where the number of non-isomorphic graphs exceeds $10^9$.

In order to tackle this challenge we plug the truncation of the Walsh index series for $3$-connected networks into our system of equations, yielding a truncation for the Walsh index series $W_{\cN_{\mathrm{s}}}^+$ and $W_{\cN_{\mathrm{s}}}^-$ of simple networks. These series relate to the Walsh index series $W_{\cC}$ of connected cubic planar graphs through the equations
\[
	W_{\cN_{\mathrm{s}}}^+ = 2 \frac{\partial W_\cC}{\partial a_1} \qquad \text{and} \qquad 	W_{\cN_{\mathrm{s}}}^- = 2 \frac{\partial W_\cC}{\partial b_1}.
\]
Integrating $W_{\cN_{\mathrm{s}}}^+$ by $a_1$ and dividing by $2$ yields a part of $W_{\cC}$, specifically the contribution of symmetries with at least one cylindrical edge $1$-cycle. If we integrate $W_{\cN_{\mathrm{s}}}^-$ by $b_1$, divide by $2$ and evaluate at $a_1=0$, we obtain a part of $W_{\cC}$ given by the contribution of symmetries with at least one m\"obius edge $1$-cycle and no cylindrical edge $1$-cycle. These two obtained parts do not overlap.

In order to calculate the missing contributions we  patch \texttt{plantri} to suppress all maps that do not admit an automorphism without edge $1$-cycles. Maps admitting automorphisms with no  edge $1$-cycles (m\"obius or cylindrical) are very rare, hence we need to de-duplicate a much smaller collection of graphs.

In this way, we overcome the bottleneck in the de-duplication stage and obtain a truncation of $W_{\cC}$ up to graphs with $38$ vertices. Using the multiset construction yields the corresponding truncation of $W_{\cG}$. We obtain:

\begin{tcolorbox}
	\tiny
	\begin{align*}
		\cC(x) &= x^{4} + x^{6} + 3x^{8} + 9x^{10} + 32x^{12} + 133x^{14} + 681x^{16} + 3893x^{18} + 24809x^{20} + 169206x^{22} + 1214462x^{24} \\
		&\quad+ 9034509x^{26} + 69093299x^{28}  + 539991437x^{30} + 4296547755x^{32} + 34707716982x^{34} \\
		&\quad + 284066443439x^{36} + 2351827694094x^{38} + O(x^{40}), \\[3pt]
		\cG(x) &= 1 + x^4+x^6+4x^8+10x^{10}+37x^{12}+146x^{14} +733x^{16}+4099x^{18}+25819x^{20}+174688x^{22} \\
		&\quad +1248012x^{24} + 9256401x^{26}+70652394x^{28}
		+551417906x^{30}+4383002696x^{32}+35378003617x^{34} \\
		&\quad +289366461497x^{36} + 2394425156287x^{38} + O(x^{40}).
	\end{align*}
\end{tcolorbox}

\begin{corollary}
	\label{co:kappabound}
	We have
	\[
		0.07082557663 \le \kappa \le 0.07124166270.
	\]
\end{corollary}
\begin{proof}
	Set $k=38$. We use the notation
	\[
	\cC^{\langle k \rangle}(x) = \sum_{j=4}^k c_j x^j.
	\]
	Let $\rho_{\mathrm{l}} \le \rho \le \rho_{\mathrm{u}}$ denote the bounds from Theorem~\ref{te:numeric}. 
	For each $i \ge 1$ we have
	\[
		\cC(\rho^i) \ge \cC^{\langle k \rangle}(\rho_{\mathrm{l}}^i).
	\]
	On the other hand,
	\begin{align*}
		\cC(\rho^i) \le \cC^{\langle k \rangle}(\rho^i_{\mathrm{u}}) + \sum_{j \ge k+2} c_j \rho^{ij}.
	\end{align*}
	Note that for all even $n \ge 4$
	\[
		c_n \le d_n,
	\]
	because any connected unlabelled cubic planar graph has at least one non-isthmus edge. Rooting and orienting that edge yields a $\cD$-network. Hence
	\begin{align*}
		\sum_{j \ge k+2} c_j \rho^{ij} &\le \rho^{(k+2)(i-1)} \sum_{j \ge k+2} c_j \rho^{j} \\
		&\le \rho^{(k+2)(i-1)} \sum_{j \ge k+2} d_j \rho^{j} \\
		&\le \rho^{(k+2)(i-1)} (\cD(\rho) - \cD^{\langle k \rangle}(\rho)) \\
		&\le \rho^{(k+2)(i-1)}_{\mathrm{u}} (D_{\mathrm{u}} - \cD^{\langle k \rangle}(\rho_{\mathrm{l}}))
	\end{align*}
	with $D_{\mathrm{u}} \ge \cD(\rho)$ denoting the upper bound from Theorem~\ref{te:numeric}.
	Summing the lower and upper bounds over $i \ge 1$ we arrive at
	\begin{align*}
		\sum_{j=4}^k c_j \log\frac{1}{1 - \rho_{\mathrm{l}}^j} &\le  
		\sum_{i \ge 1} \cC(\rho^i) / i \\
		&\le   \frac{D_\mathrm{u} - \cD^{\langle k \rangle}(\rho_{\mathrm{l}})}{\rho_{\mathrm{u}}^{k+2}} \log\frac{1}{1 - \rho_{\mathrm{u}}^{k+2}} + \sum_{j=4}^k c_j \log\frac{1}{1 - \rho_{\mathrm{u}}^j}.
	\end{align*}
	Through this, we arrive at
	\[
		    1.011471206056769 \le \exp\left( \sum_{i \ge 1} \cC(\rho^i) / i  \right) \le 1.011817089358919.
	\]
	Since
	\[
		\kappa = \gamma \exp\left( \sum_{i \ge 1} \cC(\rho^i) / i  \right),
	\]
	we may use the bounds on $\gamma$ from Theorem~\ref{te:numeric} to arrive at the stated numerical bounds for $\kappa$.
\end{proof}

\section*{Acknowledgement}
The computational results have been achieved using the
Austrian Scientific Computing (ASC) infrastructure.
This research was funded in part by the Austrian Science Fund (FWF) 10.55776/F1002 and 10.55776/PAT6732623. For open access purposes, the authors have applied a CC BY public copyright license to any author-accepted manuscript version arising from this submission.

\bibliographystyle{abbrv}
\bibliography{ucubic}

\end{document}